\documentclass{article}

\usepackage{amsfonts}
\usepackage{graphicx}
\usepackage{epstopdf}
\usepackage{todonotes}
\usepackage{algorithmicx}
\usepackage{marginnote}
\ifpdf
  \DeclareGraphicsExtensions{.eps,.pdf,.png,.jpg}
\else
  \DeclareGraphicsExtensions{.eps}
\fi

\title{Sequential Linearization Method for Bound-Constrained Mathematical Programs with Complementarity Constraints}

\author{Christian Kirches\thanks{
Institute for Mathematical Optimization, Technische Universit\"{a}t Carolo-Wilhelmina zu Braunschweig. Universit\"atsplatz 2, 38106 Braunschweig, Germany. (\email{c.kirches@tu-bs.de})}
\and Jeffrey Larson\thanks{ Mathematics and Computer Science Division, Argonne National Laboratory, Lemont, IL 60439, U.S.A. (\email{jmlarson@anl.gov}, \email{leyffer@anl.gov})}
\and Sven Leyffer\footnotemark[2]
\and Paul Manns\footnotemark[2]\ \thanks{
Faculty of Mathematics, TU Dortmund University. Vogelpothsweg 87, 44227 Dortmund, Germany. (\email{paul.manns@tu-dortmund.de})}}

\usepackage{amsopn}

\usepackage{fullpage,amsthm}
\usepackage[breaklinks=true, colorlinks=true, linkcolor={blue!90!black}, citecolor={blue!90!black}, urlcolor={blue!90!black}]{hyperref}
\newtheorem{theorem}{Theorem}[section]
\newtheorem{definition}[theorem]{Definition}
\newtheorem{lemma}[theorem]{Lemma}
\newtheorem{remark}[theorem]{Remark}
\newtheorem{proposition}[theorem]{Proposition}
\newcommand{\email}{\url}
\usepackage{algorithm}

\usepackage{algorithmicx}
\usepackage[noend]{algpseudocode}
\usepackage{placeins}
\usepackage{cite}		
\usepackage{pgfplots}
\usepackage{tikz}
\usepackage{subfig}
\usepackage{lscape}
\usepackage{multirow}
\graphicspath{{../../codes/},}

\newtheorem{assumption}{Assumption}[section]
\theoremstyle{empty}

\definecolor{darkgreen}{rgb}{0,0.5,0} 

\usepackage[normalem]{ulem} 
 
\makeatletter
\long\def\@mn@@@marginnote[#1]#2[#3]{%
  \begingroup
    \ifmmode\mn@strut\let\@tempa\mn@vadjust\else
      \if@inlabel\leavevmode\fi
      \ifhmode\mn@strut\let\@tempa\mn@vadjust\else\let\@tempa\mn@vlap\fi
    \fi
    \@tempa{%
      \vbox to\z@{%
        \vss
        \@mn@margintest
        \if@reversemargin\if@tempswa
            \@tempswafalse
          \else
            \@tempswatrue
        \fi\fi
          \rlap{%
            \if@mn@verbose
              \PackageInfo{marginnote}{xpos seems to be \@mn@currxpos}%
            \fi
            \begingroup
              \ifx\@mn@currxpos\relax\else\ifx\@mn@currxpos\@empty\else
                  \kern-\dimexpr\@mn@currxpos\relax
              \fi\fi
              \ifx\@mn@currpage\relax
                \let\@mn@currpage\@ne
              \fi
              \if@twoside\ifodd\@mn@currpage\relax
                  \kern\oddsidemargin
                \else
                  \kern\evensidemargin
                \fi
              \else
                \kern\oddsidemargin
              \fi
              \kern 1in
            \endgroup
            \kern\marginnotetextwidth\kern\marginparsep
            \vbox to\z@{\kern\marginnotevadjust\kern #3
              \vbox to\z@{%
                \hsize\marginparwidth
                \linewidth\hsize
                \kern-\parskip
                \marginfont\raggedrightmarginnote\strut\hspace{\z@}%
                \ignorespaces#2\endgraf
                \vss}%
              \vss}%
          }%
      }%
    }%
  \endgroup
}
\makeatother
\newcommand{\changed}[1]{{\color{black} #1}}
\newcommand{\revised}[2]{{\color{black} #2}}

\newcommand{\xk}{x^{k}}
\newcommand{\dn}{d^{k,l}}

\newcommand{\trn}{\Delta^{k,l}}
\newcommand{\dps}{\displaystyle}
\newcommand{\mini}{\mathop{\rm minimize}}
\newcommand{\st}{\textnormal{subject\ to:}}
\newcommand{\R}{\mathbb{R}}
\newcommand{\tn}{\textnormal}

\newcommand{\LPCC}{\tn{LPCC}}
\usepackage{amsmath}
\usepackage{mathtools}

\usepackage{siunitx}

\begin{document}

\maketitle

\begin{abstract}
  We propose an algorithm for solving bound-constrained mathematical programs
  with complementarity constraints on the variables. Each iteration of the algorithm involves
  solving a linear program with complementarity constraints in order to obtain
  an estimate of the active set. The algorithm enforces
  descent on the objective function to promote global convergence to
  B-stationary points. We provide a convergence analysis and preliminary
  numerical results on a range of test problems.
  We also study the effect of fixing the active
  constraints in a bound-constrained quadratic program that can be solved on
  each iteration in order to obtain fast convergence. 
\end{abstract}



\section{Introduction}
\label{S:intro}

We consider the bound-constrained mathematical program with complementarity
constraints (MPCC) of the form
\begin{gather}\label{E:mpcc} 
   \begin{aligned}
     \mini_{x}\ 	& f(x)		\\
     \st\ & \ell_0 \leq x_0 \leq u_0, \\
		& 0 \leq x_1 \; \perp \; x_2 \geq 0,
   \end{aligned}
\end{gather}
where $f$ is smooth and $x \coloneqq
(x_0, x_1, x_2) \in \R^n$ is a partition of the variables into
bound-constrained variables $x_0$ (e.g., controls) and complementarity
variables $x_1$, $x_2$ (e.g., states). We let $n \coloneqq n_0 + 2 n_{1}$, with
$x_0 \in \mathbb{R}^{n_0}$ and $x_1$, $x_2 \in \mathbb{R}^{n_{1}}$, where
$n_0$, $n_1$ are nonnegative integers.
We let $\ell_{0,i}$ denote element $i$ of $\ell_{0}$ and similarly for $u_0$, $x_0$,
$x_1$, and $x_2$ as well.
We also assume that,
without loss of generality, $\ell_{0,i} < u_{0,i}$ for $i \in \{1,\ldots,n_{0}\}$,
because otherwise we could remove variable $x_{0,i}$ by fixing it to be $\ell_{0,i}=u_{0,i}=x_{0,i}$.
The feasible set of~\eqref{E:mpcc} is nonempty because $\ell_0 < u_0$.

\paragraph{Motivation}
Problems of the form~\eqref{E:mpcc} appear as subproblems in an 
augmented Lagrangian approach for solving more general MPCCs.
This approach extends existing augmented Lagrangian approaches
for standard nonlinear programs such
as~\cite{doi:10.1080/10556780701577730,Friedlander2005,curtis2015adaptive},
{\tt LANCELOT}~\cite{ConGouToi:92,ConGouToi:96}, or {\tt TANGO},~\cite{algencan,birgin2014practical} to MPCCs.
To see how~\eqref{E:mpcc} can
appear as a subproblem,
consider the general MPCC
\begin{equation}\label{eq:genMPCC}
\begin{aligned}
   \mini_x\ & f(x) \\
   \st\     & c(x) = 0, \\
            & 0 \leq g(x) \; \perp \; h(x) \geq 0,
\end{aligned}
\end{equation}
where $c: \R^n \to \R^m$, and $g, h: \R^n \to \R^p$, are smooth functions for
some $m,p \in \mathbb{N}$.
By introducing slack variables $s_g, s_h \in \R^p$,~\eqref{eq:genMPCC} can be written
as a problem with simple complementarity constraints:
\[
\begin{aligned}
   \mini_{x,s_g,s_h} \ & f(x) \\
   \st\     & c(x) = 0,\;  s_g - g(x) = 0, \;  s_h - h(x) = 0, \\
   & 0 \leq s_g \; \perp \; s_h \geq 0.
\end{aligned}
\]
By introducing Lagrange multipliers $y,z_g,z_h$ for the three sets of equality constraints,
we obtain an augmented Lagrangian of~\eqref{eq:genMPCC}
\[
\begin{aligned}
  {\cal L}_{\rho}(x,s_g,s_h,y,z_g,z_h)
     \coloneqq & f(x) - y^Tc(x) - z_g^T (s_g - g(x)) - z_h^T (s_h - h(x)) \\
        & + \frac{\rho}{2}\left( \| c(x) \|_2^2 + \| s_g - g(x) \|_2^2 + \| s_h - h(x) \|_2^2 \right),
   \end{aligned}
\]
and the augmented Lagrangian subproblem associated with~\eqref{eq:genMPCC} becomes
\begin{gather}\label{eq:alsubproblem}
   \begin{aligned}
     \mini_{x,s_g,s_h}\ 	&  {\cal L}_{\rho}(x,s_g,s_h,y,z_g,z_h) \\
     \st\ &  0 \leq s_g \; \perp \; s_h \geq 0,
   \end{aligned}
\end{gather}
which has the same structure as~\eqref{E:mpcc}.

\paragraph{Related Work} We propose to solve 
\eqref{E:mpcc} with a trust-region strategy that respects the
complementarity constraints in every iteration. Because all iterates are
feasible for~\eqref{E:mpcc}, every accepted step also provides an estimate
of the active set. While we are not aware of any publication that analyzes this
described setting and method, active set and trust-region methods have
been used for MPCCs in the past. Scholtes and St{\"o}hr~\cite{SchSt:97} analyze
a trust-region method on exact penalty functions that arise from MPCCs.
Fukushima and Tseng~\cite{FukuTsen:02} iteratively
compute approximate KKT-points of nonlinear programs (NLPs) that arise
from $\varepsilon$-active sets induced by the previous iterates,
where $\varepsilon$ is a slack parameter that is driven to zero over the iterations.
J{\'u}dice et al.~\cite{judice2007complementarity} 
and Chen and Goldfarb~\cite{chen2007active} propose active set strategies
that also respect the complementarity constraints in every iteration by 
alternatingly computing descent steps (projected Newton steps
in~\cite{chen2007active})
on the null space of the active constraints and updating entries of
the active set based on the Lagrange multipliers.

\paragraph{Notation} We use subscripts to identify components of vectors or
matrices, and superscripts to indicate iterates. Similarly,
functions that are evaluated at particular points are denoted as
$f^{k} \coloneqq f(z^{k})$, for example.

\paragraph{Structure of the Remainder} We present our
algorithm to solve~\eqref{E:mpcc} in Section~\ref{S:algm}. Then its execution is
demonstrated on an example problem in Section~\ref{S:exa}.
Section~\ref{S:proof} analyzes the convergence of the
iterates. Section~\ref{sec:bqp} presents two approaches for including
second-order information into the algorithm in order to improve
its convergence speed. In Section~\ref{sec:synthetic_benchmark_results}, we
present computational results for two sets of benchmark problems.
We show in Section~\ref{S:generalization} that our developments 
generalize to lower and upper bounds on $x_1, x_2$, and mixed
complementarity conditions between $x_1$ and $x_2$.

\section{Algorithm Statement}
\label{S:algm}
\setcounter{equation}{0}

We now introduce our SLPCC algorithm
that solves a sequence of 
linear programs with
complementarity constraints.
We will show that it
converges to B-stationary (or Bouligand stationary) points.

\begin{definition}[\cite{LuoPanRal:96}, \S3.3.1]\label{DefBstat}
A feasible point $x^*$ of~\eqref{E:mpcc} is called 
{\em B-stationary\/}
if $d = 0$ is a local minimizer of
the linear program with complementarity constraints (LPCC)
obtained by linearizing $f$, about $x^*$: 
\begin{equation}\label{eq:bstat_eqn}
\begin{aligned}
   \mini_d\ & \nabla f(x^*)^{T} d \\
   \st\ & \ell_0 \leq x_0^* + d_0 \leq u_0 \\
               & 0 \leq x_1^* + d_1 \; \perp \; x_2^* + d_2 \geq 0,
\end{aligned}
\end{equation}
where the step $d$ is partitioned as $d \coloneqq \begin{pmatrix} d_0, d_1, d_2\end{pmatrix}$.
\end{definition}

The B-stationarity condition in Definition~\ref{DefBstat} is also referred to
as linearized B-stationarity~\cite{flegel2005abadie}, although it is easy to see that
the two definitions are equivalent because of the simple structure of the
constraints in~\eqref{E:mpcc}. Definition~\ref{DefBstat} is also closely related to
stationarity in nonlinear optimization, interpreted as the absence of feasible
first-order descent directions.

\subsection{Trust-Region Subproblem of the SLPCC Algorithm}\label{sec:subproblems}

We now define a subproblem that is solved repeatedly by our algorithm.
The subproblem is motivated by Definition~\ref{DefBstat}
with an additional $\ell_\infty$-norm trust-region constraint.
Given a point $x \in \mathbb{R}^n$ and a trust-region radius $\Delta>0$, the
LPCC subproblem is 
\begin{equation}\label{eq:LPCC}
  \LPCC(x,\Delta)
  \coloneqq
  \left\{ \begin{aligned}
       \mini_d\  & \nabla f(x)^T d \\
       \st\      & \ell_0 \le x_0 + d_0 \le u_0, \\
                & 0 \leq x_1 + d_1 \; \perp \; x_2 + d_2 \geq 0, \\
                & \| d \|_{\infty} \leq \Delta.
   \end{aligned} \right. 
\end{equation}
During each SLPCC iteration $k$, we solve one or more instances of  
\eqref{eq:LPCC} for a sequence of trust region radii around the current
iterate $\xk$. 
From Definition~\ref{DefBstat} and the fact that $\Delta$ is strictly
positive it follows that $\xk$ is B-stationary if and only if
$d = 0$ solves $\LPCC(\xk,\Delta)$. That is, only when the trust-region constraint is inactive.
\begin{remark}
Our global convergence results readily generalize to other
trust-region norms,
but we find that the $\ell_{\infty}$-norm
has useful properties that allow us
to solve~\eqref{eq:LPCC} efficiently.
The problem~\eqref{eq:LPCC} decomposes
by component into $n_0$ bound-constrained
and $n_1$ two-dimensional LPCCs.
Each of these problems can be solved by
evaluating at most four feasible points.
Thus, we can solve
\eqref{eq:LPCC} in
$\mathcal{O}(n_0 + n_1)$ objective evaluations.
\end{remark}

\subsection{An SLPCC Algorithm for MPCCs}
We now state the SLPCC algorithm in Algorithm~\ref{SLPCC}; this
provides an overview of our approach first while detailed steps are
provided later.
Algorithm~\ref{SLPCC} has an outer loop (indexed by $k$) and
an inner loop (indexed by $l$). The inner loop reduces the trust-region radius
$\trn$ until a new iterate $x^{k+1}$ is found or the algorithm terminates with a certificate
that the current iterate $x^{k}$ is B-stationary. A new iterate is
acceptable if it is feasible for~\eqref{E:mpcc} and the improvement at
$f(x^{k+1})$ relative to $- \nabla f(\xk)^T \dn$ is more than a fixed fraction
$\sigma \in (0,1)$. The outer loop resets the trust-region radius to
$\Delta^{k,0} \in [\underline{\Delta},\overline{\Delta}^k]$, where $\underline{\Delta}>0$
is fixed and $\overline{\Delta}^k > \underline{\Delta}$ is nondecreasing. Then, the outer loop generates the
next iterate $x^{k+1}$ by invoking the inner loop.
If $\nabla f(\xk)^T \dn = 0$ in Line~\ref{ln:bstat_check} of 
Algorithm~\ref{SLPCC}, it follows that $d = 0$ is also a solution
of LPCC$(\xk,\trn)$, and thus $\xk$ is B-stationary
by Definition~\ref{DefBstat} and the fact that $\Delta^{k,l} > 0$
and hence the trust-region constraint is inactive.

\begin{algorithm}[t]
\caption{SLPCC Algorithm for Bound-Constrained MPCCs~\eqref{E:mpcc}}\label{SLPCC}
  \fontsize{8}{8}\selectfont
\textbf{Given:} $x^0$ feasible for~\eqref{E:mpcc};
                $\overline{\Delta}^0 > \underline{\Delta} > 0$;
                $\sigma \in (0,1)$
\begin{algorithmic}[1]
  \algrenewcommand{\algorithmiccomment}[1]{\hfill \# \texttt{ #1}}
    \For{$k = 0,1,\ldots$}
    \State Reset (inner) trust-region radius $\Delta^{k,0} \in [\underline{\Delta},\overline{\Delta}^k]$\label{ln:tr_reset}
     \For{$l = 0,1,\ldots$}
       \State Solve LPCC$(\xk,\trn)$ for a first-order step $\dn$,
       (see \S\ref{sec:efficient_lpec_solution}).
       \label{ln:dn}
       \If{$d = 0$ is a local minimizer of LPCC$(\xk,\trn)$}\label{ln:bstat_check}
         \State \textbf{terminate} \Comment{$\xk$ is B-stationary}
       \EndIf    
       \State Optionally, improve $\dn$ \revised{R2.1}{via}
       $\dn \gets \texttt{FIND\_CAUCHY\_POINT}(\xk, \dn,\trn)$. (see 
\S\ref{sec:cauchy_line_search}).
       \label{ln:cauchy}
       \State Evaluate $f(\xk +  \dn)$ and compute $\rho^{k,l} \gets \frac{f(\xk) -f(\xk + \dn)}{- \nabla f(\xk)^T \dn}$
       \If{$\rho^{k,l} \ge \sigma $}\label{ln:accept}
         
         \State Set $x^{k+1} \gets\xk + \dn$ and $\overline{\Delta}^{k+1} \gets \max\{ \overline{\Delta}^k, 2 \trn \}$\label{ln:outer_it_update}
         \State \textbf{break} \Comment{$x^{k+1}$ and $\overline{\Delta}^{k+1}$ found}
       \Else
       	\State Reduce trust-region radius $\Delta^{k,l+1} \gets \trn / 2$
       \EndIf
     \EndFor
     \State Optionally, improve $x^{k+1}$ \revised{R2.1}{via}
     $x^{k+1} \gets \texttt{SOLVE\_BQP}(x^{k+1})$ (see \S\ref{sec:bqp_steps}).
     \label{ln:bqp}
   \EndFor
\end{algorithmic}   
\end{algorithm}

The parameter $\sigma$ controls the acceptable ratio between the actual
reduction $f(\xk) - f(\xk + \dn)$ and the linear predicted
reduction $- \nabla f(\xk)^T \dn$; the predicted reduction is a positive number by
definition of the subproblem LPCC$(\xk,\trn)$. For our global
convergence analysis, the parameter $\sigma$ may be chosen arbitrarily in the open
interval $(0,1)$.

Algorithm~\ref{SLPCC} also contains two optional steps, which
make use of second-order information. First,
in Line~\ref{ln:cauchy} we can search for a local minimizer
(Cauchy point) of a quadratic model along a piecewise defined path.
Second, we can add a bound-constrained quadratic minimization
in Line~\ref{ln:bqp} that uses the fixed active set of constraints identified when 
computing $x^{k+1}$.
Global convergence of Algorithm~\ref{SLPCC} to B-stationary points does not
depend on---and is not hampered by---these optional steps. 
Section~\ref{sec:bqp} discusses these second-order steps in greater detail.
We present a simpler convergence analysis
without these optional steps in Section~\ref{S:proof}.

\subsection{Initialization of Algorithm \ref{SLPCC}}\label{sec:initialization}
While our analysis assumes we have a feasible initial point,
we note that this assumption is not critical.
If a candidate initial point $\hat{x}^0$  is not feasible, then
we can project its first $n_0$ components into the bounds, such that
$l_0 \leq x_0^0\leq u_0$. A similar operation produces feasible
components of $x^0$ for the complementarity constraints:
\[ x_{1,i}^0 \coloneqq \max\{ \hat{x}_{1,i}^0, 0 \}, \quad x_{2,i}^0 \coloneqq \max\{ \hat{x}_{2,i}^0, 0 \}, \quad
\begin{cases}
  x_{1,i}^0  \coloneqq 0 & \text{if} \; x_{1,i}^0 \leq x_{2,i}^0 \\
  x_{2,i}^0  \coloneqq 0 & \text{otherwise,}
\end{cases}
\]
where the two $\max$ operations are performed before the case statement. 
Therefore, to simplify the presentation, we assume that the
initial iterate $x^0$ in feasible for~\eqref{E:mpcc}.

\subsection{Efficient LPCC Solution}\label{sec:efficient_lpec_solution}

Next, we show that the trust-region subproblem LPCC$(x,\Delta)$ in Algorithm~\ref{SLPCC} can be solved efficiently.
We can rewrite the objective of LPCC$(x,\Delta)$ as
\[ \nabla f(x)^T d = \sum_{i=1}^{n_0} \nabla f(x)_{0,i} d_{0,i} + \sum_{i=1}^{n_{1}} \nabla f(x)_{1,i} d_{1,i}  + \sum_{i=1}^{n_{1}} \nabla f(x)_{2,i} d_{2,i}, \]
where subscript index pairs identify entries of $\nabla f(x)$ corresponding to the entries of $d$ and $x$.

With this new objective, the following proposition shows that LPCC$(x,\Delta)$ can be decomposed into $n_0$ independent one-dimensional linear programs (LPs)
and $n_1$ independent two-dimensional LPCCs, all of which can be solved 
independently, making the computational effort for solving each LPCC$(x,\Delta)$ linear in $n_0+n_1$.

\begin{proposition}\label{prp:decomposition}
  The problem \LPCC$(x,\Delta)$ in~\eqref{eq:LPCC} decomposes into $n_0 + n_{1}$
independent subproblems, namely, $n_0$ one-dimensional LPs:
\begin{gather}\label{eq:sp_d0}
\begin{aligned}
	\mini_{d_{0,i}}\ &\nabla f(x)_{0,i}d_{0,i}\\
  \st\ &\ell_{0,i} \le x_{0,i} + d_{0,i} \le u_{0,i}, |d_{0,i}| \le \Delta
\end{aligned}
\end{gather}
for $i \in \{1,\ldots,n_0\}$, and $n_1$ two-dimensional LPCCs: 
\begin{gather}\label{eq:sp_d12}
\begin{aligned}
\mini_{d_{1,i},d_{2,i}}\ &\nabla f(x)_{1,i}d_{1,i} + \nabla f(x)_{2,i}d_{2,i} \\
\st\ &0 \le x_{1,i} + d_{1,i} \perp x_{2,i} + d_{2,i}\ge 0, \|(d_{1,i}, d_{2,i})^T\|_\infty \le \Delta
\end{aligned}
\end{gather}
for $i \in \{1,\ldots,n_{1}\}$.
\end{proposition}
\begin{proof}
  The decomposition follows from the linearity of the objective
  of~\eqref{eq:LPCC} and the fact that there are no coupling constraints
  between $d_0$
  and $(d_1,d_2)$ apart from simple two-dimensional complementarity
  constraints. In particular, for all $i \in \{1,\ldots,n_{1}\}$, the variables
  $d_{1,i}$ and $d_{2,i}$ are coupled only by the constraints $0 \le x_{1,i} + d_{1,i} \perp x_{2,i} + d_{2,i} \ge 0$.
\end{proof}

\begin{figure}
\minipage{0.249\textwidth}
\centering
\begin{tikzpicture}
\begin{axis}[height=4.6cm, xmin=-.5, xmax=1.25, ymin=-.75, ymax=1.15, xtick={0}, ytick={0}, extra x ticks={0},	extra y ticks={0}, extra tick style={grid=major},
unit vector ratio={1 1}]
	\addplot[mark=none, thick] coordinates { (0,1) (0,0) (1,0) };
	\addplot[mark=none, dashed, darkgray] coordinates { (-0.25,-0.5) (0.75,-0.5) (0.75,0.5) (-0.25,0.5) (-0.25,-0.5) };
	\node[] at (axis cs:1.0,-0.25) {$x_{1,i}$};
	\node[] at (axis cs:-0.25,1.0) {$x_{2,i}$};		
	\addplot[mark=square*,darkgray] coordinates { (0.25,0.) };	
	\node[above right] at (axis cs:0.25,0) {$A^0$};	
	\addplot[mark=*,darkgray] coordinates { (0.75,0.) };
	\node[above right] at (axis cs:0.75,0) {$A^1$};
	\addplot[mark=*,darkgray] coordinates { (0.0,0.0) };
	\node[above right] at (axis cs:0,0) {$A^2$};	
	\addplot[mark=*,darkgray] coordinates { (0.0,0.5) };
	\node[above right] at (axis cs:0,0.5) {$A^3$};
\end{axis}
\end{tikzpicture}
\endminipage\hfill\minipage{0.249\textwidth}
\centering
\begin{tikzpicture}
\begin{axis}[height=4.6cm, xmin=-.5, xmax=1.25, ymin=-.75, ymax=1.15, xtick={0}, ytick={0}, extra x ticks={0},	extra y ticks={0}, extra tick style={grid=major},
unit vector ratio={1 1}]
	\addplot[mark=none, thick] coordinates { (0,1) (0,0) (1,0) };
	\addplot[mark=none, dashed, darkgray] coordinates { (-0.4,-0.3) (0.4,-0.3) (0.4,0.5) (-0.4,0.5) (-0.4,-0.3) };
	\node[] at (axis cs:1.0,-0.25) {$x_{1,i}$};
	\node[] at (axis cs:-0.25,1.0) {$x_{2,i}$};	
	\addplot[mark=square*,darkgray] coordinates { (0,0.1) };	
	\node[above left] at (axis cs:0,0.1) {$B^0$};
	\addplot[mark=*,darkgray] coordinates { (0.4,0.) };
	\node[above right] at (axis cs:0.4,0) {$B^1$};	
	\addplot[mark=*,darkgray] coordinates { (0.0,0.0) };
	\node[above right] at (axis cs:0,0) {$B^2$};	
	\addplot[mark=*,darkgray] coordinates { (0.0,0.5) };
	\node[above right] at (axis cs:0,0.5) {$B^3$};
\end{axis}
\end{tikzpicture}
\endminipage\hfill
\minipage{0.249\textwidth}
\centering
\begin{tikzpicture}
\begin{axis}[height=4.6cm, xmin=-.5, xmax=1.25, ymin=-.75, ymax=1.15, xtick={0}, ytick={0}, extra x ticks={0},	extra y ticks={0}, extra tick style={grid=major},
unit vector ratio={1 1}]
	\addplot[mark=none, thick] coordinates { (0,1) (0,0) (1,0) };
	\addplot[mark=none, dashed, darkgray] coordinates { (0.1,-0.3) (0.7,-0.3) (0.7,0.3) (0.1,0.3) (0.1,-0.3) };
	\node[] at (axis cs:1.0,-0.25) {$x_{1,i}$};
	\node[] at (axis cs:-0.25,1.0) {$x_{2,i}$};
	\addplot[mark=square*,darkgray] coordinates { (.4,0) };	
	\node[above right] at (axis cs:.4,0) {$C^0$};	
	\addplot[mark=*,darkgray] coordinates { (0.1,0) };
	\node[above right] at (axis cs:0.1,0) {$C^1$};	
	\addplot[mark=*,darkgray] coordinates { (0.7,0) };
	\node[above right] at (axis cs:0.7,0) {$C^2$};	
\end{axis}
\end{tikzpicture}
\endminipage\hfill
\minipage{0.249\textwidth}
\centering
\begin{tikzpicture}
\begin{axis}[height=4.6cm, xmin=-.5, xmax=1.25, ymin=-.75, ymax=1.15, xtick={0}, ytick={0}, extra x ticks={0},	extra y ticks={0}, extra tick style={grid=major},
unit vector ratio={1 1}]
	\addplot[mark=none, thick] coordinates { (0,1) (0,0) (1,0) };
	\addplot[mark=none, dashed, darkgray] coordinates { (-0.25,0.1) (0.25,0.1) (0.25,0.6) (-0.25,0.6) (-0.25,0.1) };
	\node[] at (axis cs:1.0,-0.25) {$x_{1,i}$};
	\node[] at (axis cs:-0.25,1.0) {$x_{2,i}$};
	\addplot[mark=square*,darkgray] coordinates { (0,.35) };	
	\node[above right] at (axis cs:0,.35) {$D^0$};		
	\addplot[mark=*,darkgray] coordinates { (0.0,0.1) };
	\node[above right] at (axis cs:0,0.1) {$D^1$};		
	\addplot[mark=*,darkgray] coordinates { (0.0,0.6) };
	\node[above right] at (axis cs:0,0.6) {$D^2$};		
\end{axis}
\end{tikzpicture}
\endminipage
\caption{Possible cases A, B, C, and D for the intersection of the
trust region
(dashed) with the feasible set of the complementarity constraint in the
$x_{1,i}$-$x_{2,i}$ plane (from left to right). The current point (center of
the trust region) is indicated by a square.}\label{fig:lpec_solve_cases}
\end{figure}
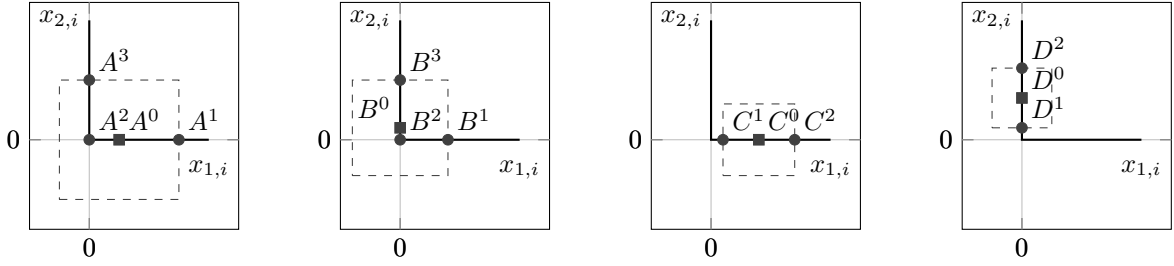

Proposition~\ref{prp:decomposition} ensures that LPCC$(x,\Delta)$ 
can be solved efficiently.
To represent the solution of the subproblems explicitly, we note that
for $x_{1,i}$, $x_{2,i}$ such that $0 \le x_{1,i} \perp x_{2,i} \ge 0$ and
$\Delta > 0$, we can distinguish four mutually exclusive and exhaustive cases,
\begin{gather}\label{eq:cases}
\begin{aligned}
\text{A: } & \Delta \ge x_{1,i} \ge 0 \text{ and } x_{2,i} = 0, \\
\text{B: } & x_{1,i} = 0 \text{ and } \Delta \ge x_{2,i} > 0,\\
\text{C: } & x_{1,i} > \Delta \text{ and } x_{2,i} = 0,\\
\text{D: } & x_{1,i} = 0 \text{ and } x_{2,i} > \Delta,
\end{aligned}
\end{gather}
where \changed{cases C and D and cases A and B} are symmetric.
\revised{R2.2}{Note, however, that the biactive components are only
included in case A.}
Figure~\ref{fig:lpec_solve_cases} shows a sketch of the constraint and
the trust region in the $x_{1,i}$-$x_{2,i}$-plane
for each of these cases. We represent a solution of LPCC$(x,\Delta)$ 
in the following proposition.

\begin{proposition}\label{prp:lpec_solve}
  Let $x \in \R^n$ be feasible for~\eqref{E:mpcc} and let $\Delta > 0$. Then the
  problem $\LPCC(x,\Delta)$ in~\eqref{eq:LPCC} is solved by any vector $d =
  \begin{pmatrix} d_0^T & d_1^T & d_2^T \end{pmatrix}^T$ that satisfies
\begin{equation*}
	d_{0,i} \coloneqq \begin{cases} 
	\min\{u_{0,i} - x_{0,i}, \Delta\}     & \text{ if }\nabla f(x)_{0,i} < 0, \\
	\max\{\ell_{0,i} - x_{0,i}, -\Delta\}	& \text{ if } \nabla f(x)_{0,i} > 0,\\
	0 &\text{ else}
	\end{cases}
\end{equation*}
for $i \in \{1,\ldots,n_0\}$ and
\begin{align}\label{eq:sp_d12_casedistinction}
	\begin{pmatrix}
	d_{1,i} \\ d_{2,i}
	\end{pmatrix} &\in \arg \min
	\left\{\begin{array}{l}
	\nabla f(x)_{1,i}d_{1,i} + \nabla f(x)_{2,i}d_{2,i} : 
	\begin{pmatrix}
	d_{1,i} \\ d_{2,i}
	\end{pmatrix} \in
  D_i
	\end{array}\right\}
\end{align}
for all $i \in \{1,\ldots,n_{1}\}$,
where $D_i$ is a finite set of points:
\begin{equation}
  D_i \coloneqq
  \begin{cases}
    \{ (0,0),  (\Delta,0), (-x_{1,i},0), (-x_{1,i},\Delta)\} & \text{if \changed{$(x_{1,i},x_{2,i})$ satisfy} case \tn{A}},\\
    \{ (0,0), (\Delta,-x_{2,i}), (0,-x_{2,i}), (0,\Delta) \} & \text{if \changed{$(x_{1,i},x_{2,i})$ satisfy} case \tn{B}},\\
    \{ (0,0), (-\Delta,0), (\Delta, 0) \} & \text{if \changed{$(x_{1,i},x_{2,i})$ satisfy} case \tn{C}},\\
    \{ (0,0), (0,-\Delta), (0,\Delta) \} & \text{if \changed{$(x_{1,i},x_{2,i})$ satisfy} case \tn{D}}\\	 
  \end{cases}
\end{equation}
for the cases in~\eqref{eq:cases}.
\end{proposition}

\begin{proof}
The result for $d_0$ is straightforward and therefore omitted.

Now consider~\eqref{eq:sp_d12_casedistinction}. The objective of~\eqref{eq:LPCC} is separable for
each $i \in \{1,\ldots,n_0\}$ and for each pair of complementarity
variables $d_{1,i}, d_{2,i}$ for $i \in \{1,\ldots,n_1\}$. The feasible set for each index $i$
consists of the \revised{R2.3}{union or one of the following two line segments,
$\big\{(d_{1,i},d_{2,i}) : x_{1,i} + d_{1,i} \in 
[\max\{0,x_{i,1}-\Delta\}, x_{i,1} + \Delta], x_{2,i} + d_{2,i} = 0\big\}$
and $\big\{(d_{1,i},d_{2,i}) : x_{1,i} + d_{1,i} = 0, x_{2,i} + d_{2,i} \in [\max\{0,x_{i,2}-\Delta\}, x_{i,2} + \Delta]\big\}$.}
Because the objective of~\eqref{eq:LPCC} is linear, an optimum
occurs at $(d_{1,i},d_{2,i})=0$ if $(\nabla f_{1,i},\nabla f_{2,i})=0$, or at the
boundary of one of the two line segments, or at the origin of the feasible
region of~\eqref{eq:LPCC}. Therefore,
to find the (global) minimizer of the partial minimization,
problem~\eqref{eq:sp_d12},
we need to evaluate the linear
objective only at three points (cases C and D if the origin
$(x_{1,i},x_{2,i}) = (0,0)$ is outside the trust region) or at four points (cases
A and B if the point $(x_{1,i},x_{2,i}) = (0,0)$ is inside in the
trust region).

The possible minimizers in the brackets correspond to the points in the cartoons
in Figure~\ref{fig:lpec_solve_cases} in the order of the numbering
of the cartoon.
\end{proof}

\begin{remark}
  Our construction for solving the LPCC$(x,\Delta)$ can be
  interpreted as projecting the steepest-descent
  direction $-\nabla f$ onto the feasible set intersected by the trust region.
  This point of view makes the algorithm a
  projected-gradient algorithm.
\end{remark}

We also allow the choice $(d_{1,i},d_{2,i}) \coloneqq (0,0)$ in the case distinction
in~\eqref{eq:sp_d12_casedistinction}  to
detect B-stationarity.
Algorithmically this means that if the $\arg \min$
in~\eqref{eq:sp_d12_casedistinction} contains
$(0,0)$, then we choose $(0,0)$.

Proposition~\ref{prp:lpec_solve} shows that LPCC$(x,\Delta)$ can be solved by
evaluating $\sum_{i=1}^{n_1} \left| D_i  \right| \leq 4 n_1$ points,
providing an efficient solution approach. We note that other approaches such
as pivoting schemes~\cite{FangLeyfMuns:12} and mixed-integer
approaches~\cite{Hu2008} also provide means to solve LPCC$(x,\Delta)$; these
approaches
are typically less efficient, however,  because they do not exploit the structure of
LPCC$(x,\Delta)$.

\section{Illustrative Example}\label{S:exa}

We now demonstrate the behavior of our SLPCC algorithm on an
illustrative example. This example highlights how the trust region radius is
reset and shows how SLPCC overcomes a degenerate situation, where a second-order
approach may stop at a suboptimal point.
The contraction of the trust region at points that may be
suboptimal is also a possible outcome for the trust region strategy,
see~\cite[Proposition 4.5]{SchSt:97}.

Consider the two-dimensional MPCC
\begin{gather}\label{eq:ex1}
\mini_x\ x_1^3 - (x_2 - 0.5x_2^2)\quad \st\ 0 \le x_1 \perp x_2 \ge 0,
\end{gather}
which has the
global minimizer at $(x_1,x_2) = (0, 1)$
with an objective value of $-0.5$.

We first consider solving~\eqref{eq:ex1} with a sequential quadratic
programming (SQP) approach (see, for example,~\cite[Ch.~18]{nocedal2006no}) that
uses exact Hessian information and with initial iterate values
$x^{0}_1 \in (0,2)$ and $x^{0}_2 \coloneqq 0$.
For simplicity, we assume that ${\Delta^k > x^{0}_1}$
is sufficiently large for all $k$
so that the trust region does not restrict
the steps.
In such an SQP approach, the first quadratic minimization stops at the
local  minimizer $x^{1} = 0.5x^{0}$ of the quadratic model
$m(x_1) = (x_1^0)^3 + 3(x_1^0)^2(x_1 - x_1^0) + 3x_1^0 (x_1 - x_1^0)^2$ arising
from the second-order
Taylor approximation of the objective of~\eqref{eq:ex1}. Because
$x_1^0 > 0$, the choice $x_2^0 = 0$ is necessary for feasibility;
and we will show below that an SQP approach generates a sequence of iterates converging to
$(0,0)$ with $x_1^k > 0$ and $x_2^k = 0$.

Inductively, we obtain $\xk = 0.5x^{k-1}$ for the $k$th
iteration. Setting $0=\frac{\partial m}{\partial x_1}$, we have 
that
\[ \frac{\partial m}{\partial x_1} = 0\; \Rightarrow \;
0 = 3(x_1^k)^2 + 6 x_1^k(x_1 - x_1^k) \; \Rightarrow \;
x_1^{k+1} = \frac{1}{2}x_1^k,
\]
which imply that $x_1^k \searrow 0$ and $x_2^k=0$ for all $k$.
Thus, we obtain $\xk  \to (0,0)$.
The limit point, $(0,0)$, is a so-called M-stationary point~\cite{outrata1999optimality}, which is
not a local minimizer of~\eqref{eq:ex1} or even a B-stationary point.
Therefore,
the tangent cone~\cite[Definition 12.2]{nocedal2006no} at $(0,0)$ contains directions
that are not contained in the tangent cones of the iterates
$\xk$, and thus an SQP approach applied to~\eqref{eq:ex1} starting from $x^0_1 \in
(0,2)$ and $x^0_2 \coloneqq 0$ will converge to $(0,0)$ even though there is a
descent direction at $(0,0)$. This example is a
reason for developing our approach using sequential linear models rather than a sequential
quadratic approach,
which would suffer from the same behavior as an SQP approach.

In contrast, the proposed SLPCC algorithm applied to~\eqref{eq:ex1} will
not converge to the origin when started from a point on the positive $x_1$
axis. This is because it generates iterates
that lie at the boundary of the trust region or at
the origin. If we assume that $\xk \to (0,0)$, then once the algorithm is sufficiently
close ($x_1^k < \Delta$) to the origin, the LPCC problem will detect the
descent direction $(-x_1^k,1)$ from the origin, allowing it to ``turn the corner'' and converge
to $x^*=(0,1)$. The trust region radius reset,
$\Delta^{k,0} \in [\underline{\Delta},\overline{\Delta}^k]$,
in Algorithm~\ref{SLPCC} ensures that the trust region radius does not go to zero too
quickly.

Figure~\ref{fig:guiding_example} illustrates the iterates of both
Algorithm~\ref{SLPCC} and an SQP approach
when starting in the point $(2,0)$ and using the initial trust region radius
$\Delta^{k,0}= 0.5$ for the inner loop.

\begin{figure}[th!]
\minipage{0.48\textwidth}
\centering
\begin{tikzpicture}
\begin{axis}[height=4.5cm, xmin=-.5, xmax=2.5, ymin=-.5, ymax=1.5, xtick={0}, ytick={0}, extra x ticks={0,1,2},	extra y ticks={0,1}, extra tick style={grid=major},
unit vector ratio={1 1}]
	\addplot[mark=none, thick] coordinates { (0,1.25) (0,0) (2.25,0) };
	\node[] at (axis cs:2.1,-0.25) {$x^{0}$};	
	\addplot[mark=*,darkgray] coordinates { (2.0,0.) };
	\node[] at (axis cs:1.6,-0.25) {$x^{1}$};	
	\addplot[mark=*,darkgray] coordinates { (1.5,0.) };
	\node[] at (axis cs:0.6,-0.25) {$x^{2}$};
	\addplot[mark=*,darkgray] coordinates { (0.5,0.) };
	\node[] at (axis cs:-0.25,1.0) {$x^{3}$};		
	\addplot[mark=*,darkgray] coordinates { (0.,1.) };
\end{axis}
\end{tikzpicture}
\endminipage\hfill\minipage{0.48\textwidth}
\centering
\begin{tikzpicture}
\begin{axis}[height=4.5cm, xmin=-.5, xmax=2.5, ymin=-.5, ymax=1.5, xtick={0}, ytick={0}, extra x ticks={0,1,2},	extra y ticks={0,1}, extra tick style={grid=major},
unit vector ratio={1 1}]
	\addplot[mark=none, thick] coordinates { (0,1.25) (0,0) (2.25,0) };
	\node[] at (axis cs:2.1,-0.25) {$x^{0}$};	
	\addplot[mark=*,darkgray] coordinates { (2.0,0.) };
	\addplot[mark=*,darkgray] coordinates { (1.,0.) };
	\addplot[mark=*,darkgray] coordinates { (.5,0.) };
	\addplot[mark=*,darkgray] coordinates { (0.25,0.) };
	\node[] at (axis cs:0.26,-0.25) {$x^{4}$};		
	\addplot[mark=*,darkgray] coordinates { (0.125,0.0) };
	\addplot[mark=*,darkgray] coordinates { (0.06125,0.0) };
	\addplot[mark=*,darkgray] coordinates { (0.030625,0.0) };
	\addplot[mark=*,darkgray] coordinates { (0.0153125,0.0) };
	\addplot[mark=*,darkgray] coordinates { (0.00765625,0.0) };
\end{axis}
\end{tikzpicture}
\endminipage
\caption{Iterates of Algorithm~\ref{SLPCC}
  (left) and an SQP approach
  with exact Hessian approximation (right) on~\eqref{eq:ex1}.}\label{fig:guiding_example}
\end{figure}
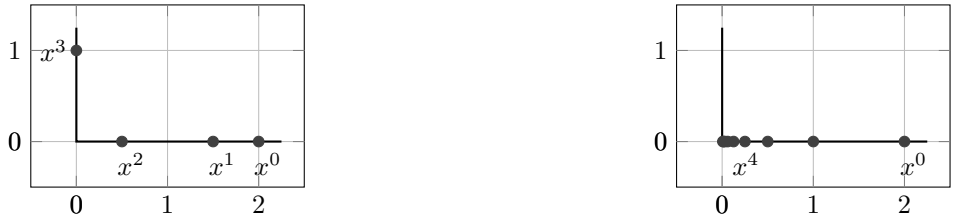

\section{SLPCC Convergence Proof}
\label{S:proof}
\setcounter{equation}{0}

We now establish that the SLPCC algorithm converges to B-stationary points.

\subsection{Preliminaries}\label{sec:proof_prelims}

We first derive a test for B-stationarity that is equivalent to
Definition~\ref{DefBstat} while also being more constructive and simplifying our analysis.
To this end, we consider the constraints that are active at a point $x$
and split them into three disjoint sets that correspond to
the respective entries of $x_0$, $x_1$, and $x_2$:
\begin{equation}\label{eq:Asets}
\begin{array}{ll}
   {\cal A}_0(x) &\coloneqq \left\{ i : x_{0,i} = \ell_{0,i} \text{ or } x_{0,i} =  u_{0,i} \right\}, \\
   {\cal A}_1(x) &\coloneqq \left\{ i : x_{1,i} = 0 \right\}, \text{ and} \\
   {\cal A}_2(x) &\coloneqq \left\{ i : x_{2,i} = 0 \right\}.
\end{array}
\end{equation}
The active set ${\cal A}_0(x)$ for $x_0$ may be decomposed into the indices
where the lower bound is active and the indices where the upper bound is active, 
namely,
\begin{equation*}
	{\cal A}_{0\ell}(x) \coloneqq \{ i : x_{0,i} = \ell_{0,i} \}, \text{ and }\;
	{\cal A}_{0u}(x) \coloneqq \{ i : x_{0,i} = u_{0,i} \}.
\end{equation*} 
These sets satisfy ${\cal A}_{0\ell}(x) \cap {\cal A}_{0u}(x) = \emptyset$
because $\ell < u$, and ${\cal A}_{0\ell}(x) \cup {\cal A}_{0u}(x) = {\cal
A}_{0}(x)$. The set of degenerate (i.e.,~nonstrict) complementarity
constraints at $x$ is denoted
\[ {\cal D}(x) \coloneqq {\cal A}_1(x) \cap {\cal A}_2(x). \]
In addition, we define the set of binding complementarity constraints, namely,
those where strict complementarity holds and either $x_{1,j}=0$ or 
$x_{2,j} = 0$ (but not both), as
\begin{equation*}
   {\cal A}_{1+}(x)  \coloneqq \left\{ j \in {\cal A}_1(x) : x_{2,j} > 0 \right\}, \text{ and } \;
   {\cal A}_{2+}(x)  \coloneqq \left\{ j \in {\cal A}_2(x) : x_{1,j} > 0 \right\} .
\end{equation*}
With this notation, we can now characterize feasible directions
along which the objective of the LPCC~\eqref{eq:bstat_eqn} can be reduced if a
given $x \in \R^n$ is feasible for~\eqref{E:mpcc} but not B-stationary.
We formalize this result in Proposition~\ref{P:non_opt_descent}.

\begin{proposition}\label{P:non_opt_descent}
Let $x$ be feasible for the MPCC~\eqref{E:mpcc} but not be B-stationary.
Then there exist $\varepsilon > 0$, a direction $s \in \mathbb{R}^n$
with $\|s\|_\infty = 1$, and a partition $({\cal D}_1, {\cal D}_2)$ of ${\cal D}(x)$ such that
\begin{subequations}\label{E:NonOpte}
\begin{align}
\nabla f(x)^T s &\le - \varepsilon & & \\
                     s_{0,i} &\ge 0 & \forall i &\in {\cal A}_{0\ell}(x),\\
                     s_{0,i} &\le 0 & \forall i &\in {\cal A}_{0u}(x),\\
                     s_{1,i} &= 0 & \forall i &\in {\cal A}_{1+}(x), \\
                     s_{2,i} &= 0 & \forall i &\in {\cal A}_{2+}(x), \\
                     s_{1,i} &= 0 \text{ and } s_{2,i} \ge 0 & \forall i &\in {\cal D}_{1}, \\
                     s_{1,i} &\ge 0 \text{ and } s_{2,i} = 0 & \forall i &\in {\cal D}_{2}.
\end{align}
\end{subequations}   
\end{proposition}
\begin{proof}
This result follows from equivalent reformulations of the LPCC 
\eqref{eq:bstat_eqn}, which is used to define B-stationarity, by considering
branch problems per component, see also~\cite[Corollary 3.3.1]{LuoPanRal:96}
and~\cite[Theorem 3.3.4]{LuoPanRal:96}.
\end{proof}

The conditions on the entries in $\mathcal{A}_{0\ell}(x)$, ${\cal A}_{0u}(x)$,
${\cal A}_{1+}(x)$, ${\cal A}_{2+}(x)$, ${\cal D}_{1}$, and ${\cal D}_{2}$ in
\eqref{E:NonOpte} ensure the existence of a direction $s$ that points into the
feasible set while $\nabla f(x)^T s \le -\varepsilon$
gives a means to identify that $x$ is not B-stationary.
We will exploit this existence of feasible descent directions from nonoptimal points
in our convergence analysis.
Note that an expensive enumeration of partitions
is not necessary in practice because the results from Section~\ref{sec:efficient_lpec_solution} show that we need to solve only 
$2 | {\cal D}(x^*)|$ LPs.

\subsection{Main Convergence Result}

To derive our convergence results, we make the following assumption
on the MPCC problem.

\begin{assumption}\label{A:C2}
The function $f$ is \revised{R2.5}{continuously differentiable, and $\nabla f$
is locally Lipschitz continuous}.
\end{assumption}

Our main convergence result shows that the SLPCC algorithm generates a
subsequence that converges to a B-stationary point.

\begin{theorem}\label{T:main}
Let Assumption~\ref{A:C2} hold,
let $\sigma \in (0,1)$ be fixed, and let $x^{0}$ be feasible for
\eqref{E:mpcc}. Then one of the following mutually exclusive outcomes must occur.
\begin{description}
  \item[O1] Algorithm~\ref{SLPCC} terminates at a B-stationary point; that is, $d=0$ solves
    the subproblem \LPCC$(\xk,\Delta^{k,l})$ for some $k$
    and $l$.
  \item[O2] Algorithm~\ref{SLPCC} generates an infinite sequence of iterates $\{ x^k \}$ with decreasing objective values.
If this sequence has an accumulation point,
then any such accumulation point is feasible
and B-stationary.
\end{description}
\end{theorem}

This theorem is proven after Lemma~\ref{L:Fstep}. The outcomes {\bf O1} and {\bf O2} with a bounded sequence of iterates
correspond to normal asymptotics of the algorithm. If we make additional
assumptions on $f(x)$, such as  coercivity (that is, $f(x) \to \infty$ if
$\|x\|\to\infty$), then we can exclude the case that a subsequence of iterates
becomes unbounded. We prefer not to make such assumptions, however, and instead
detect unboundedness in our implementation by checking whether $f(\xk) \le -U$
for some large $U > 0$.

{\em Outline of Convergence Proof.}
First, in Lemma \ref{L:Taylor}, we revisit the relationship between the actual and
linear predicted reduction of the objective following from Taylor's theorem;
this result is used in the subsequent proofs.
Second, we show in Lemma \ref{L:finite} that the inner loop of Algorithm~\ref{SLPCC}
always terminates after finitely many iterations.
Consequently, Algorithm~\ref{SLPCC}
produces a sequence of feasible iterates that have decreasing objective
values because
\begin{equation*}
  f(\xk) - f(x^{k+1}) \ge - \sigma \nabla f(x^k)^T \dn > 0.
\end{equation*}
Because the feasible set of~\eqref{E:mpcc} is closed, all accumulation
points of the sequence $\xk$ are feasible. Therefore, it remains to prove that the
accumulations points are also B-stationary.
We show in Lemma \ref{L:Fstep} that in a neighborhood of a
feasible but not B-stationary point, the LPCC will eventually generate a step
that is accepted and implies a reduction in the objective that
is bounded from below by a multiple of the
trust-region radius.
Finally, we \revised{R1.1}{synthesize} these steps in Theorem~\ref{T:main}.

\subsection{SLPCC Convergence Proof}

As indicated above, we start with a well-known lemma
about the reduction implied by the LPCC step.
\begin{lemma}\label{L:Taylor}
  Let $f$ satisfy Assumption~\ref{A:C2}.
  Let $x \in \mathbb{R}^n$ and $r > 0$ be given, and let
  $d \in \mathbb{R}^n$ be such that
  $\|d\|_\infty \le r$. Define
  \[
  M \coloneqq \sup \Bigg\{ \frac{\|\nabla f(y) - \nabla f(z)\|_{\changed{1}}}{\|y - z\|_\infty}\,:\,
  \begin{aligned}  
  &y,z \in \R^n \text{ with } y \neq z\text{ and } \\
                           &\|x - y\|_\infty \le r\text{ and }
                           \|x - z\|_\infty \le \changed{2}r \end{aligned} \Bigg\}.
 \]
Then \changed{$M<\infty$,} and the linearly predicted reduction and actual reduction for $d$ satisfies
\begin{gather}\label{E:Taylor}
   f(\changed{y}) - f(\changed{y} + d) \ge  -\nabla f(\changed{y})^T d - \frac{1}{2}\|d\|_{\infty}^2 M
\end{gather}
\changed{for all $y$ with $\|x - y\|_\infty \le r$.}
\end{lemma}
\begin{proof}
This result follows with standard arguments \changed{and} Taylor's theorem.
\end{proof}

Next, we employ Lemma \ref{E:Taylor} to prove
that the inner loop of Algorithm~\ref{SLPCC} always terminates
after finitely many iterations.

\begin{lemma}\label{L:finite}
Let Assumption~\ref{A:C2} hold. Then the
inner loop of Algorithm~\ref{SLPCC} terminates in a finite number of steps.
\end{lemma}
\begin{proof}
If $\xk$ is B-stationary, then $d=0$
solves LPCC$(\xk,\Delta)$ for any $\Delta > 0$,
and thus the subproblem solver chooses $\dn$ such
that $\nabla f(\xk)^T \dn = 0$ in the first iteration of the inner loop,
which yields a termination of the algorithm.

If $\xk$ is not B-stationary, then Proposition~\ref{P:non_opt_descent} guarantees that there exist
$s \in \mathbb{R}^n$ with $\|s\|_\infty = 1$
and $\varepsilon > 0$ such that
\[
-\nabla f(\xk)^T\dn \ge -\nabla f(\xk)^T 
(\trn s) \ge \trn \varepsilon,
\]
where the first inequality holds because 
$\dn$ is the solution
of a minimization problem for which
$\trn s$ is a feasible point.

The inner loop terminates when the sufficient reduction condition $f(\xk) -
f(\xk + \dn) \ge -\sigma f(\xk)^T \dn > 0$ is satisfied. For all $l \in
\mathbb{N}$, Lemma~\ref{L:Taylor} gives
\[ f(\xk) - f(\xk + \dn) \ge -\sigma\nabla f(\xk)^T \dn - (1 - \sigma)\nabla f(\xk)^T \dn - \frac{1}{2}\|\dn\|^2_\infty M \]
with $M$ as defined in Lemma \ref{E:Taylor} for the choice
$r = \Delta^{k,0}$. Because
\[
\|\dn\|_\infty \le \trn = \Delta^{k,0} / 2^{l},
\]
it follows that 
\[ f(\xk) - f(\xk + \dn) \ge - \sigma \nabla f(\xk)^T \dn \]
for $\trn < 2 (1 - \sigma) \varepsilon / M$
with the estimates from Lemma~\ref{L:Taylor}, and the inner iteration terminates finitely as soon as
$l$ is sufficiently large.
\end{proof}

Lemma~\ref{L:finite} implies that if Algorithm~\ref{SLPCC}
does not terminate with Outcome {\bf O1}, then it generates an
infinite sequence of iterates. If the iterates remain bounded,
then the sequence admits at least one accumulation point.

Next, we show that \revised{R2.7}{LPCC steps yield a reduction
of the objective that is bounded below by a fraction of the trust-region
radius near any feasible point that is not B-stationary.}

\begin{lemma}\label{L:Fstep}
Let $\sigma \in (0,1)$ and $f$ satisfy Assumption~\ref{A:C2}.
Let $x^{\infty}$ be feasible for~\eqref{E:mpcc} but not B-stationary. Then
there exist an $\varepsilon > 0$, a direction $s \in \R^n$
with \revised{R2.6}{$\|s\|_\infty=1$} and \revised{}{$\nabla f(x^k)^Ts\leq -\varepsilon$}, a relative neighborhood ${\cal N}^{\infty}$ of $x^{\infty}$, and constants $\eta > 0$ and $\kappa > 0$
such that for any sequence $\{\xk\} \subset {\cal N}^{\infty}$ with
$\xk \to x^{\infty}$,
the \LPCC$(\xk,\Delta)$ produces a descent direction $d^k$
for all $k$ sufficiently large that produces at least
a fraction of decrease as realized by $\Delta s$.
That is, there exists a sequence $\delta_k \to 0$ as
$\xk \to x^{\infty}$ such that
there is a feasible
$d^{k}$ for \LPCC$(\xk,\Delta)$ with
\begin{gather}\label{eq:suff_decrease_seq}
	f(\xk) - f(\xk + d^k) \ge
	-\sigma \nabla f(\xk)^T(\Delta s)
	\ge \sigma \varepsilon \Delta\revised{R2.7}{,}
\end{gather}
\changed{and
\begin{gather}\label{eq:step_accept_seq}
	f(\xk) - f(\xk + d^k) \ge
	-\sigma \nabla f(\xk)^T d^k
\end{gather}}
for all trust region radii $\Delta$ satisfying

\begin{gather}\label{eq:rho_bounds}
\delta_k \eta \le \Delta \le \kappa.
\end{gather}
Furthermore, for all $k$ sufficiently large,
the interval $[\delta_k\eta,\kappa]$ in
\eqref{eq:rho_bounds} is nonempty.
\end{lemma}
\begin{proof}
Because $x^{\infty}$ is not B-stationary,
Proposition~\ref{P:non_opt_descent} ensures the existence of
$\varepsilon_0 > 0$ and a direction $s$ with $\|s\|_\infty = 1$
such that $\nabla f(x^\infty)^T s < -\varepsilon_0$. Moreover, 
the characterization of $s$ given in~\eqref{E:NonOpte}
implies that one can choose $t_0 > 0$ sufficiently small
such that $x^\infty + t s$ is feasible for~\eqref{E:mpcc}
for all $0 \le t \le t_0$.
Next, we choose a relative neighborhood ${\cal N}^\infty$
of $x^\infty$. (By relative neighborhood we mean the intersection of the
neighborhood of $x^{\infty}$ in $\mathbb{R}^n$ with the feasible set of 
\eqref{E:mpcc}.)
Because Assumption~\ref{A:C2} implies that $\nabla f$ is a continuous 
function, we may choose ${\cal N}^\infty$ sufficiently small so that
it satisfies the following two properties.
\begin{enumerate}
\item There exists $0 < \varepsilon \le \varepsilon_0$ such that
$\nabla f(x)^T s \le -\varepsilon$ for all $x \in {\cal N}^\infty$.
\item $\|x - x^\infty\|_\infty \le \frac{1}{2} t_0$ for all $x \in {\cal N}^\infty$.
\end{enumerate}

Let $\Delta \le \frac{1}{2} t_0$. The next step of the
proof is to construct a feasible point $d$ for LPCC($x,\Delta$) for a given
$x \in {\cal N}^\infty \cap \{\xi \in \mathbb{R}^n \,:\, \|\xi - x^\infty\|_\infty
\le \Delta\}$. Afterwards, we will use the characterization
to obtain the estimate~\eqref{eq:suff_decrease_seq} under the
condition~\eqref{eq:rho_bounds} on $\Delta$.

To this end, we consider the active sets ${\cal A}_0(x^\infty)$, ${\cal
A}_{1+}(x^\infty)$, ${\cal A}_{2+}(x^\infty)$, and ${\cal D}(x^\infty)$ from 
Section~\ref{sec:proof_prelims} and the decomposition $({\cal D}_1^\infty,{\cal
D}_2^\infty)$ of ${\cal D}(x^\infty)$ that exists by virtue of
Proposition~\ref{P:non_opt_descent}.
We highlight that the active sets may differ between $x^\infty$ and $x$,
but this analysis requires the sets at $x^\infty$.

We construct $d$ by combining a step from
$x \in {\cal N}^\infty \cap
\{\xi \in \mathbb{R}^n \,:\, \|\xi - x^\infty\|_\infty \le \Delta\}$
towards the activities defined at $x^{\infty}$ and a step of length
$\Delta$. We start by defining the projection onto the activities
\[ \widehat{x}_{1,i} \coloneqq \left\{ \begin{array}{ll}
	x_{1,i} \quad & \text{if} \;
	i \in {\cal A}_{2+}(x^\infty) \cup {\cal D}_2^\infty, \\
    0 &  \text{if} \; i \in {\cal D}_1^\infty \cup {\cal A}_{1+}(x^\infty),
    \end{array}\right.\]
and similarly for $\widehat{x}_{2}$, and we set
$\widehat{x}_0 \coloneqq x_{0}$. Then it follows that
\[ \widehat{p} \coloneqq \widehat{x} - x \]
is the orthogonal projection of $x$ onto the degenerate indices.
We construct the step
\begin{equation}\label{eq:d_def}
  d \coloneqq \widehat{p} + \Delta s. 
\end{equation}
The step $d$ is feasible for LPCC($x,\Delta$) by construction
of $\hat{p}$, $\Delta \le \frac{1}{2} t_0$,
and the second property of ${\cal N}^\infty$.
The choices of $\hat{x}_{1,i}$ and
$\hat{x}_{2,i}$ are made so that a pivoting of the
inactive coordinates can happen by adding $\Delta s$
to $\hat{x} = x + \hat{p}$ if this is necessary
to obtain the descent following from the direction $s$.
The feasibility follows from the fact that the projection
gives a step of at most $\Delta$ in an inactive coordinate
of $x$. If this step is nonzero, the inactive coordinate is pivoted at the kink
and a step of length of at most $\Delta$ is added to the new inactive coordinate.
If this step was zero, the inactive coordinate does not change
and a step of length of at most $\Delta$ is added to the
inactive coordinate. In both cases, the $x + d$ stays in the
$\ell_\infty$-ball of radius $\Delta$ around $x$.

To visualize this construction of $d$,
Figure~\ref{fig:project_to_bstat_cases} shows sketches
of the three situations that can occur for a pair
of coordinates $x_{1,i}$, $x_{2,i}$
if one coordinate is strictly positive
and $x \in {\cal N}^\infty$. The rightmost sketch shows
how the solution of LPCC$(x,\Delta)$ can detect
a descent direction for the point $x^\infty$
even if the descent is not available at $x$
by virtue of the intermediate projection of $x$ to $\hat{x}$.

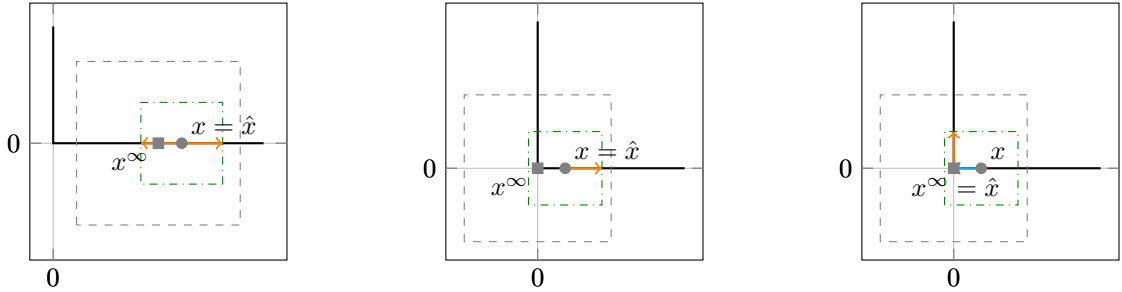
\begin{figure}[h]
\minipage{0.33\textwidth}
\centering
\begin{tikzpicture}
\begin{axis}[height=5cm, xmin=-.1, xmax=1.0, ymin=-.5, ymax=.6, xtick={0}, ytick={0}, extra x ticks={0},
extra y ticks={0}, extra tick style={grid=major},
unit vector ratio={1 1}]
	\addplot[mark=none, thick] coordinates { (0,.5) (0,0) (.9,0) };

	\addplot[mark=none, dashed, gray] coordinates 
      { (0.1,-0.35) (0.1,0.35) (0.8,0.35) (0.8,-0.35) (0.1,-0.35) };	
	\addplot[mark=none, dash dot, darkgreen] coordinates	
	  { (0.725,-0.175) (0.725,0.175) (0.375,0.175) (0.375,-0.175) (0.725,-0.175) };	
	\addplot[mark=square*,gray] coordinates { (0.45,0.) };	
	\node[below left] at (axis cs:0.45,0) {$x^\infty$};	
	\addplot[mark=*,gray] coordinates { (0.55,0.) };	
	\node[above right] at (axis cs:0.55,0) {$x = \hat{x}$};
	\draw[<->,thick,orange] (axis cs:0.375,0) -- (axis cs:0.725,0);	
\end{axis}
\end{tikzpicture}
\endminipage\hfill
\minipage{0.33\textwidth}
\centering
\begin{tikzpicture}
\begin{axis}[height=5cm, xmin=-.5, xmax=0.9, ymin=-.5, ymax=0.9, xtick={0}, ytick={0}, extra x ticks={0},
extra y ticks={0}, extra tick style={grid=major},
unit vector ratio={1 1}]
	\addplot[mark=none, thick] coordinates { (0,.8) (0,0) (.8,0) };
	\addplot[mark=none, dashed, gray] coordinates
	  { (-0.4,-0.4) (-0.4,0.4) (0.4,0.4) (0.4,-0.4) (-0.4,-0.4) };
	\addplot[mark=none, dash dot, darkgreen] coordinates	
	  { (-0.05,-0.2) (-0.05,0.2) (0.35,0.2) (0.35,-0.2) (-0.05,-0.2) };	  	
	\addplot[mark=square*,gray] coordinates { (0.,0.) };	
	\node[below left] at (axis cs:0.0,0) {$x^\infty$};	
	\addplot[mark=*,gray] coordinates { (0.15,0.) };	
	\node[above right] at (axis cs:0.15,0) {$x = \hat{x}$};
	\draw[->,thick,orange] (axis cs:0.15,0) -- (axis cs:0.35,0);	
\end{axis}
\end{tikzpicture}
\endminipage\hfill
\minipage{0.33\textwidth}
\centering
\begin{tikzpicture}
\begin{axis}[height=5cm, xmin=-.5, xmax=0.9, ymin=-.5, ymax=0.9, xtick={0}, ytick={0}, extra x ticks={0},
extra y ticks={0}, extra tick style={grid=major},
unit vector ratio={1 1}]
	\addplot[mark=none, thick] coordinates { (0,.8) (0,0) (.8,0) };
	\addplot[mark=none, dashed, gray] coordinates
	  { (-0.4,-0.4) (-0.4,0.4) (0.4,0.4) (0.4,-0.4) (-0.4,-0.4) };
	\addplot[mark=none, dash dot, darkgreen] coordinates	
	  { (-0.05,-0.2) (-0.05,0.2) (0.35,0.2) (0.35,-0.2) (-0.05,-0.2) };	  	
	\addplot[mark=square*,gray] coordinates { (0.,0.) };	
	\node[below left] at (axis cs:0.30,0) {$x^\infty = \hat{x}$};	
	\addplot[mark=*,gray] coordinates { (0.15,0.) };	
	\node[above right] at (axis cs:0.15,0) {$x$};
	\draw[->,thick,cyan] (axis cs:0.15,0) -- (axis cs:0.,0.);		
	\draw[->,thick,orange] (axis cs:0,0) -- (axis cs:0.,0.2);
\end{axis}
\end{tikzpicture}
\endminipage
\caption{Three configurations of $x^\infty$ (square),
$x \in {\cal N}^\infty$ (circle), and
$\hat{x}$ in two dimensions, where
one coordinate of $x$ is strictly positive.
The max-norm ball of radius $t_0$ around $x^\infty$
is depicted with a dashed line (gray).
The max-norm ball of radius $.5 t_0$ around $x$
is depicted with a dashed-dotted line (green).
If available, the vector $\hat{p}$ originating from
$x$ is depicted in blue.
The possible vectors $\Delta s$ 
with $\Delta = \frac{1}{2} t_0$ originating from
$\hat{x}$ are depicted in orange.}\label{fig:project_to_bstat_cases}
\end{figure}

Next, we show the estimate~\eqref{eq:suff_decrease_seq} holds
under the condition~\eqref{eq:rho_bounds} on $\Delta$ for $d$. Afterwards, we
transfer the result to the sequence $\xk \to x^\infty$
and show that the condition on $\|\hat{p}\|_\infty$ is satisfied for $k$
sufficiently large.
We choose $\delta 
= \|\widehat{p}\|_\infty$ and $L \coloneqq \sup \{ \|\nabla f(\xi)\|_1 : \xi \in {\cal N}^\infty \}$, and obtain the estimate
\begin{align}
f(x) - f(x + d)
&\ge - \nabla f(x)^T d  - \frac{1}{2} \|d\|^2_\infty M \nonumber \\
&\ge - \nabla f(x)^T (\Delta s) -\nabla f(x)^T \widehat{p} - (\Delta)^2 M - \delta^2 M\label{eq:ie_w_pres} \\
&\ge - \sigma \nabla f(x)^T (\Delta s) - (1 - \sigma) \nabla f(x)^T (\Delta s) - L \delta - (\Delta)^2 M - \delta^2 M, \label{eq:ie_w_s}
\end{align}
where the first inequality follows from Lemma~\ref{L:Taylor} with
$M$ computed for the choices $x = x^\infty$ and $r = \sup\{ \|\xi -
x^\infty\|_\infty : \xi \in {\cal N}^\infty \}$. The second inequality follows
from the definition of $d$ in~\eqref{eq:d_def}, the triangle inequality, and
the inequality $(a + b)^2 \le 2a^2 + 2b^2$ for real numbers $a$ and $b$. The
third inequality follows from the Cauchy--Schwarz inequality and the
definitions of $L$ and $\delta$.

If the condition $\Delta \le (1 - \sigma)\varepsilon / (4M)$ holds, then it follows that
\[ (\Delta)^2 M + \delta^2 M \le 2(\Delta)^2 M \le \frac{\Delta (1 - \sigma)\varepsilon}{2} \le -\frac{1}{2}(1 - \sigma) \nabla f(x)^T (\Delta s), \]
where the first inequality follows from
$\delta = \|\hat{p}\|_\infty \le \Delta$ by construction and
the last inequality holds because $0 < \varepsilon \le - \nabla f(x)^T s$.

Moreover, if $\Delta \ge 2 \delta L / ((1 - \sigma)\varepsilon)$
holds, then it follows that
\[ L\delta \le \frac{\Delta (1 - \sigma)\varepsilon}{2} \le -\frac{1}{2}(1 - \sigma) \nabla f(x)^T (\Delta s). \]
Thus, with the definitions
\[ \eta \coloneqq \frac{2L}{(1 - \sigma)\varepsilon}
\enskip\text{ and }\enskip
\kappa \coloneqq \min\left\{\frac{1}{2} t_0, \frac{(1 - \sigma)\varepsilon}{4M}\right\}, \] 
the analysis above implies that
if $\Delta \in [\eta\delta,\kappa]$, then 
\begin{gather}\label{eq:suff_decrease_noseq}
	f(x) - f(x + d) \ge - \sigma \nabla f(x)^T (\Delta s)
\ge \sigma \varepsilon \Delta.
\end{gather}
\revised{R2.7}{Moreover,
\begin{gather}\label{eq:step_accept_noseq}
	f(x) - f(x + d) \ge - \sigma \nabla f(x)^T d
\end{gather}
follows by replacing the terms $\sigma \nabla f(x)^T (\Delta s)$, 
$\nabla f(x)^T \widehat{p}$, 
and
$L\delta$ with $\sigma \nabla f(x)^T d$,
$(1 - \sigma)\nabla f(x)^T \widehat{p}$, and $(1-\sigma)L\delta$ in \eqref{eq:ie_w_pres}, and
\eqref{eq:ie_w_s}.
Note that \eqref{eq:step_accept_noseq} already
holds if $\Delta \in [(1 - \sigma)\eta\delta,\kappa]$.}

Finally, we start from the obtained estimate
\eqref{eq:suff_decrease_noseq} from $x \in {\cal N}^\infty$
to prove the claim for a sequence $\xk \to x^\infty$ in ${\cal N}^\infty$.
To this end, we consider a sequence $\xk \to x^\infty$ in
${\cal N}^\infty$ and inspect the quantities introduced above
for the choices $x = \xk$ and denote them with the superscript
$k$. For all $0 < \Delta \le \kappa$, there exists $k_0 \in \mathbb{N}$
such that for all $k \ge k_0$ it holds that
$\delta^k = \|\widehat{p}^k\|_\infty \le \|x^\infty - x^{k}\|_\infty
\le \Delta$ by construction of $\widehat{p}^k$ and the fact that
$\|x^\infty - x^{k}\|_\infty \to 0$.

Thus,~\eqref{eq:suff_decrease_noseq} \revised{R2.7}{and \eqref{eq:step_accept_noseq}
imply} that \eqref{eq:suff_decrease_seq} \changed{and \eqref{eq:step_accept_seq}
hold} for all $\Delta$ satisfying~\eqref{eq:rho_bounds}
because $\eta$ and $\kappa$ are constant for the neighborhood
${\cal N}^\infty$ while $\xk \to x^\infty$ provides \changed{that $\delta^k \to 0$, and hence} that
$[\delta^k\eta,\kappa]$ is nonempty eventually.
\end{proof}

By construction, the lower bound $\eta\delta^k$
of the interval~\eqref{eq:rho_bounds}
converges to zero as $\xk \to x^{\infty}$, while the upper bound of this
interval is bounded away from zero. By virtue of Lemma~\ref{L:Fstep},
$\varepsilon$ depends only on $x^\infty$ and
is independent of $k$.

Hence, \revised{R2.7}{the iterates cannot converge to a non-B-stationary point 
$x^{\infty}$ because Lemma \ref{L:Fstep} shows that sufficiently small steps
in a neighborhood of $x^{\infty}$ can always be accepted.
These steps satisfy a sufficient reduction condition bounded below by the
trust-region radius so that the algorithm must select steps that are acceptable
and eventually improve over $x^\infty$.}

We are now in a position to prove
our main convergence result, Theorem~\ref{T:main},
which in particular shows that every accumulation point is B-stationary.

\begin{proof}[Proof of Theorem~\ref{T:main}]
We need to consider only Outcome {\bf O2}, because in the other case we obtain
a B-stationary point by virtue of the finite termination condition. We deduce
inductively over the iterations that $\xk$ is feasible for~\eqref{E:mpcc} if
the inner loop terminates finitely for every $k$ because if $\dn$ is feasible
for LPCC$(\xk,\trn)$, it follows that $x^{k+1} = \xk + \dn$ is feasible for
\eqref{E:mpcc}. Steps are accepted only when there is a reduction in the
objective. Thus all iterations reduce the objective function if the inner loop
terminates finitely. The inner loop terminates finitely by
Lemma~\ref{L:finite}.

It remains to show that every accumulation point of the sequence of iterates is B-stationary.
We seek a contradiction and assume that an accumulation point
of the sequence $x^{\infty}$ is not B-stationary.
To ease the notation, we denote the approximating subsequence by the same
symbol, that is, $\xk \to x^\infty$. Moreover, we further restrict ourselves to
a subsequence such that its elements are in the neighborhood ${\cal N}^\infty$
given by Lemma~\ref{L:Fstep}. Lemma~\ref{L:Fstep} implies that for any $\Delta$
in the range
\[ \delta_k \eta \le \Delta \le \kappa \]
the conditions for a step that is accepted
by Algorithm~\ref{SLPCC}, line~\ref{ln:accept},
and also leads to a linear reduction with respect
to $\Delta$ are satisfied. 
We note that the upper bound of this range is constant and that
$\delta_k \to 0$.
Thus, there exists $k_0 \in \mathbb{N}$ such that for
all $k \ge k_0$ it holds that $\delta_k \eta < \frac{1}{2}\kappa$
and $\delta_k < \frac{1}{2}\kappa\varepsilon / \sup\{ \|\nabla f(x)\|_1
\,|\, x \in {\cal N}^\infty \}$.
Moreover, Lemma~\ref{L:Fstep} also asserts the existence
of a fixed $\varepsilon > 0$ that enters the lower-bound
estimate of the reduction, which we use frequently below.

Next, we distinguish two cases. First, assume that in outer
iteration $k$ the inner loop accepts $\dn$ such that
$\|\dn\|_\infty = \trn > \kappa$.
Then, by virtue of Lemma~\ref{L:Fstep}, there exists another
step $\tilde{d}^k$ that
is the solution of
LPCC$(\xk,\kappa)$, which we can write as
$\tilde{d}^k = \Delta s + \hat{p}^k$
with $\|\hat{p}^k\|_{\infty} = \delta^k$,
see~\eqref{eq:d_def} in the proof 
of Lemma~\ref{L:Fstep}.
We obtain
\[ -\sigma \nabla f(\xk)^T \tilde{d}^k
= -\sigma \nabla f(\xk)^T (\kappa s + \hat{p}^k)
\ge \frac{1}{2}\varepsilon  \sigma \kappa, \]
where the inequality follows from
the estimate
$-\sigma \nabla f(\xk)^T \kappa s
\ge \varepsilon \sigma \kappa$
from Lemma~\ref{L:Fstep} and the estimate
$|\sigma \nabla f(\xk)^T \hat{p}^k|
\le \sigma \delta_k \|\nabla f(\xk)\|_1
\le \frac{1}{2}\sigma\kappa\varepsilon$
due to the choice of $k_0$. We deduce that
\begin{equation}\label{eq:suffRedn1}
f(\xk) - f(\xk + \dn) \ge -\sigma \nabla f(\xk)^T\dn
  \ge -\sigma \nabla f(\xk)^T \tilde{d}^k \ge
  \varepsilon \frac{1}{2} \sigma \kappa,
\end{equation}
where the first inequality follows from the acceptance
of the step $d^{k,l}$ and the second inequality
follows from the fact that $\tilde{d}^k$
is feasible for LPCC$(\xk,\trn)$ because $\tilde{d}^k$
is feasible for LPCC$(\xk,\kappa)$ and
$\kappa < \trn$.

Next, we assume that the inner loop accepts
$\dn$ such that $\|\dn\|_\infty = \trn \le \kappa$.
Since the upper bound $\overline{\Delta}^k$
on $\Delta^{k,0}$ is nondecreasing,
the interval $[\underline{\Delta},\overline{\Delta}^k]$
is never empty.
Starting from the reset trust region radius $\underline{\Delta} \le \Delta^{k,0}$,
the inner loop will eventually choose a trust region radius
$\trn \in \{\Delta^{k,0}, \Delta^{k,0}/2, \Delta^{k,0}/4, \ldots\}$.
We know that the sufficient reduction condition in
Algorithm~\ref{SLPCC}, line~\ref{ln:accept},
was not satisfied for $\trn  > \kappa$.
By virtue of Lemma~\ref{L:Fstep}, our choice of $k_0$
and the fact that
the inner loop always halves the tested trust region
radius, we obtain $\frac{1}{2} \kappa < \trn$ for the trust region radius
$\trn$ in inner loop $l$ that leads to acceptance in Algorithm~\ref{SLPCC},
line~\ref{ln:accept}. In this case we also obtain
\begin{equation}\label{eq:suffRedn2}
f(\xk) - f(\xk + \dn) \ge -\sigma \nabla f(\xk)^T(\trn s)
> \varepsilon \frac{1}{2} \sigma \kappa\revised{R2.8}{.}
\end{equation}
%

Combining~\eqref{eq:suffRedn1} and~\eqref{eq:suffRedn2}, we bound the actual reduction $f(\xk) - f(\xk + \dn)$
from below for all iterations $k \ge k_0$ of the considered subsequence.
It follows that
\begin{gather}\label{eq:diverging_sum} f(\xk) = f(x^{k_0})
- \sum_{m=k_0}^{k - 1} f(x^{m}) - f(x^{m + 1}) \le f(x_{k_0}) -  (k - k_0) \sigma \varepsilon \frac{\kappa}{2}  \to -\infty
\end{gather}
for $k \to \infty$.
This contradicts the fact that the iterates $\xk$ of the subsequence
converge to $x^\infty$, which implies $f(\xk) \to f(x^\infty) \in  \R$
because $f$ is continuous.
\end{proof}

\begin{remark}
A crucial ingredient of the proof is that the trust region
radius is reset in every iteration of the outer loop, which is different
from classical statements of trust region algorithms.
Consequently, whenever a subsequence starts to approach a non-B-stationary
point, the trust region radius cannot contract to zero, which 
implies that the subsequence eventually moves away from this point,
due to Lemma~\ref{L:Fstep}. A similar technique is used in the convergence
analysis in~\cite{FleLeyToi:02}.
\end{remark}

\section{Including Second-Order Information}\label{sec:bqp}

The subproblem~\eqref{eq:LPCC} uses only first-order information, resulting in
a step to the boundary of the trust region or a step to the kink of the
complementarity constraints, which generally result in slow convergence, as shown in
Figure~\ref{fig:slpecbqp_ft_slpec} for an example that is
described in Section~\ref{sec:example_slow}.
In this section, we discuss two ways to include second-order information in
Algorithm~\ref{SLPCC}. The first approach proposes a Cauchy line search
along a quadratic approximation, and the second approach adds an additional 
bound-constrained quadratic programming (BQP) step on the active set identified
by the LPCC. In this section, we assume that $f$ \revised{R2.5}{is twice 
continuously differentiable}.

\subsection{Cauchy Line Search with Quadratic Models}\label{sec:cauchy_line_search}

We first show how to perform a 
line search on a quadratic model in the direction of the LPCC step.
We use the second-order Taylor expansion to define the quadratic model
that is minimized.

\begin{figure}
\minipage{0.48\textwidth}
\centering
\begin{tikzpicture}
\begin{semilogyaxis}[height=6cm, xmin=0., xmax=4.5, ymin=4, ymax=10000, xtick={0}, extra x ticks={0,1,2,3,4,5}]
	\addplot[mark=*, very thick, darkgreen] coordinates { (0,5.1524e+03) 
	                                      (1,4.6316e+02)
	                                      (2,3.9115e+01)
                                          (3,1.2198e+01)};                                       
    \addlegendentry{plain}
	\addplot[mark=*] coordinates { (0,+5.1524e+03) 
	                               (1,4.6316e+02)
	                               (2,5.4809e+01)
                                   (3,1.3366e+01) 
                                   (4,1.2198e+01) };                                                                                
    \addlegendentry{cauchy}                   
\end{semilogyaxis}
\end{tikzpicture}
\endminipage\hfill\minipage{0.48\textwidth}
\centering
\begin{tikzpicture}
\begin{loglogaxis}[height=6cm, xmin=0.2, xmax=3000, ymin=4, ymax=10000]
\addplot[mark=*, very thick, darkgreen] coordinates { (0.2,5.152406250000000000e+03)
(1.000000000000000000e+00,2.537725694444444343e+03) };
\addplot[mark=*] coordinates { (0.2,5.152406250000000000e+03)
(1.000000000000000000e+00,2.537725694444444343e+03) };
\addplot[mark=*, very thick, darkgreen] table[x index=0,y index=1,col sep=comma] {./data/fvals_nash1aAL_plain.csv};
\addlegendentry{plain} 
\addplot[mark=*] table[x index=0,y index=1,col sep=comma] {./data/fvals_nash1aAL_csi.csv}; 
\addlegendentry{cauchy}                   
\end{loglogaxis}
\end{tikzpicture}
\endminipage
\caption{Convergence of the objective value
for the example \eqref{eqn:figure_4_example}
described in Section~\ref{sec:example_slow}
over the iterations of
Algorithm~\ref{SLPCC} with BQP acceleration
in ln.\ \ref{ln:bqp} (left) vs.
Algorithm~\ref{SLPCC} without BQP acceleration (right).
The convergence is plotted with
(\texttt{cauchy}, black) and without (\texttt{plain}, green) an optional Cauchy line search in ln.\ \ref{ln:cauchy} of Algorithm~\ref{SLPCC}.}
\label{fig:slpecbqp_ft_slpec}
\end{figure}
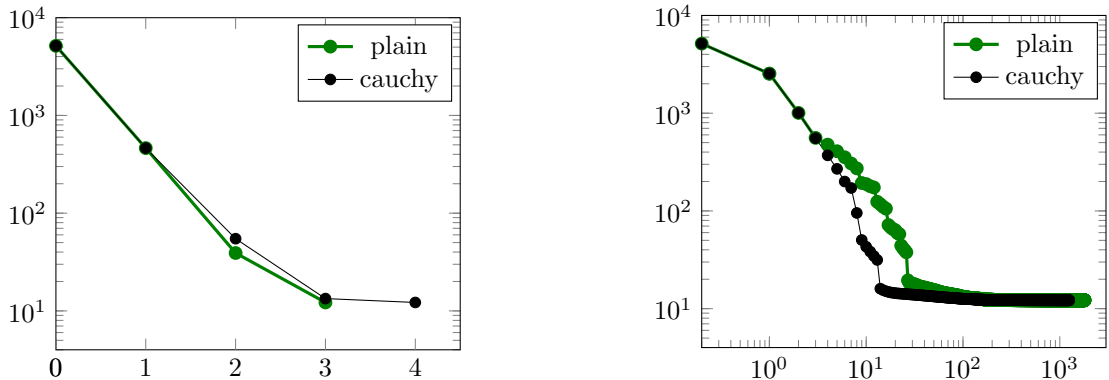

We follow the procedure for bound-constrained quadratic optimization 
that is described in Chapter 16.7 of~\cite{nocedal2006no}.
There, the piecewise path is defined by following the negative
gradient direction until a bound is reached in one of the coordinates.
Then, this coordinate is fixed, and the path continues by following
the negative gradient direction only in the other coordinates,
thereby resulting in a piecewise-linear path.
The Cauchy point is defined as the first local minimizer of the
quadratic model along this path.

We can define the search for a Cauchy point for problems of the
form~\eqref{E:mpcc} using a quadratic model
for $f$ along with some subtle changes
to the approach in~\cite{nocedal2006no}.
As in~\cite{nocedal2006no}, we use $-\nabla f(\xk)$ as the
search direction to compute the piecewise-linear path starting
from $\xk$.
For the entries of $x_0$, the procedure is the same as
in Chapter 16.7 of~\cite{nocedal2006no}.
For the other coordinates, we note that $x_{1i}$
and $x_{2i}$ cannot both become nonzero
but for
$i \in \{1,\ldots,n_1\}$, only one of the coordinates
$x_{1i}$ and $x_{2i}$ may be changed to remain inside
the feasible set, that is, to satisfy complementarity.
If the path reaches a kink, 
not only do we stop changing the coordinate with which we
arrived at the kink but we also pivot to the other coordinate
involved in the kink if the corresponding entry of the current search direction is positive. 

If a complementarity constraint is biactive at the start of this
piecewise-linear path, that is $\xk_{1i} = 0 = \xk_{2i}$ for some
$i \in \{1,\ldots,n_1\}$,
and if the search direction
is positive in both coordinates, we must
choose which coordinate is increased and which one is fixed to zero.
In this case, we make a greedy choice and
increase the one with
the larger entry in the search direction.
Ties are broken in lexicographical order.
The procedure is summarized in Algorithm~\ref{Cauchy}.
The symbol $\mathbf{1}_n$ in Algorithm \ref{Cauchy}
denotes the vector in $\R^n$ that is equal to $1$ in all components.

We note that the Cauchy point computation 
following the negative gradient direction vector
does not give us global convergence to B-stationary
points. The counterexample given in~\eqref{eq:ex1}
is also valid for the Cauchy point computation.
However, there are several possibilities to obtain
convergence by leveraging the LPCC$(\xk,\trn)$ step.
We first observe that by construction, the computed Cauchy point
is feasible and contained in an $\ell_\infty$-ball around $\xk$.

One option is to base the sufficient reduction condition on the
LPCC$(\xk,\trn)$ step.
We accept the Cauchy point, if its actual reduction is larger than the
actual reduction of the LPCC$(\xk,\trn)$ step (otherwise we use the
LPCC$(\xk,\trn)$ step). With this option, the Cauchy point
computation becomes a post-processing step and convergence
follows from the analysis of the LPCC steps.

An alternative approach that allows us to save evaluations of the 
objective function is to check whether the actual reduction from the 
Cauchy point satisfies the sufficient decrease condition with respect
to the linear predicted reduction from the LPCC$(\xk,\trn)$ step.
Only if this is not the case, do we need to check the sufficient decrease
condition as well.
This approach guarantees that the sufficient decrease condition
to obtain convergence of Algorithm~\ref{SLPCC}
to a local minimizer of~\eqref{E:mpcc}
will eventually be satisfied by virtue of the arguments in 
Section~\ref{S:proof}.
We include this variant of using the Cauchy point computation
in Algorithm~\ref{SLPCC} and Algorithm~\ref{Cauchy}.
The empirical results in Section~\ref{sec:synthetic_benchmark_results}
indicate that we can reduce the number of iterations
with the introduction of this step.

We note that there is at least one other approach
to define a piecewise-linear path along which a quadratic model, $q_k$ in Algorithm~\ref{Cauchy},
is minimized. One can also generate and record a pivoting sequence,
similar to what the solver PATH does to solve a set of linear
and complementarity conditions; see~\cite{cottle1967complementary,ferris1999interfaces}.
Having obtained this pivoting sequence, one can backtrack the path
and minimize a quadratic model on the different segments.

\begin{algorithm}[htb]
\caption{\texttt{FIND\_CAUCHY\_POINT}$(\xk,\dn,\trn)$}\label{Cauchy}
  \fontsize{8}{8}\selectfont
\begin{algorithmic}[1]
  \State Compute piecewise-linear path $s : [0,1] \to \R^n$
  from $\xk$ in direction $-\nabla f(\xk)$
  that satisfies the complementarity constraints
  with bounds given by $\ell_0$, $u_0$, $\xk - \mathbf{1}_n \trn$,
  and $\xk + \mathbf{1}_n \trn$.
  \State $t^* \gets $ first local minimizer of $q_k(t) = 0.5 s(t)^T \nabla^2 f(\xk) s(t) + \nabla f(\xk)^T s(t)$ for $t \in [0,1]$
  \State Evaluate $f(\xk + s(t^*))$ and compute $\rho^{k,l} \gets \frac{f(\xk) -f(\xk + s(t^*))}{- \nabla f(\xk)^T \dn}$
  \If{$\rho^{k,l} \ge \sigma$}
  \State {\bf return} $s(t^*)$
  \Else
  \State {\bf return} $\dn$
  \EndIf
\end{algorithmic}   
\end{algorithm}

The Cauchy step alone does not alleviate the slow convergence of the LPCC
algorithm, and, hence, we consider adding a second-order step next.

\subsection{Bound-Constrained Quadratic Programming  (BQP) Step}\label{sec:bqp_steps}

Here, we propose a second-order step, similar to sequential linear/quadratic programming
techniques for nonlinear programming (SLQP), see, e.g.,~\cite{FleSai:89,ChiFle:03,ByrdGoulNoceWalt04:mp,lenders2017pysleqp}.
Given an estimate of the active set produced by the LPCC step (or the Cauchy
point), we consider the predicted active sets
\[ {\cal A}_0(\xk+\dn), \; {\cal A}_1(\xk+\dn), \; {\cal A}_2(\xk+\dn), \; \text{and} \; {\cal D}(\xk+\dn) \]
as defined
in Section~\ref{sec:proof_prelims}. One important question is how to handle the
degenerate constraints ${\cal D}(\xk+\dn)$. We may choose any partition ${\cal
D}_1^k, {\cal D}_2^k \subset {\cal D}(\xk+\dn)$ to define the components
of the biactive set that are free variables for the BQP step.
In our implementation, we partition the
set greedily
with respect to the gradient; that is, we
partition ${\cal D}(\xk+\dn)$ into
\begin{align*}
 {\cal D}_1^k &= \displaystyle \left\{ i \in {\cal D}(\xk+\dn)\,|\, \nabla f(\xk+\dn)_{1,i} \ge \nabla f(\xk+\dn)_{2,i} \right\},\\
	{\cal D}_2^k &= {\cal D}(\xk+\dn)\setminus
	{\cal D}_1^k,
\end{align*}
which corresponds to fixing the component
with the larger gradient. Given these active sets,
we define the BQP as follows:
\begin{gather*}
\mbox{BQP}(x^k,\Delta) 
    \left\{
    \begin{aligned}
      \mini_d\  & q_k(d) \coloneqq \nabla f_k^T d + \frac{1}{2} d^T H_k d & \\
      \st\      & \ell_0 \le x^k_{0} + d_{0} \le u_{0}, & \\
      & x^k_{1,i} + d_{1,i} = 0 \text{ and } x^k_{2,i} + d_{2,i} \ge 0 & \forall i \in {\cal A}_1(\xk+\dn)\backslash {\cal D}_1^k, \\
      & x^k_{1,i} + d_{1,i} \ge 0 \text{ and } x^k_{2,i} + d_{2,i} = 0 & \forall i \in {\cal A}_2(\xk+\dn)\backslash {\cal D}_2^k, \\
& \|d\|_\infty \le \Delta,
    \end{aligned} \right. 
\end{gather*}
which is used to compute a second-order step, as described in Algorithm~\ref{SLPCC_BQP}.
Here $H_k$ is an approximation
of the Hessian $\nabla^2f(\xk)$.

\begin{algorithm}[htb]
\caption{\texttt{SOLVE\_BQP}$(\xk,\Delta^k_{QP},\rho^{k,l})$}\label{SLPCC_BQP}
  \fontsize{8}{8}\selectfont
\begin{algorithmic}[1]
  \algrenewcommand{\algorithmiccomment}[1]{\hfill \# \texttt{ #1}}
  \State Compute a BQP$(\xk, \Delta_{QP}^k)$ step $d^{k}_{QP}$
	   and ratio $\rho^k_{QP} = \frac{f(x^k) - f(x^k+ d^{k}_{QP})}{q(0) - q(d^{k}_{QP})}$
	   \State Update
	   \[ \Delta^{k+1}_{QP} = \left\{ \begin{aligned}
	   \min\{\overline{\Delta}^k,2\Delta^k_{QP}\} & \text{ if } \rho^k_{QP}  \ge 0.75,\\
	   \Delta^k_{QP} & \text{ if } 0.75 > \rho^k_{QP} \ge 0.25,\\
	   \frac{1}{4} \Delta^k_{QP} & \text{ otherwise}
	   \end{aligned}\right.\]
	   \State Accept step if $\rho_{QP}^k \geq \frac{\rho^{k,l}}{2}$,
	   update $x^{k+1} = x^k + d_{QP}^k$\label{ln:acceptance}
\end{algorithmic}   
\end{algorithm}

The identity
${\cal D}(\xk+\dn) = {\cal A}_1(\xk+\dn) \cap {\cal A}_2(\xk+\dn)$
implies that for all $i \in \{1,\ldots,n_1\}$, 
either $i \in {\cal A}_1(\xk+\dn)\setminus \revised{R1.2}{\cal D}_1^k$ or 
$i \in {\cal A}_2(\xk+\dn)\setminus \changed{\cal D}_2^k$.
Consequently $0 \le \xk_{1,i} + d_{1,i} \perp \xk_{2,i} + d_{2,i} \ge 0$
holds for all $d$ that are feasible for BQP$(\xk, \Delta_{QP}^k)$
and hence $\xk + d$ is feasible for~\eqref{E:mpcc}.

We note that the BQP step in Algorithm~\ref{SLPCC_BQP} differs from the SLQP approaches,
because we only fix the complementarity constraints, and solve a bound-constrained QP,
rather than an equality-constrained QP. Because our problem involves only bound constraints,
solving it is computationally not much harder than solving an equality-constrained QP.

\revised{R2.9}{We briefly comment on the convergence of SLPCC with BQP steps.}
The introduction of the additional BQP step does not change the outline
of the proof of Theorem~\ref{T:main}. The only small difference is that the
right-hand side of the second inequality
in~\eqref{eq:diverging_sum} now requires the factor $1/2$ in front
of the term $(k - k_0)\sigma \varepsilon \frac{\underline{\Delta}}{2^l}$.
This follows from the acceptance criterion of the step in
Algorithm~\ref{SLPCC_BQP}, line~\ref{ln:acceptance}.
Under such an acceptance criterion it does not matter whether we use
BQP steps or solve other subproblems
to accelerate the convergence of Algorithm~\ref{SLPCC}.

If the inner loop (that is, the solution of LPCC$(\xk,\trn)$) identifies
the optimal active set for all iterations $k \ge k_0$ for some
finite $k_0 \in \mathbb{N}$ and appropriate
Hessian approximations are used, then the BQP steps can be regarded 
as SQP steps on the reduced problem, which yield superlinear convergence.

\subsection{Illustrative Example}\label{sec:example_slow}

We now demonstrate the potential effect
of using BQP steps in Algorithm~\ref{SLPCC}.
We choose an augmented Lagrangian
subproblem derived from
\texttt{nash1a} in the library \texttt{MacMPEC}~\cite{leyffer2000macmpec}, 
with $n_0 = 4$ and $n_1 = 2$. We use the augmented Lagrangian parameters
$\rho = 2$, $\lambda_1 = 3.9375$, $\lambda_2 = -6.5$, $\lambda_3 = -0.25$,
and $\lambda_4 = 2.5$.
The resulting objective function becomes
\begin{gather}\label{eqn:figure_4_example}
\begin{aligned}
f(x) \coloneqq &\ \frac{1}{2}\left((x_{0,1} - x_{0,3})^2 + (x_{0,2} - x_{0,4})^2\right) \\
&\ + \lambda_1\left(-34 + 2 x_{0,3} + \frac{8}{3}x_{0,4} + x_{2,1}\right)\\
&\ - \lambda_2\left(-24.25 + 1.25x_{0,3} + 2 x_{0,4} + x_{2,2}\right) \\
&\ - \lambda_3\left(x_{1,1} + x_{0,2} + x_{0,3} - 15\right) \\
&\ + \lambda_4\left(x_{1,2} + x_{0,1} - x_{0,4} - 15\right) \\
&\ + .5\rho\Big( (-34 + 2x_{0,3} + \frac{8}{3}x_{0,4} + x_{2,1})^2
                 + (-24.25 + 1.25x_{0,3} +  2 x_{0,4} + x_{2,2})^2\\
&\quad + (x_{1,1} + x_{0,2} + x_{0,3} - 15)^2
       + (x_{1,2} + x_{0,1} - x_{0,4} - 15)^2\Big)
\end{aligned}
\end{gather} 
We run Algorithm~\ref{SLPCC} with Algorithm~\ref{SLPCC_BQP} using
\revised{R2.10}{$\underline{\Delta} = 2$}, $\overline{\Delta}^0 = 2$\changed{,}
and initial point $x^0 = 0$ until a first-order optimality of $10^{-7}$
is reached.

For this example, Algorithm~\ref{SLPCC} with Algorithm~\ref{SLPCC_BQP}
reaches a B-stationary point within a tolerance of $10^{-7}$
after $3$ iterations,
while Algorithm~\ref{SLPCC} reaches the tolerance after 1250 (with
Algorithm~\ref{Cauchy} in ln.\ \ref{ln:cauchy}, labeled \texttt{cauchy}) and 1798
(without Algorithm~\ref{Cauchy}, labeled \texttt{plain})
illustrating the advantage of BQP steps.
Figure~\ref{fig:slpecbqp_ft_slpec} shows the decrease in the objective value over the iterations
for both algorithms.
The slow convergence of Algorithm~\ref{SLPCC}
is due to the fact that, without using Line~\ref{ln:bqp}, 
the algorithm behaves like a
steepest descent method whose rate
of convergence is at best linear.
When using second-order information
at the expense of
an additional BQP solve, the algorithm converges quickly because
the active sets ${\cal A}_1(x^*) = \{0, 1\}$ and
${\cal A}_2(x^*) = \emptyset$ are correctly identified after the first
iteration and remain constant over the remaining iterations.

\section{Numerical Results for Synthetic Benchmark Problems}\label{sec:synthetic_benchmark_results}

To obtain quantitative results, we implemented
Algorithm~\ref{SLPCC} in Python and benchmarked
it on two classes of benchmark problems: quadratic problems and general 
nonlinear problems.

All instances were solved with two variants of Algorithm~\ref{SLPCC}. The
first variant includes taking BQP steps (labeled \texttt{plain}),
and the second variant includes both Cauchy steps as well as
BQP steps (labeled \texttt{cauchy}) as presented in Section~\ref{sec:bqp}.
The experiments were executed on a compute server
with four \textsc{Intel(R) Xeon(R) CPU E7-8890 v4} CPUs, clocked at 2.20\,GHz.
Because our problem instances are nonconvex due to the complementarity constraints,
the two versions of our algorithm, \texttt{plain} and \texttt{cauchy}, may return 
different local solutions. We used the open source library
ALGLIB\footnote{\url{https://www.alglib.net/}, Sergey Bochkanov}
to compute the BQP steps.

We also compare our implementations with four state-of-the-art NLP
solvers, namely filterSQP~\cite{FleLey:02}, IPOPT~\cite{Wachter2005}, MINOS~\cite{MurSaun:03,Murtagh1978}, and SNOPT~\cite{GilMurSau:05}.
NLP methods have been shown to currently be arguably the most efficient solvers
for MPECs; see, for example,~\cite{FletLeyf:04,FLRS:02,LeyLopNoc:06,RaghBieg:05}. NLP solvers
reformulate the complementarity constraint in~\eqref{E:mpcc} as a set of inequalities,
\[ x_1 \geq 0,\; x_2 \geq 0,\; x_1^T x_2 \leq 0, \]
where we do not need to enforce equality on the nonlinear constraint because the solvers
will maintain feasible iterates with respect to the simple bounds. (It has been shown 
that using $x_1^T x_2=0$ has worse theoretical properties and produces inferior
numerical results; see~\cite{FLRS:02}.)
We provide the numbers of outer iterations, running times, and
achieved objective values for our approach and the NLP solvers. The NLP solvers
generally have different per-iteration complexities than our approach, but
we include this information mainly to understand whether our approach solves the problems
within a reasonable number of outer iterations. Moreover, these quantities are
difficult to compare because the considered problems have nonconvex feasible
sets and different optimizers may converge to different stationary points.
We list run times as reported by the solvers
themselves.

\subsection{Quadratic Test Problems}
We consider four sets of quadratic test problems.
Each set consists of 10 instances of~\eqref{E:mpcc}
with a quadratic objective function $f$.
The instances in the problem sets are generated randomly and
differ in their size and their spectral properties of the
Hessian of $f$.
\begin{enumerate}
\item Set \texttt{20-ind}: $n_0 = n_1 = 20$,
      $\nabla^2 f(x)$ indefinite.
\item Set \texttt{20-psd}: $n_0 = n_1 = 20$,
      $\nabla^2 f(x)$ positive semidefinite.
\item Set \texttt{40-ind}: $n_0 = n_1 = 40$,
      $\nabla^2 f(x)$ indefinite.
\item Set \texttt{40-psd}: $n_0 = n_1 = 40$,
      $\nabla^2 f(x)$ positive semidefinite.
\end{enumerate}
Further details on the test problem instances are given in 
Appendix~\ref{sec:qpccs}.

The averaged relative difference between the locally
optimal objective values for \texttt{plain} and \texttt{cauchy}
was small for all quadratic problems considered:
($-1\%$ for \texttt{20-ind}, $0\%$ for \texttt{20-psd},
 $2\%$ for \texttt{40-ind}, and $0\%$ for \texttt{40-psd}).

\begin{figure}
  \centering
  \subfloat[][Number of iterations of \\ the outer loop.]
    {\includegraphics[trim={35 20 40 15}, clip, width=0.32\linewidth]{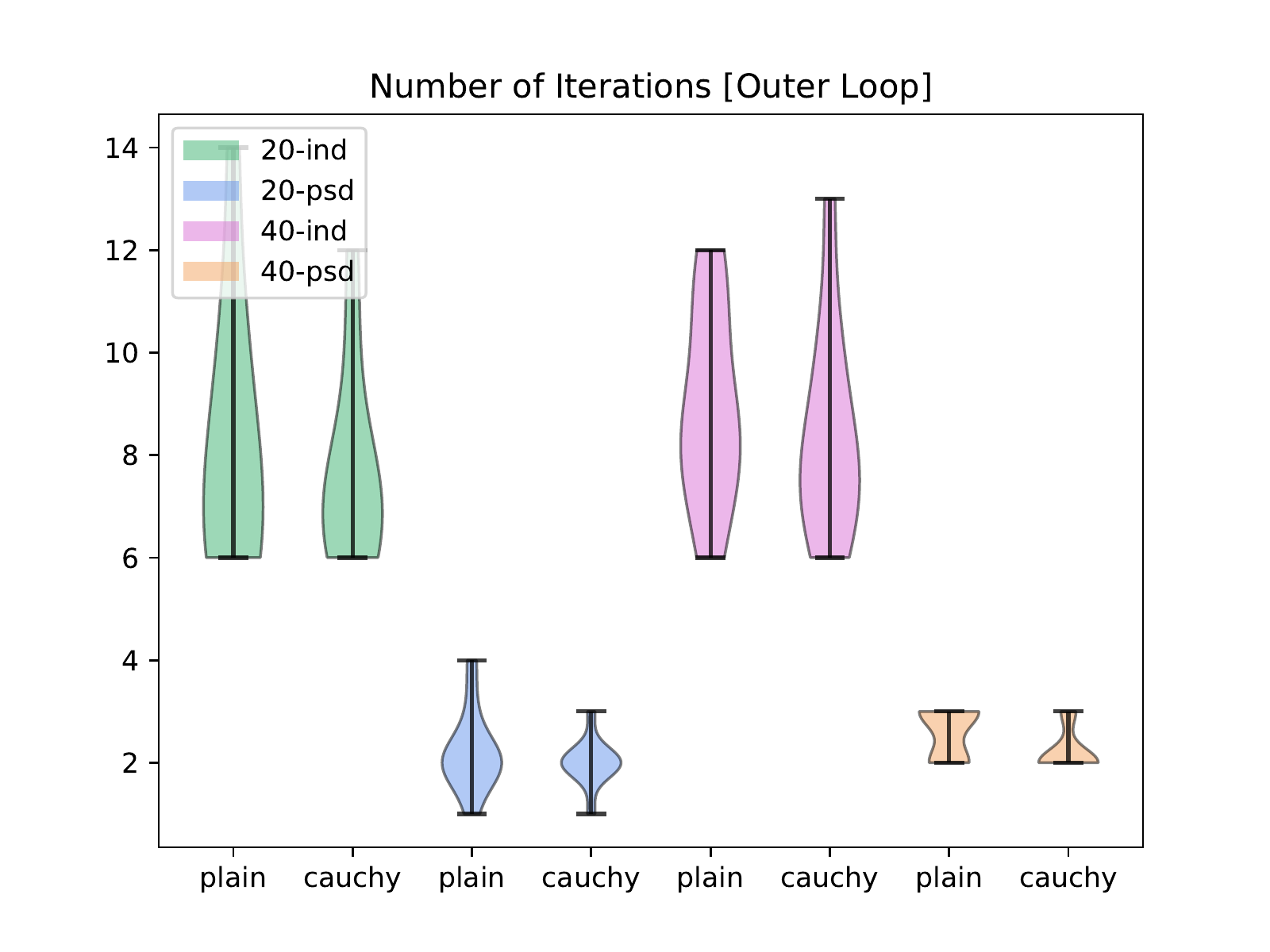}}
    \hfill
  \subfloat[][Average number of \\ iterations of the inner loop.]
    {\includegraphics[trim={35 20 40 15}, clip, width=0.32\linewidth]{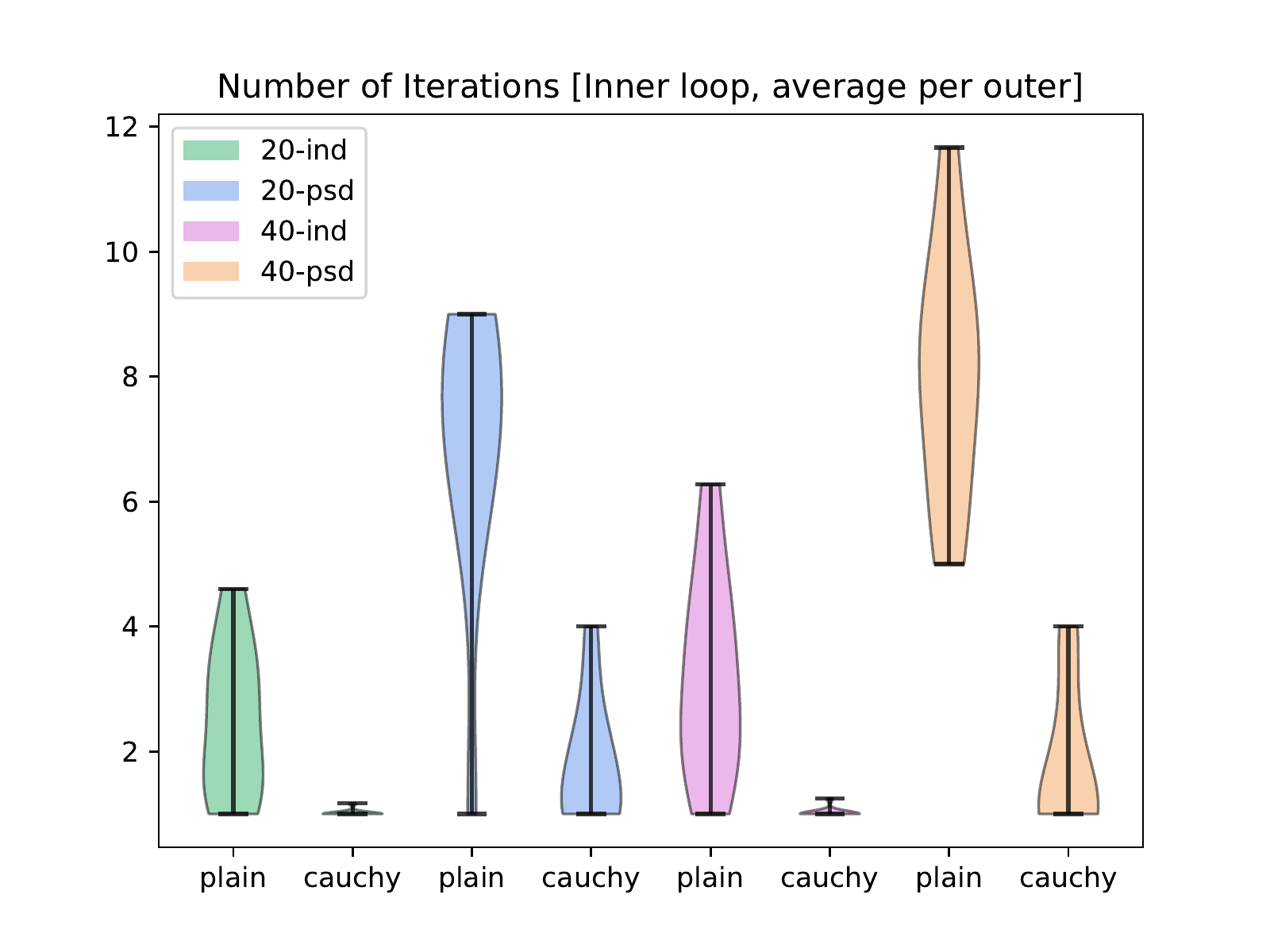}}
    \hfill
  \subfloat[][Average number of BQP \\ iterations.]
    {\includegraphics[trim={25 20 40 15}, clip, width=0.32\linewidth]{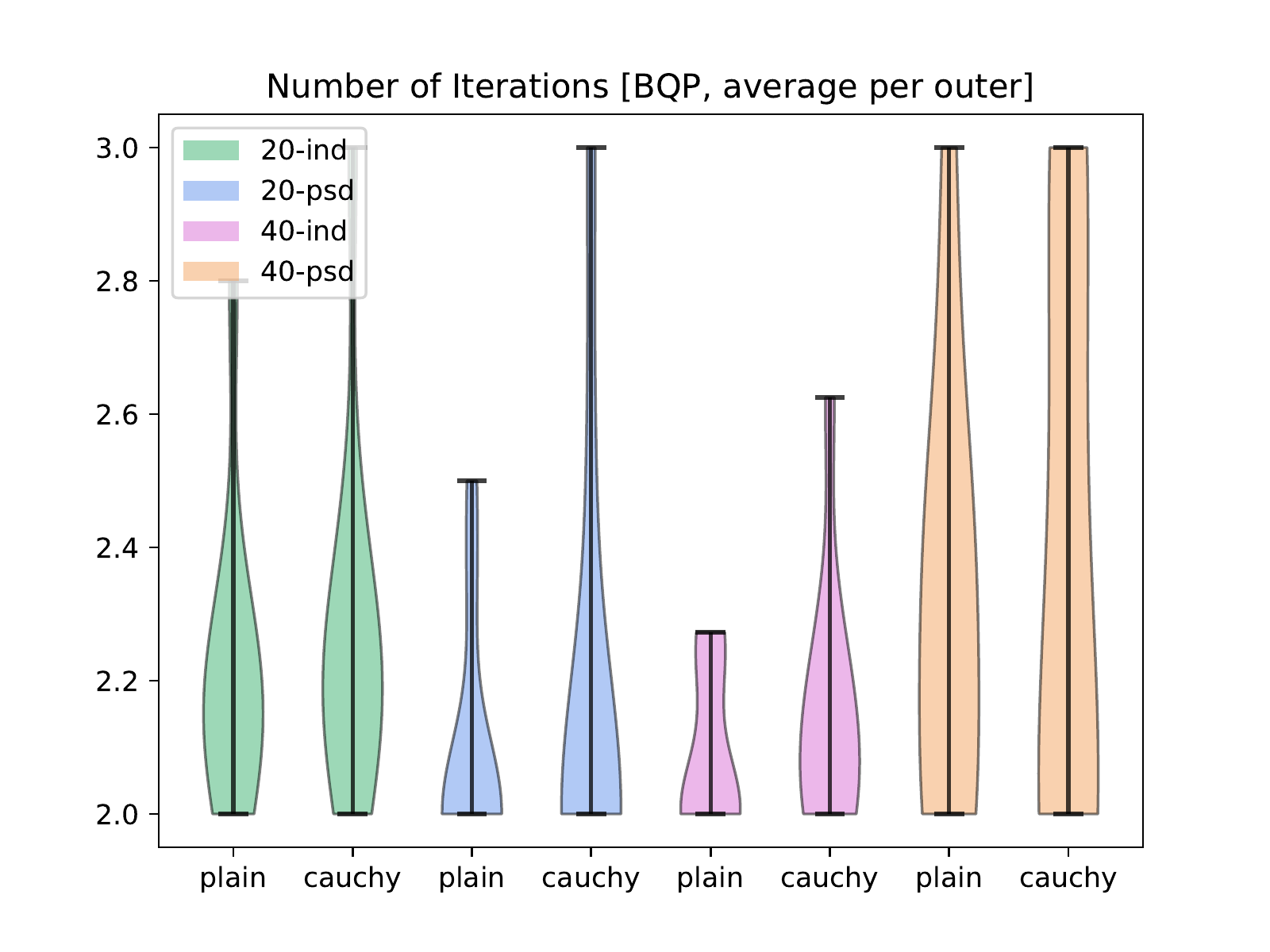}}  
  \caption{Computational results obtained with two variants of
  Algorithm~\ref{SLPCC} on test problem sets
  \texttt{20-ind}, \texttt{20-psd}, \texttt{40-ind}, \texttt{40-psd}.
  Both variants include BQP steps in
  ln.\ \ref{ln:bqp},
  the variant \texttt{cauchy}
  does and the variant \texttt{plain} does not
  include the \texttt{SEARCH\_CAUCHY\_POINT} routine after the LPCC step
  in ln.\ \ref{ln:cauchy}.}
  \label{fig:slpecbqp_1opt}
\end{figure}

Both variants of Algorithm~\ref{SLPCC} converge in a modest
number of outer  iterations (less than 10). Similarly, the average number of
inner iterations per outer iteration is small, which indicates that our
trust region update strategy is efficient. In comparing the two variants,
we note that the addition of the Cauchy step reduces the average number of 
inner iterations by a factor of 2--3 and slightly improves the number of
outer iterations.
Figure~\ref{fig:slpecbqp_1opt} shows violin plots comparing the performance of
the two variants. The plots represent the distribution of the respective
results on each of the four sets of problems. One can see that {\tt cauchy}
is slightly more efficient in terms of iteration numbers
than {\tt plain} and that both variants require a similar number of BQP iterations.

The number of major iterations of our approach
is similar to or slightly less than the best NLP solver.
Figure~\ref{fig:slpec_vs_others} shows violin plots for the
largest problem instances (the results for the smaller instances are similar).
We note that one run of SNOPT reached the maximum iteration limit of 200.
Moreover, we have excluded four runs of MINOS of the set 40-psd,
which stopped at infeasible points. We note that several runs of MINOS
reported convergence to optimality but stopped with a first-order optimality
tolerance higher than $10^{-5}$. To provide a realistic impression of the
iteration counts, we have decided to include these runs in the plots.

\begin{figure}
  \centering
  \subfloat[][Case \texttt{40-ind}]
    {\includegraphics[scale=.4]{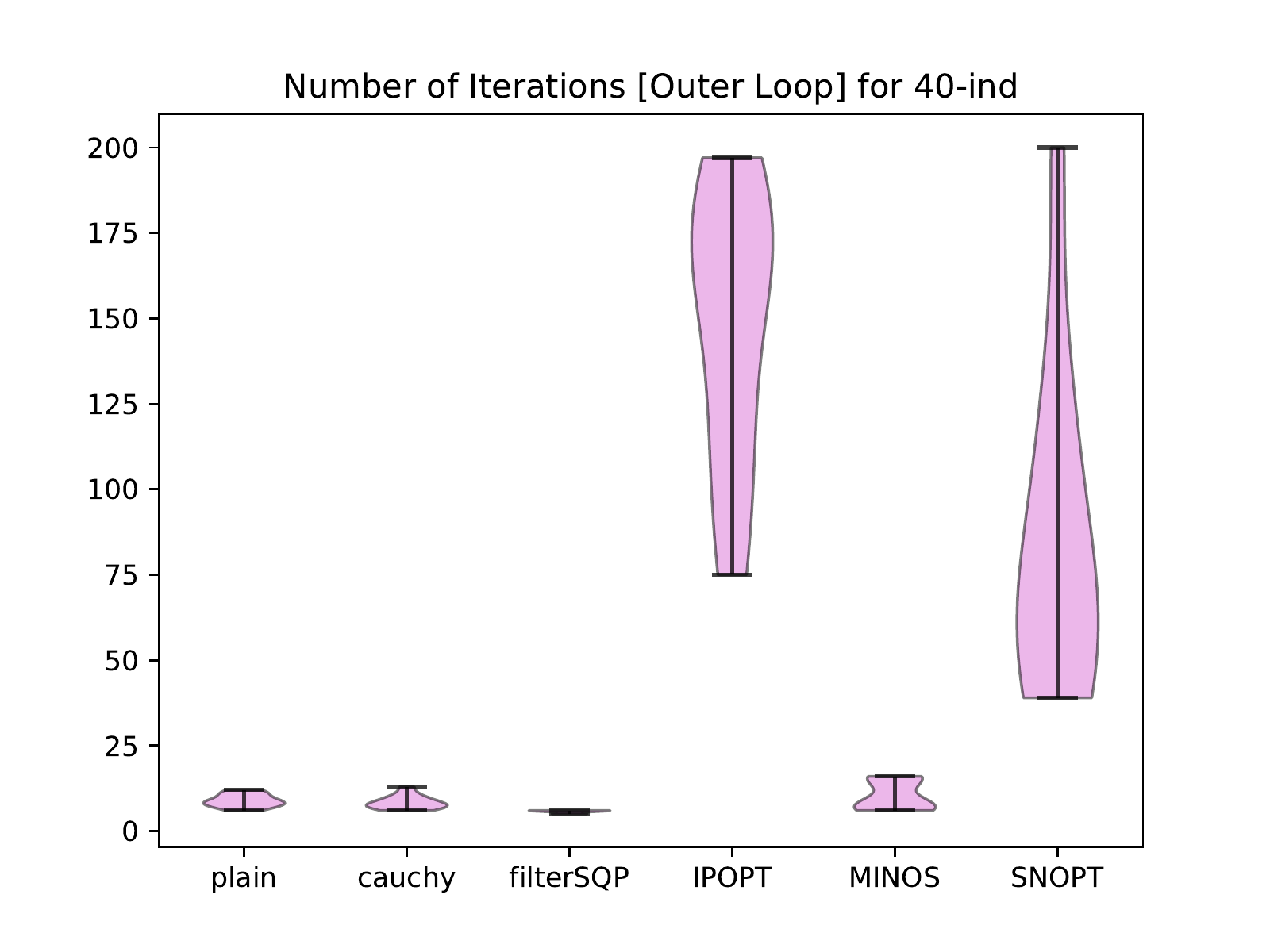}}
  \subfloat[][Case \texttt{40-psd}]
    {\includegraphics[scale=.4]{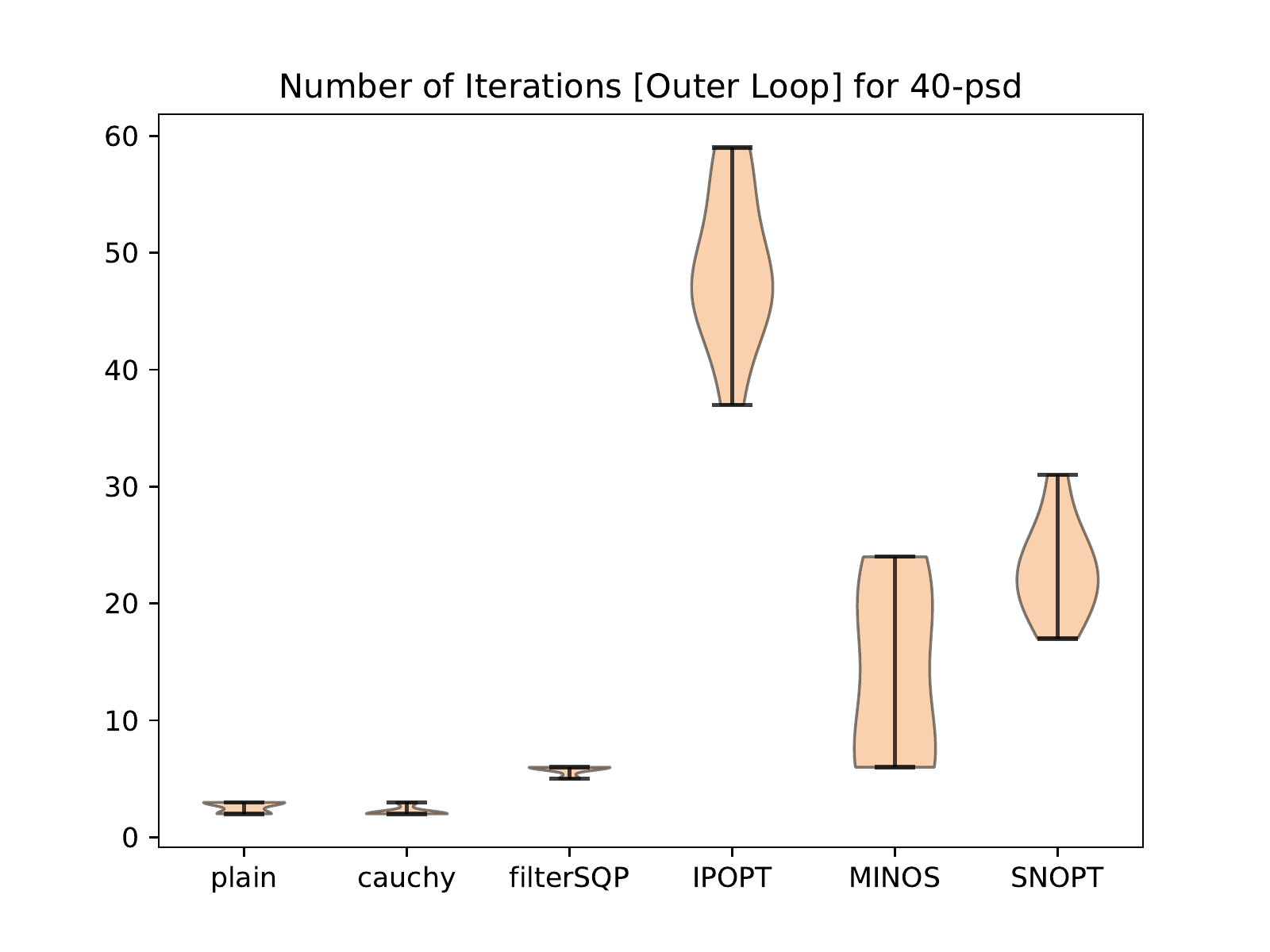}}   
  
  \caption{Number of outer iterations obtained with two variants of
  Algorithm~\ref{SLPCC} on test problem sets
  \texttt{20-ind}, \texttt{20-psd}, \texttt{40-ind}, \texttt{40-psd}
  compared with iterations of filterSQP, IPOPT, MINOS, and SNOPT.
  Both variants include BQP steps, the variant the variant \texttt{cauchy}
  does and the variant \texttt{plain} does not
  include the \texttt{SEARCH\_CAUCHY\_POINT} routine after the LPCC
  step.}\label{fig:slpec_vs_others}
\end{figure}

The objective values and iteration numbers achieved with our 
implementation are comparable to those of the NLP solvers while the running
times of our implementations are comparable to IPOPT but slower than
the running times of filterSQP, MINOS and SNOPT. Although not always
the case, the variant \texttt{cauchy} is often 
the slowest solver, but this may be attributed to the use of Python.

We provide detailed results in Appendix~\ref{sec:detailed_computational_results}. Specifically,
Table~\ref{tab:complete_results} provides results of our two
implementations on the quadratic problems. The rows of the table are
the test problem instances with the names introduced above. For each test
problem instance, the objective values for \texttt{plain} and
\texttt{cauchy} are given as well as the number of
outer, inner, and BQP iterations.
Table~\ref{tab:bqp_all_iterations} provides the major/outer
iteration counts of the NLP solvers filterSQP, IPOPT, MINOS, and SNOPT
and our implementations. Table~\ref{tab:bqp_all_running} provides the
running times of the NLP solvers and our implementations and
Table~\ref{tab:bqp_all_objective} provides the achieved objective values
for all solvers.

\subsection{General Nonlinear Test Problems}

We have also run our implementations \texttt{cauchy} and \texttt{plain} 
described above on twenty nonlinear test problems that are detailed in
Appendix~\ref{sec:nlpccs}. For the instances that are called
\texttt{20-fletcher0}, \texttt{20-fletcher1}, \texttt{40-fletcher0},
and \texttt{40-fletcher1} the reduced Hessian in the BQP subproblem is
nearly singular and has a condition number larger than $10^8$ at the
final iterate.

With the exception of the aforementioned degenerate instances,
our implementation of
Algorithm~\ref{SLPCC} always terminates with
an iterate that satisfies a first-order optimality tolerance 
of $10^{-6}$ or less using relatively few iterations,
approximately comparable to the test instances with 
quadratic objectives.
Note that a first-order tolerance of $10^{-6}$ is reached by
\texttt{plain} for \texttt{20-fletcher0}, and \texttt{20-fletcher1}
and by \texttt{cauchy} for
\texttt{20-fletcher1}, \texttt{40-fletcher0}, and \texttt{40-fletcher1}. 
The implementation terminates for the remaining instances because the
trust region contracts to zero at the final iterate
(our implementation stops after halving trust region radius 50 times),
with a first-order error of around $10^{-5}$, which we count as a failure
of our algorithm.

\begin{figure}
  \centering
  \includegraphics[scale=.7]{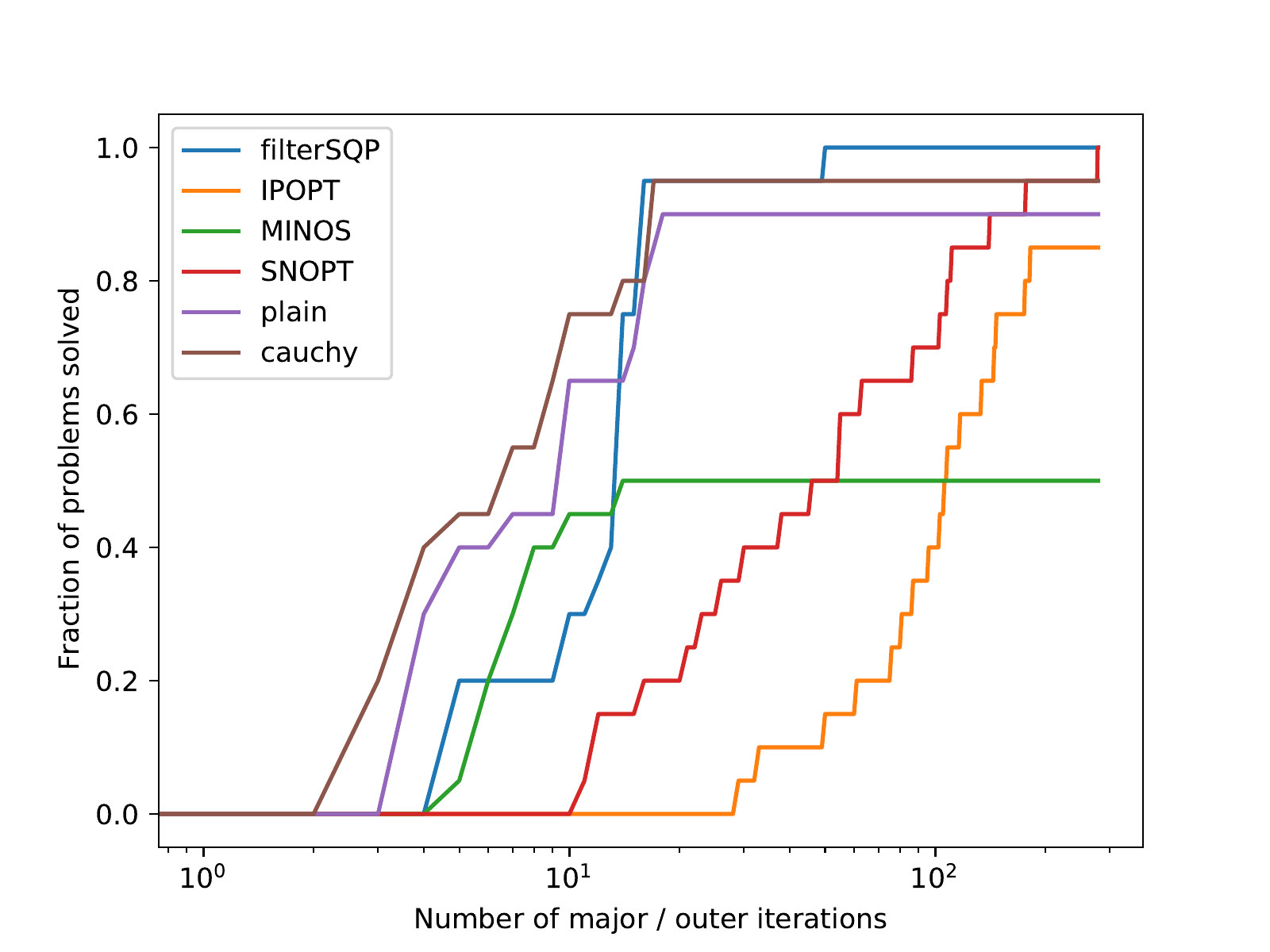}
  \caption{Number of major/outer iterations vs.~fraction
  of problems solved for two variants of Algorithm~\ref{SLPCC}
  on the nonlinear test problem set
  compared with filterSQP, IPOPT, MINOS, and SNOPT.
  Both variants include BQP steps, the variant the variant \texttt{cauchy}
  does and the variant \texttt{plain} does not
  include the \texttt{SEARCH\_CAUCHY\_POINT} routine after the LPCC step.}
  \label{fig:nlpcc_slpec_vs_others}
\end{figure}

We have run the same four NLP solvers on the general nonlinear
test problems. The NLP solvers filterSQP and MINOS require a similar amount of
outer iterations. The solvers SNOPT and IPOPT require significantly more 
iterations.
IPOPT does not find a solution of sufficient first-order optimality for the
degenerate instances within 3000 iterations.
The solver MINOS terminates at infeasible points in three runs. In seven 
further runs, it shows a first-order error of $10^{-5}$ or higher on 
termination. We give a performance profile for the nonlinear test cases
in Figure~\ref{fig:nlpcc_slpec_vs_others}, where we count the aforementioned
runs of IPOPT, MINOS, and our implementation as failures.
We observe that our implementations are competitive with the best NLP solvers
on the set of general nonlinear benchmark problems.
As for the quadratic test problems, the objective values
achieved with our implementation are comparable to those of the NLP solvers
while the running time of our implementation \texttt{plain}
is comparable to IPOPT but slower than
the running times by filterSQP, MINOS and SNOPT. Again, the variant \texttt{cauchy}
is often significantly slower than all other solvers.

We provide detailed results in Appendix~\ref{sec:detailed_computational_results}.
Specifically,
Table~\ref{tab:nlpcc_complete_results} provides detailed results of our two
implementations on the quadratic problems. The rows of the table are
the test problem instances with the names introduced above. For each test
problem instance, the objective values for \texttt{plain} and
\texttt{cauchy} computation are given as well as the number of
outer, inner, and BQP iterations.
Table~\ref{tab:nlpcc_all_iterations} provides the major/outer
iteration counts of the NLP solvers filterSQP, IPOPT, MINOS, and SNOPT
and our implementations. 
Table~\ref{tab:nlpcc_all_running} provides the
running times of the NLP solvers and our implementations and
Table~\ref{tab:nlpcc_all_obj} provides the achieved objective values
of the NLP solvers and our implementations.

\section{Conclusion and Extension}\label{S:generalization}
\setcounter{equation}{0}
We have introduced a new sequential LPCC algorithm for
bound-constrained MPCCs and shown that it  converges to a B-stationary point.
Such an approach can be used as a (nonsmooth)
subproblem solver for general MPCCs. Our approach is shown to be competitive
with state-of-the-art approaches for solving a collection of synthetic
benchmark problems. The outer and inner loops of the
benchmarked variants \texttt{plain} and \texttt{cauchy} of Algorithm~\ref{SLPCC}
have been implemented in Python; switching to a different language may result in some performance
improvements. We also believe that a more efficient implementation of the 
\texttt{SEARCH\_CAUCHY\_POINT} procedure would reduce the current gap in running times
between the \texttt{plain} and \texttt{cauchy} variants.

To test the algorithmic developments in the
context of our motivation, we have implemented
an augmented Lagrangian method based on~\cite[Section\,17.4]{nocedal2006no}
that solves subproblems of the
form~\eqref{eq:alsubproblem} using Algorithm~\ref{SLPCC}.
We test this approach on the \texttt{nash1} problem from \texttt{MacMPEC}.
Compared to an augmented Lagrangian method that treats the complementarities
as general nonlinear constraints, we observe a different qualitative behavior
on this test instance. In particular, the penalty parameter exhibits
a slower growth over the iterations
and a different sequence
of iterates is taken. Both methods converge to the same
strongly stationary point. In particular,
our approach reaches minimal values
for the constraint violation and stationarity
measure after $4$ iterations at a penalty
parameter value of $10^3$; 
the method that handles the complementarity constraint as a general
nonlinear constraint
reaches minimal values for
constraint violation and stationarity measure
after $8$ iterations at a penalty parameter value of $10^7$. See 
Appendix~\ref{sec:alintegration} for more detailed results and plots.

We note that it is straightforward to extend the developments of the preceding
sections to 
formulations of complementarity-constrained problems of the form
\begin{gather}\label{eq:generalized_cc}
\begin{aligned}
	\min_x \ & f(x) \\
	\text{s.t.\ } & \ell_0 \le x_0 \le u_0, \\
	& \ell_{1} \le x_1 \le u_1 \; \perp \; \ell_2 \le x_2 \le u_2,
\end{aligned}
\end{gather}
where for all $i \in \{1,\ldots,n_1\}$ exactly two of $\{\ell_{1,i}, u_{1,i}, \ell_{2,i}, u_{2,i} \}$ are finite.
This format allows more general mixed-complementarity expressions and mimics the definition of
complementarity constraints in AMPL~\cite{FerrFourGay:99}.
We tabulate two such types of
complementarity formulations and sketch their active sets
in Table~\ref{T:cc_types}.
Other general forms are easily derived by swapping components between $x_1$
and $x_2$, shifting bounds, or negating variables.
We note that we do not reformulate these complementarity constraints using slack variables, because 
such a reformulation would introduce additional linear constraints, making it harder
to apply our trust region algorithm.

\newcommand{\centered}[1]{\begin{tabular}{l} #1 \end{tabular}}
{
\begin{table}
\caption{Different types of complementarity constraints per
coordinate pair $(x_{1,i}, x_{2,i})$:
short notation, description of feasible set, 
and sketch of the feasible set. The coordinate index is omitted.}
\label{T:cc_types}
\begin{tabular}{c@{\hspace{2pt}}c@{\hspace{2pt}}c@{\hspace{2pt}}c}
\hline
& \revised{R2.12}{Short notation} & \changed{Feasible set} & \changed{Sketch}
\\ \hline\\
(a) & \centered{$\ell_1 \le x_1 \le u_1 \perp x_2$}
& \centered{$\ell_1 \le x_1 \le u_1$,\\ $x_2 \in \left\{
\begin{aligned}
[0,\infty) & : & x_1 = \ell_1 \\
\{0\} & : & x_1 \in (\ell_1,u_1) \\
(-\infty, 0] & : & x_1 = u_1
\end{aligned}
\right. $}
& \centered{\begin{tikzpicture}
\draw[dashed,gray] (0,1.5) -- (0,-1.5);
\draw[dashed,gray] (-0.5,0) -- (1.5,0);
\draw[dashed,gray] (1,1.5) -- (1,-1.5);
\draw (0,1) -- (0,0) -- (1,0) -- (1,-1);
\node[anchor=north] at (0,-1.5) {$\ell_1$};
\node[anchor=north] at (1,-1.5) {$u_1$};
\node[anchor=east] at (-0.5,0) {$0$};
\end{tikzpicture}}\\
(b) & \centered{$\ell_1 \le x_1 \perp x_2 \le u_2$}
& \centered{$\ell_1 \le x_1$, $x_2 \le u_2$,\\
$(x_1 - \ell_1)(u_2 - x_2) \le 0$}
& \centered{\begin{tikzpicture}
\draw[dashed,gray] (-0.5,0) -- (1.5,0);
\draw[dashed,gray] (0,.5) -- (0,-1.5);
\draw (0,-1) -- (0,0) -- (1,0);
\node[anchor=north] at (0,-1.5) {$\ell_1$};
\node[anchor=east] at (-0.5,0) {$u_2$};
\end{tikzpicture}}\\ \hline
\end{tabular}
\end{table}
}

The trust region subproblem corresponding to general LPCCs is given by
\[ \mbox{G-LPCC}(x,\Delta)
   \left\{ \begin{array}{ll} \dps
       \mini_d  & \nabla f(x)^T d \\
       \st      & l_0 \leq x_0 + d_0 \le u_0 \\
                & \ell_1 \le x_1 + d_1 \le u_1
       \; \perp \;
       \ell_2 \le x_2 + d_2 \le u_2, \\
                & \| d \|_{\infty} \leq \Delta ,
   \end{array} \right. 
\]
where $\ell_0$, $\ell_1$, $\ell_2$, $u_0$, $u_1$, and $u_2$ satisfy the conditions in
\eqref{eq:generalized_cc}.

It is a straightforward exercise to show that the trust region subproblem G-LPCC$(x,\Delta)$
can be solved as efficiently as~\eqref{eq:LPCC}, again by considering all possible
solutions for each index $i$ independently. The convergence results then follow
from the sufficient reduction condition, which is unaffected by the form of the
complementarity constraints.

\section*{Acknowledgments}
This work was supported by the U.S.~Department of Energy, Office of Science, Office of Advanced Scientific Computing Research, Scientific Discovery through
Advanced Computing (SciDAC) Program through the FASTMath Institute under Contract No. DE-AC02-06CH11357
C.~Kirches and P.~Manns acknowledge funding by Deutsche Forschungsgemeinschaft through Priority Programme 1962, grant Ki1839/1-2.
C.~Kirches has received funding from the German Federal Ministry of Education
and Research through grants 05M17MBA-MoPhaPro, 05M18MBA-MOReNet, 05M20MBA-LEO\-PLAN
and 01/S17089C-ODINE.
The authors thank two anonymous referees for their helpful suggestions and comments.

\FloatBarrier
\bibliographystyle{siamplain}
\bibliography{MPEC,NLP}

\appendix
\section{Description of Test Problems}
Here, we briefly describe the two classes of test problems that we used in our computational experiments.

The problems were generated in Matlab
and written out as AMPL model and data files.
All problem instances, as well as the Matlab routine used to generate them are available at
\url{https://wiki.mcs.anl.gov/leyffer/index.php/BndMPCC}.

\subsection{Quadratic MPCCs}\label{sec:qpccs}
The quadratic test problems are of the form
\begin{gather}\label{E:qpcc} 
   \begin{aligned}
     \mini_{x}\ 	& \frac{1}{2} x^T H x + g^T x		\\
     \st\ & \ell_0 \leq x_0 \leq u_0, \\
		& 0 \leq x_1 \; \perp \; x_2 \geq 0,
   \end{aligned}
\end{gather}
where $x_0 \in \R^{n_0}$, $x_1, x_2 \in \R^{n_1}$, $H$ is a symmetric sparse matrix with density $(n_0+2n_1)^2/4$
whose entries are normally distributed, and $g$ is a vector in $\R^{n_0+2n_1}$ whose components are
uniform random numbers in the range $[-10,10]$. The bounds $l_0, u_0$ are 
uniform random numbers in the range $[-10,10]$ and $[0,20]$, and we ensure that $l_{0,i} < u_{0,i}$.
We round all data to four digits, because we have observed that this makes the problems harder
to solve. In addition, we believe that real-life problems are not typically described in terms of
double precision data.

\subsection{General Nonlinear MPCCs}\label{sec:nlpccs}

We have also curated a set of nonquadratic test problems of the form~\eqref{E:mpcc} by adding bounds
and complementarity constraints to some well-known nonlinear test problems. For each nonlinear function,
we created two sets of instances by varying the indices in the complementarity constraints. In all
cases, $n=n_0+2n_1$, with $n_0=n_1=20$ or $n_0=n_1=40$, is used in our experiments.
All functions are taken from~\cite{andrei2008unconstrained}.

\begin{description}
\item[fletcher] 
  $\dps f(x) = \sum_{i=1}^{n-1} 100 \big(x_{i+1} - x_{i} + 1 - x_{i}^2\big)^2$
\item[himmelblau] 
  $\dps f(x) = \sum_{i=1}^{n/2} \big( (x_{2 i-1} + x_{2 i} - 11)^2
                               + (x_{2 i-1} + x_{2 i}^2 - 7)^2 \big) $
\item[mccormick] 
  $\dps f(x) = \sum_{i =1}^{n-1} \big(-1.5 x_{i} + 2.5 x_{i+1} + 1 + (x_{i} - x_{i+1})^2 + \sin(x_{i}+x_{i+1})\big) $
\item[powell] \phantom{+} 
~

  $\dps \hspace*{-1cm}
  f(x) = \sum_{i =1}^{n/4} \big( (x_{4 i-3} + 10 x_{4 i-2})^2
                               + 5 (x_{4 i-1} - x_{4 i})^2
                               + (x_{4 i-2} - 2 x_{4 i-1})^4
                               + 10 (x_{4 i-3} - x_{4 i})^4\big)$
\item[rosenbrock] 
  $\dps f(x) = \sum_{i =1}^{n-1} \big( 100 (x_{i+1}-x_{i}^2)^2 + (1-x_{i})^2  \big)$
\end{description}
For each function, we generated two instance classes with different complementarity constraints:
\begin{eqnarray*}
  \text{Class 0:} \quad 0 \leq x_{i} \; & \perp & \; x_{n_1+i} \geq 0, \quad \forall i \in 1,\ldots,n_1 \\
  \text{Class 1:} \quad 0 \leq x_{2i-1} \; & \perp & \; x_{2i} \geq 0, \quad \forall i \in 1,\ldots,n_1.
\end{eqnarray*}
For all nonlinear test problem instances\changed{,} the lower bound \revised{R2.13}{$l_0$} on $x_0$ was set to zero in all coordinates. The upper bound 
\revised{R2.13}{$u_0$} on $x_0$,
was set to $10^8$ in all coordinates.
\clearpage

\section{Detailed Computational Results}\label{sec:detailed_computational_results}

\begin{table}[H]
  \centering
\caption{Performance for Algorithm~
\ref{SLPCC} with Algorithm~\ref{SLPCC_BQP},
and with (\texttt{cauchy})
and without (\text{plain})
Algorithm~\ref{Cauchy}
on 40 problem instances. (All first-order optimality measures are less than $10^{-13}$.) \changed{For each problem instance, boldface numbers indicate the algorithm that delivered the best (lowest) objective values and (total) iteration counts.}}\label{tab:complete_results}
  \scriptsize
\begin{tabular}{lrrrrrrrr}
\hline
\multicolumn{1}{|l|}{\multirow{2}{*}{Prob-Inst.}}
& \multicolumn{2}{c|}{Obj.\ val.}
& \multicolumn{2}{c|}{outer iters}
& \multicolumn{2}{c|}{\revised{R.14}{total} inner iters}
& \multicolumn{2}{c|}{BQP iters}\\
\cline{2-9} 
\multicolumn{1}{|l|}{} & 
\multicolumn{1}{c|}{plain} & \multicolumn{1}{c|}{cauchy} &
\multicolumn{1}{c|}{plain} & \multicolumn{1}{c|}{cauchy} &
\multicolumn{1}{c|}{plain} & \multicolumn{1}{c|}{cauchy} &
\multicolumn{1}{c|}{plain} & \multicolumn{1}{c|}{cauchy}\\
  \hline
20-ind-0 & -6459.47 & -6459.47 & 10 & \textbf{8} & 46 & \textbf{8} & 28 & \textbf{24} \\
20-ind-1 & \textbf{-7373.48} & -7135.19 & 14 & \textbf{12} & 28 & \textbf{14} & 30 & \textbf{26} \\
20-ind-2 & \textbf{-2618.47 }& -2538.46 & 6 & 6 & 8 & \textbf{6} & 13 & 13 \\
20-ind-3 &\textbf{ -3115.40} & -3115.40 & 8 & \textbf{7} & 24 & \textbf{7} & 18 & \textbf{17} \\
20-ind-4 & \textbf{-6272.72 }& -6219.23 & 8 & 8 & 12 & \textbf{8} & \textbf{16} & 18 \\
20-ind-5 & \textbf{-2829.78 }& -2748.66 & 6 & \textbf{7} & 8 & \textbf{7} & \textbf{13} & 16 \\
20-ind-6 & -8374.58 & -8374.58 & 11 & \textbf{10 }& 31 & \textbf{10} & 22 & \textbf{20} \\
20-ind-7 & -1588.59 & -1588.59 & 6 & 6 & 20 & \textbf{6} & 14 & \textbf{13} \\
20-ind-8 & -2045.32 & -2045.32 & 8 & \textbf{7} & 30 & \textbf{7} & 17 & \textbf{15} \\
20-ind-9 & -5622.22 & -5622.22 & 6 & 6 & 6 & 6 & 13 & 13 \\
\hline
20-psd-0 & 1147.55 & 1147.55 & 2 & 2 & 12 & \textbf{4} & 4 & 4 \\
20-psd-1 & 1043.67 & 1043.67 & 2 & 2 & 16 & \textbf{2} & 4 & 4 \\
20-psd-2 & 1772.49 & 1772.49 & 2 & 2 & 14 & \textbf{4} & 4 & 4 \\
20-psd-3 & 566.63 & 566.63 & 1 & 1 & 1 & 1 & 2 & 2 \\
20-psd-4 & \textbf{896.98  }&898.00 & 2 & 2 & 18 & \textbf{6} & 4 & 4 \\
20-psd-5 & 1507.17 & 1507.17 & 2 & 2 & 12 & \textbf{2} & 4 & 4 \\
20-psd-6 & 751.21 & 751.21 & 4 & \textbf{3} & 36 & \textbf{3} & 8 & \textbf{6} \\
20-psd-7 & 2050.00 & 2050.00 & 2 & 2 & 16 & \textbf{8} & 4 & 4 \\
20-psd-8 & 1109.89 & 1109.89 & 3 & \textbf{2} & 25 & \textbf{2} & 7 & \textbf{6} \\
20-psd-9 & 1090.73 & 1090.73 & 2 & 2 & 14 & \textbf{4} & 5 & 5 \\
\hline
40-ind-0 & \textbf{-7999.59 }& -7997.27 & \textbf{6} & 7 & \textbf{6} & 7 & \textbf{12} & 15 \\
40-ind-1 & -17702.74 & -17702.74 & 11 & \textbf{10} & 45 & \textbf{10} & 25 & \textbf{22} \\
40-ind-2 & \textbf{-23716.62} & -23462.24 & 8 & \textbf{7} & 16 & \textbf{7} & 16 & 16 \\
40-ind-3 & \textbf{-7710.66 }& -7589.43 & 7 & \textbf{6} & 17 & \textbf{6} & 14 & \textbf{12} \\
40-ind-4 & -13571.15 & \textbf{-17301.43} & \textbf{12} & 13 & 46 & \textbf{13} & \textbf{24} & 26 \\
40-ind-5 & -10395.51 & -10395.51 & 9 & \textbf{8} & 31 & \textbf{8} & 19 & \textbf{16} \\
40-ind-6 & -4889.03 & \textbf{-4890.66 }& 8 & \textbf{7} & 40 & \textbf{7} & 18 & \textbf{15} \\
40-ind-7 & \textbf{-17301.07} & -17147.55 & \textbf{9} & 10 & 19 & \textbf{10} & \textbf{18} & 21 \\
40-ind-8 & -14414.19 & \textbf{-14198.29} & 8 & 8 & 14 & \textbf{10} & \textbf{18} & 21 \\
40-ind-9 & -10073.14 & -10073.14 & 11 & \textbf{8} & 69 & \textbf{8} & 22 & \textbf{16} \\
\hline
40-psd-0 & 4009.29 & 4009.29 & 3 & 3 & 23 & \textbf{7} & 6 & 6 \\
40-psd-1 & 4147.27 & 4147.27 & 3 & \textbf{2} & 29 & \textbf{2} & 7 & \textbf{4} \\
40-psd-2 & \textbf{3105.85} & 3116.39 & 2 & 2 & 10 & \textbf{2} & \textbf{5} & 6 \\
40-psd-3 & 4944.42 & 4944.42 & 3 & \textbf{2} & 35 & \textbf{2} & 8 & \textbf{5} \\
40-psd-4 & 2452.84 & 2452.84 & 3 & \textbf{2} & 25 & \textbf{2} & 6 & 6 \\
40-psd-5 & 2365.25 & 2365.25 & 3 & \textbf{2} & 29 & \textbf{2} & 7 & \textbf{6} \\
40-psd-6 & 4035.26 & 4035.26 & 2 & 2 & 12 & \textbf{4} & \textbf{4} & 5 \\
40-psd-7 & 3154.05 & 3154.05 & \textbf{2} & 3 & 12 & \textbf{11} & 6 & 6 \\
40-psd-8 & 2220.87 & 2220.87 & 2 & 2 & 16 & \textbf{8} & 4 & 4 \\
40-psd-9 & 3657.14 & 3657.14 & 3 & \textbf{2} & 25 & \textbf{2} & 7 & \textbf{4} \\
\hline
\end{tabular}
\end{table}

\clearpage

\begin{table}[htb]
  \centering
  \caption{Number of major/outer iterations for solving random instances of bound-constrained MPCCs
(INF = Termination at infeasible point, F = Termination with violated
  optimality tolerance).}
  \label{tab:bqp_all_iterations}
  \scriptsize  
\begin{tabular}{lrrrrrr}  
\hline
\multicolumn{1}{|l|}{\multirow{2}{*}{Prob-Inst.}}
& \multicolumn{6}{c|}{Number of major/outer iterations}
\\
\cline{2-7} 
\multicolumn{1}{|l|}{} 
& \multicolumn{1}{c|}{plain}
& \multicolumn{1}{c|}{cauchy}
& \multicolumn{1}{c|}{filterSQP}
& \multicolumn{1}{c|}{IPOPT}
& \multicolumn{1}{c|}{MINOS}
& \multicolumn{1}{c|}{SNOPT}\\ \hline
    20-ind-0 & 10 & 8 & 6 & 108 & F\;\phantom{0}7 & 41\\
    20-ind-1 & 14 & 12  & 4 & 113 & 15 & 31\\
    20-ind-2 & 6 & 6 & 5 & 79 & 7 & 25\\
    20-ind-3 & 8 & 7 & 5 & 99	& 5 & 76\\
    20-ind-4 & 8 & 8 & 6 & 108 & F\;20 & 68\\
    20-ind-5 & 6 & 7 & 6 & 84 & 5 & 30\\
    20-ind-6 & 11 & 10 & 6 & 91 & 6 & 36\\
    20-ind-7 & 6 & 6 & 5 & 51 & 6 & 34\\
    20-ind-8 & 8 & 7 & 6 & 92	& 7 & 33\\
    20-ind-9 & 6 & 6 & 5 & 91 & 7 & 41\\ \hline
    20-psd-0 & 2 & 2 & 6 & 43 & F\;18 & 15\\
    20-psd-1 & 2 & 2 & 5 & 50 & 6 & 18\\
    20-psd-2 & 2 & 2 & 5 & 39 & 6 & 15\\
	20-psd-3 & 1 & 1 & 5 & 40 & F\;19 & 18\\
	20-psd-4 & 2 & 2 & 5 & 43 & F\;18 & 16\\
	20-psd-5 & 2 & 2 & 6 & 38 & F\;18 & 17\\
	20-psd-6 & 4 & 3 & 5 & 42	& 7 & 19\\
	20-psd-7 & 2 & 2 & 5 & 42	& 16 & 15\\
	20-psd-8 & 3 & 2 & 5 & 42 & 14 & 19\\
	20-psd-9 & 2 & 2 & 6 & 48	& 7 & 30 \\ \hline
    40-ind-0 & 6 & 7 & 6 & 132 & 7 & F\;200 \\
    40-ind-1 & 11 & 10 & 5 & 163 & 6 & 64 \\
    40-ind-2 & 8 & 7 & 6 & 96 & F\;15 & 47 \\
    40-ind-3 & 7 & 6 & 6 & 175 & 8 & 54 \\
    40-ind-4 & 12 & 13 & 6 & 178 & 7 & 118 \\
    40-ind-5 & 9 & 8 & 5 & 197 & F\;16 & 47 \\
    40-ind-6 & 8 & 7 & 6 & 75 & F\;\phantom{0}7 & 39 \\
    40-ind-7 & 9 & 10 & 6 & 182 & F\;16 & 78  \\
    40-ind-8 & 8 & 8 & 6 & 167 & F\;15 & 122 \\
    40-ind-9 & 11 & 8 & 5 & 117 & 7 & 74\\ \hline
    40-psd-0 & 3 & 3 & 6 & 49 & 19 & 27 \\
    40-psd-1 & 3 & 2 & 6 & 59 & 7 & 20 \\
    40-psd-2 & 2 & 2 & 6 & 43 & 24 & 23 \\
    40-psd-3 & 3 & 2 & 5 & 46	& 6 & 22 \\
    40-psd-4 & 3 & 2 & 6 & 54 & INF & 19 \\
    40-psd-5 & 3 & 2 & 6 & 48	& INF & 31 \\
    40-psd-6 & 2 & 2 & 6 & 46 & 20 & 17 \\
    40-psd-7 & 2 & 3 & 6 & 47 & INF & 25 \\
    40-psd-8 & 2 & 2 & 5 & 57 & INF & 21 \\
    40-psd-9 & 3 & 2 & 6 & 37 & 7 & 24 \\ \hline
  \end{tabular}
\end{table}

\clearpage

\begin{table}[H]
  \centering
\caption{Run times (seconds) for Algorithm~
\ref{SLPCC} with Algorithm~\ref{SLPCC_BQP},
and with (\texttt{cauchy})
and without (\text{plain})
Algorithm~\ref{Cauchy}
as well as the NLP solvers
on 40 problem instances (INF = Termination at infeasible point, F = Termination with violated optimality tolerance).}\label{tab:bqp_all_running}
  \scriptsize
\begin{tabular}{lSSSSSS}
\hline
\multicolumn{1}{|l|}{\multirow{2}{*}{Prob-Inst.}}
& \multicolumn{6}{c|}{Run times (seconds)}
\\
\cline{2-7}
\multicolumn{1}{|l|}{} & 
\multicolumn{1}{c|}{plain} & \multicolumn{1}{c|}{cauchy} &
\multicolumn{1}{c|}{filterSQP} & \multicolumn{1}{c|}{IPOPT} &
\multicolumn{1}{c|}{MINOS} & \multicolumn{1}{c|}{SNOPT}\\
  \hline
20-ind-0 &  0.039 & 0.098
         &  0.0053 
         &  0.066 
         &  F\;0.01
         &  0.03 \\
20-ind-1 & 0.046 & 0.12
         & 0.0027
         & 0.073
         & 0.01	
         & 0.02 \\
20-ind-2 & 0.019 & 0.059
         & 0.0025
         & 0.040
         & 0.01
         & 0.01 \\
20-ind-3 & 0.028 & 0.076
         & 0.0024
         & 0.058
         & 0.01	
         & 0.03 \\
20-ind-4 & 0.024 & 0.092
         & 0.0030
         & 0.063 
         & F\;0.01
         & 0.04 \\
20-ind-5 & 0.019 & 0.083
         & 0.0037
         & 0.041
         & 0.01	
         & 0.02 \\
20-ind-6 & 0.039 & 0.12
         & 0.0036
         & 0.055
         & 0.01	
         & 0.03 \\
20-ind-7 & 0.021 & 0.059
         & 0.0027
         & 0.028
         & 0.01	
         & 0.02 \\
20-ind-8 & 0.030 & 0.066
         & 0.0035
         & 0.059
         & 0.01	
         & 0.02 \\
20-ind-9 & 0.018 & 0.093
         & 0.0032 
         & 0.050
         & 0.01	
         & 0.03 \\
\hline
20-psd-0 & 0.0093 & 0.020
         & 0.0030
         & 0.027
         & F\;0.01	
         & 0.01 \\
20-psd-1 & 0.010 & 0.026
         & 0.0034
         & 0.029
         & 0.01	
         & 0.01 \\
20-psd-2 & 0.0098 & 0.027
         & 0.0047
         & 0.019
         & 0.01
         & 0.01 \\
20-psd-3 & 0.0036 & 0.016
         & 0.0033
         & 0.011
         & F\;0.01
         & 0.01\\
20-psd-4 & 0.010 & 0.021
         & 0.0031
         & 0.024
         & F\;0.03
         & 0.01 \\
20-psd-5 & 0.0092 & 0.021
         & 0.0037
         & 0.024
         & F\;0.01
         & 0.01 \\
20-psd-6 & 0.022 & 0.041
         & 0.0032
         & 0.024
         & 0.01	
         & 0.01 \\
20-psd-7 & 0.0098 & 0.026
         & 0.0029
         & 0.022
         & 0.01	
         & 0.01 \\
20-psd-8 & 0.016 & 0.023
         & 0.0033
         & 0.027
         & 0.02
         & 0.01 \\
20-psd-9 & 0.010 & 0.029
         & 0.0044
         & 0.029
         & 0.01
         & 0.02 \\
\hline
40-ind-0 & 0.058 & 0.24
         & 0.016
         & 0.23
         & 0.03
         & F\;0.67 \\
40-ind-1 & 0.12 & 0.38
         & 0.011
         & 0.278
         & 0.02
         & 0.04 \\
40-ind-2 & 0.082 & 0.41
         & 0.016
         & 0.18
         & F\;0.03
         & 0.03 \\
40-ind-3 & 0.075 & 0.25
         & 0.014
         & 0.32
         & 0.03
         & 0.04\\
40-ind-4 & 0.13 & 0.40
         & 0.016
         & 0.34
         & 0.04
         & 0.08 \\
40-ind-5 & 0.12 & 0.35
         & 0.012
         & 0.36
         & F\;0.03
         & 0.03 \\
40-ind-6 & 0.087 & 0.27
         & 0.013
         & 0.12
         & F\;0.03
         & 0.03 \\
40-ind-7 & 0.093 & 0.39
         & 0.015
         & 0.33
         & F\;0.05
         & 0.05 \\
40-ind-8 & 0.082 & 0.34
         & 0.014
         & 0.305
         & F\;0.04	
         & 0.08 \\
40-ind-9 & 0.13 & 0.38
         & 0.013
         & 0.20
         & 0.03
         & 0.05 \\
\hline
40-psd-0 & 0.038 & 0.10
         & 0.016
         & 0.069
         & 0.05
         & 0.02\\
40-psd-1 & 0.040 & 0.080
         & 0.016
         & 0.087
         & 0.03
         & 0.02 \\
40-psd-2 & 0.024 & 0.081
         & 0.015
         & 0.068
         & 0.18
         & 0.02 \\
40-psd-3 & 0.043 & 0.093
         & 0.013
         & 0.059
         & 0.02
         & 0.02 \\
40-psd-4 & 0.040 & 0.077
         & 0.017
         & 0.085
         & INF
         & 0.02 \\
40-psd-5 & 0.041 & 0.11
         & 0.019
         & 0.062
         & INF
         & 0.03 \\
40-psd-6 & 0.023 & 0.081
         & 0.014
         & 0.061
         & 0.05
         & 0.01 \\
40-psd-7 & 0.025 & 0.13
         & 0.018
         & 0.060
         & INF	
         & 0.02  \\
40-psd-8 & 0.026 & 0.11
         & 0.015
         & 0.070
         & INF
         & 0.02 \\
40-psd-9 & 0.038 & 0.086
		 & 0.016
         & 0.053
         & 0.03
         & 0.02 \\
\hline
\end{tabular}
\end{table}

\clearpage

\begin{table}[H]
  \centering
\caption{Achieved objective values for Algorithm~
\ref{SLPCC} with Algorithm~\ref{SLPCC_BQP},
and with (\texttt{cauchy})
and without (\text{plain})
Algorithm~\ref{Cauchy}
as well as the NLP solvers
on 40 problem instances (INF = Termination at infeasible point, F = Termination with violated optimality tolerance).
\changed{For each problem instance, the lowest achieved objective values are printed in boldface.}}
\label{tab:bqp_all_objective}
  \scriptsize
\begin{tabular}{lrrrrrr}
\hline
\multicolumn{1}{|l|}{\multirow{2}{*}{Prob-Inst.}}
& \multicolumn{6}{c|}{Achieved objective value}
\\
\cline{2-7} 
\multicolumn{1}{|l|}{} & 
\multicolumn{1}{c|}{plain} & \multicolumn{1}{c|}{cauchy} &
\multicolumn{1}{c|}{filterSQP} & \multicolumn{1}{c|}{IPOPT} &
\multicolumn{1}{c|}{MINOS} & \multicolumn{1}{c|}{SNOPT}\\
  \hline
20-ind-0 & -6459.47 & -6459.47
         & \textbf{-7582.63}
         & -7550.90
         & F\;\textbf{-7582.63}
         & -6459.47 \\
20-ind-1 & \textbf{-7373.48} & -7135.19 
         & \textbf{-7373.48}
         & -7135.19
         & -7334.84
         & -7336.22 \\
20-ind-2 & -2618.47 & -2538.46
         & -2618.47
         & -2627.96
         & \textbf{-3098.27}
         & -2594.18 \\
20-ind-3 & \textbf{-3115.40} & \textbf{-3115.40}
         & -3038.25
         & -2953.03
         & -2946.15
         & \textbf{-3115.40} \\
20-ind-4 & -6272.72 & -6219.23
         & \textbf{-8735.77}
         & -7648.74
         & F\;-8370.13
         & -7819.34 \\
20-ind-5 & \textbf{-2829.78} & -2748.66
         & -2623.97
         & -2108.59
         & -2131.34
         & -2116.81 \\
20-ind-6 & \textbf{-8374.58} & \textbf{-8374.58} 
         & -8172.41
         & \textbf{-8374.58}
         & \textbf{-8374.58}
         & \textbf{-8374.58} \\
20-ind-7 & -1588.59 & -1588.59 
         & \textbf{-1923.79}
         & -1589.24
         & -1800.37
         & -1588.59 \\
20-ind-8 & \textbf{-2045.32} & \textbf{-2045.32} 
         & -1969.94
         & -1970.99
         & -1963.62
         & -1969.94 \\
20-ind-9 & \textbf{-5622.22} & \textbf{-5622.22}
         & -3245.52
         & -5399.08
         & -3098.11
         & \textbf{-5622.22} \\
\hline
20-psd-0 & \textbf{1147.55} & \textbf{1147.55}
         & 1147.69
         & \textbf{1147.55}
         & F\;1147.88
         & \textbf{1147.55} \\
20-psd-1 & 1043.67 & 1043.67
         & 1043.67
         & 1043.67
         & 1043.67
         & 1043.67 \\
20-psd-2 & 1772.49 & 1772.49
         & 1772.49
         & 1772.49
         & 1772.49
         & 1772.49 \\
20-psd-3 & 566.63 & 566.63 
         & \textbf{564.89}
         & \textbf{564.89}
         & F\;565.63
         & \textbf{564.89} \\
20-psd-4 & \textbf{896.98} & 898.00
         & \textbf{896.98}
         & \textbf{896.98}
         & F\;897.95
         & \textbf{896.98} \\
20-psd-5 & 1507.17 & 1507.17
         & \textbf{1487.97}
         & \textbf{1487.97}
         & F\;1488.45
         & 1489.83 \\
20-psd-6 & 751.21 & 751.21 
         & \textbf{725.97}
         & \textbf{725.97}
         & \textbf{725.97}
         & \textbf{725.97} \\
20-psd-7 & 2050.00 & 2050.00 
         & 2050.00
         & 2050.00
         & 2050.00
         & 2050.00 \\
20-psd-8 & 1109.89 & 1109.89
         & 1109.89
         & 1109.89
         & 1109.89
         & 1109.89 \\
20-psd-9 & 1090.73 & 1090.73 
         & \textbf{1075.17}
         & 1075.23
         & \textbf{1075.17}
         & 1081.99 \\
\hline
40-ind-0 & -7999.59 & -7997.27
         & \textbf{-8254.03}
         & -8089.13
         & -8197.27
         & F\;-8137.61 \\
40-ind-1 & -17702.74 & -17702.74 
         & -17891.48
         & -17702.74
         & \textbf{-17970.51}
         & -17702.74 \\
40-ind-2 & \textbf{-23716.62} & -23462.24 
         & -23315.13
         & -22970.72
         & F\;-22994.94
         & -23117.33 \\
40-ind-3 & -7710.66 & -7589.43 
         & \textbf{-7866.74}
         & -7589.43
         & -7735.33
         & -7800.41 \\
40-ind-4 & -13571.15 & -17301.43 
         & -14451.51
         & -14746.28
         & -14746.28
         & \textbf{-16984.42} \\
40-ind-5 & -10395.51 & -10395.51 
         & -10668.66
         & -10668.66
         & F\;-10456.41
         & \textbf{-10672.87} \\
40-ind-6 & -4889.03 & -4890.66 
         & -4851.44
         & -4728.64
         & F\;-4767.99
         & \textbf{-5073.21} \\
40-ind-7 & -17301.07 & -17147.55 
         & \textbf{-20378.55}
         & -17361.50
         & F\;-17412.02
         & -17835.51 \\
40-ind-8 & -14414.19 & -14198.29
         & \textbf{-15355.83}
         & -14339.09
         & F\;-13804.13
         & -14309.26 \\
40-ind-9 & -10073.14 & -10073.14 
         & \textbf{-10848.63}
         & -10732.42
         & \textbf{-10848.63}
         & -10633.96 \\
\hline
40-psd-0 & 4009.29 & 4009.29 
         & \textbf{4002.46}
         & 4002.77
         & 4005.18
         & 4009.60 \\
40-psd-1 & 4147.27 & 4147.27 
         & \textbf{4130.51}
         & \textbf{4130.51}
         & \textbf{4130.51}
         & 4147.27 \\
40-psd-2 & \textbf{3105.85} & 3116.39
         & \textbf{3105.85}
         & \textbf{3105.85}
         & 3106.95
         & \textbf{3105.85} \\
40-psd-3 & 4944.42 & 4944.42 
         & \textbf{4941.11}
         & \textbf{4941.11}
         & \textbf{4941.11}
         & \textbf{4941.11} \\
40-psd-4 & 2452.84 & 2452.84 
         & \textbf{2446.08}
         & 2450.67
         & INF
         & 2452.84 \\
40-psd-5 & 2365.25 & 2365.25 
         & \textbf{2364.77}
         & \textbf{2364.77}
         & INF
         & 2365.25 \\
40-psd-6 & 4035.26 & 4035.26 
         & \textbf{4035.23}
         & \textbf{4035.23}
         & \textbf{4035.23}
         & 4035.26 \\
40-psd-7 & 3154.05 & 3154.05 
         & \textbf{3152.99}
         & \textbf{3152.99}
         & INF
         & 3154.05 \\
40-psd-8 & 2220.87 & 2220.87 
         & 2220.35
         & 2222.91
         & INF
         & \textbf{2220.32} \\
40-psd-9 & \textbf{3657.14} & \textbf{3657.14} 
         & \textbf{3657.14}
         & \textbf{3657.14}
         & \textbf{3657.14}
         & 3657.88 \\
\hline
\end{tabular}
\end{table}

\clearpage

\begin{table}[H]
  \centering
\caption{Performance for Algorithm~
\ref{SLPCC} with Algorithm~\ref{SLPCC_BQP},
and with (\texttt{cauchy})
and without (\text{plain})
Algorithm~\ref{Cauchy}
on nonlinear test instances.
\changed{For each problem instance, the lowest (total) iteration numbers are printed in boldface.}}
\label{tab:nlpcc_complete_results}
  \scriptsize
\setlength{\tabcolsep}{4pt}\begin{tabular}{lSSrrrrrr}
\hline
\multicolumn{1}{|l|}{\multirow{2}{*}{Prob-Inst.}}
& \multicolumn{2}{c|}{Obj.\ val.}
& \multicolumn{2}{c|}{outer iters}
& \multicolumn{2}{c|}{\revised{R2.14}{total} inner iters}
& \multicolumn{2}{c|}{BQP iters}\\
\cline{2-9} 
\multicolumn{1}{|l|}{} & 
\multicolumn{1}{c|}{plain} & \multicolumn{1}{c|}{cauchy} &
\multicolumn{1}{c|}{plain} & \multicolumn{1}{c|}{cauchy} &
\multicolumn{1}{c|}{plain} & \multicolumn{1}{c|}{cauchy} &
\multicolumn{1}{c|}{plain} & \multicolumn{1}{c|}{cauchy}\\
  \hline
20-fletcher0 & 2246.86 & 2246.86 & \textbf{6} & 8 & \textbf{184} & 366 & \textbf{16} & 17 \\
20-fletcher1 & 4046.86 & 4046.86 & 14 & \textbf{13} & 228 & \textbf{129} & 44 & \textbf{40} \\
20-himmelblau0 & 826.29 & 826.29 & 9 & 9 & 93 & \textbf{49} & 9 & 9 \\
20-himmelblau1 & 230.31 & 230.31 & 9 & \textbf{8} & 93 & \textbf{58} & 9 & \textbf{8} \\
20-mccormick0 & 58.93 & 58.93 & 3 & 3 & \textbf{21} & 23 & 3 & \textbf{2} \\
20-mccormick1 & 58.93 & 58.93 & 4 & \textbf{3} & 60 & \textbf{23} & 4 & \textbf{2} \\
20-powell0 
& \num{9.6e-10} 
& \num{7.5e-10}
& 17 & \textbf{16} & 449 & \textbf{364} & 37 & \textbf{35} \\
20-powell1 
& \num{3.0e-9}
& \num{1.1e-18}
& 15 & \textbf{6} & 441 & \textbf{12} & 28 & \textbf{7} \\
20-rosenbrock0 & 58.61 & 58.61 & 3 & \textbf{2} & 81 & \textbf{40} & 4 & \textbf{2} \\
20-rosenbrock1 & 58.61 & 58.61 & 3 & \textbf{2} & 83 & \textbf{40} & 4 & \textbf{2} \\
\hline
40-fletcher0 & 4246.86 & 4246.86 & 9 & \textbf{4} & 445 & \textbf{72} & 18 & \textbf{8} \\
40-fletcher1 & 8046.86 & 8046.86 & 18 & \textbf{16} & 512 & \textbf{222} & 48 & \textbf{39} \\
40-himmelblau0 & 1652.58 & 1652.58 & 9 & 9 & 93 & \textbf{49} & 10 & \textbf{9} \\
40-himmelblau1 & 460.63 & 460.63 & 9 & \textbf{8} & 91 & \textbf{58} & 9 & \textbf{8} \\
40-mccormick0 & 118.93 & 118.93 & 3 & 3 & \textbf{21} & 23 & 3 & \textbf{2} \\
40-mccormick1 & 118.93 & 118.93 & 4 & \textbf{3} & 60 & \textbf{23} & 4 & \textbf{2} \\
40-powell0 
& \num{9.8e-9}
& \num{1.5e-9}
& 16 & 16 & 396 & \textbf{364} & 35 & 35 \\
40-powell1 
& \num{5.9e-9}
& \num{2.1e-18}
& 15 & \textbf{6} & 441 & \textbf{12} & 28 & \textbf{7} \\
40-rosenbrock0 & 118.20 & 118.20 & 3 & \textbf{2} & 81 & \textbf{40} & 4 & \textbf{2} \\
40-rosenbrock1 & 118.21 & 118.21 & 3 & \textbf{2} & 83 & \textbf{40} & 4 & \textbf{2} \\
\hline
\end{tabular}
\end{table}

\begin{table}[htb]
  \centering
  \caption{Number of major/outer iterations for solving instances of bound-constrained nonlinear MPCCs (INF = Termination at infeasible point, F = Termination with violated
  optimality tolerance).}
  \label{tab:nlpcc_all_iterations}
  \scriptsize  
\begin{tabular}{lrrrrrr}  
\hline
\multicolumn{1}{|l|}{\multirow{2}{*}{Prob-Inst.}}
& \multicolumn{6}{c|}{Number of major/outer iterations}
\\
\cline{2-7} 
\multicolumn{1}{|l|}{} 
& \multicolumn{1}{c|}{plain}
& \multicolumn{1}{c|}{cauchy}
& \multicolumn{1}{c|}{filterSQP}
& \multicolumn{1}{c|}{IPOPT}
& \multicolumn{1}{c|}{MINOS}
& \multicolumn{1}{c|}{SNOPT}\\ \hline
    20-fletcher-0 & 6 & F\;\phantom{0}8 & 49 & F\;3000 & F\;\phantom{0}7 & 45\\
    20-fletcher-1 & 14 & 13 & 13 & F\;3000 & F\;\phantom{0}8 & 62 \\
    20-himmelblau-0 & 7 & 7 & 13 & 107 & 7 & 277 \\
    20-himmelblau-1 & 9 & 8 & 13 & 28 & F\;\phantom{0}5 & 20 \\
    20-mccormick-0 & 3 & 3 & 4 & 75 & 4 & 15 \\
    20-mccormick-1 & 4 & 3 & 4 & 116 & 5 & 11 \\
    20-powell-0 & 17 & 16 & 15 & 146 & 6 & 29 \\
    20-powell-1 & 15 & 6 & 15 & 60 & 13 & 25 \\
    20-rosenbrock-0 & 3 & 2 & 13 & 80 & 7 & 107 \\
    20-rosenbrock-1 & 3 & 2 & 9 & 105 & INF & 86 \\
    \hline
    40-fletcher-0 & F\;\phantom{0}9 & 4 & 12 & F\;3000 & F\;\phantom{0}7 & 176 \\
    40-fletcher-1 & F\;\phantom{}18 & 16 & 13 & 144 & F\;\phantom{0}9 & 110 \\
    40-himmelblau-0 & 9 & 9 & 11 & 102 & F\;\phantom{0}3 & 140 \\
    40-himmelblau-1 & 9 & 8 & 13 & 32 & F\;\phantom{0}3 & 22 \\
    40-mccormick-0 & 3 & 3 & 4 & 175 & 5 & 11 \\
    40-mccormick-1 & 4 & 3 & 4 & 133 & 5 & 10 \\
    40-powell-0 & 16 & 16 & 15 & 181 & 9 & 37 \\
    40-powell-1 & 15 & 6 & 15 & 49 & 6 & 54 \\
    40-rosenbrock-0 & 3 & 2 & 13 & 86 & INF & 102 \\
    40-rosenbrock-1 & 3 & 2 & 9 & 95 & INF & 54 \\
    \hline
  \end{tabular}
\end{table}

\begin{table}[htb]
  \centering
  \caption{Run times (seconds) for Algorithm~
\ref{SLPCC} with Algorithm~\ref{SLPCC_BQP},
and with (\texttt{cauchy})
and without (\text{plain})
Algorithm~\ref{Cauchy}
as well as the NLP solvers
for solving instances of bound-constrained nonlinear MPCCs
(INF = Termination at infeasible point, F = Termination with violated
  optimality tolerance).}
  \label{tab:nlpcc_all_running}
  \scriptsize  
\begin{tabular}{lSSSSSS}  
\hline
\multicolumn{1}{|l|}{\multirow{2}{*}{Prob-Inst.}}
& \multicolumn{6}{c|}{Run times (seconds)}
\\
\cline{2-7} 
\multicolumn{1}{|l|}{} 
& \multicolumn{1}{c|}{plain}
& \multicolumn{1}{c|}{cauchy}
& \multicolumn{1}{c|}{filterSQP}
& \multicolumn{1}{c|}{IPOPT}
& \multicolumn{1}{c|}{MINOS}
& \multicolumn{1}{c|}{SNOPT}\\ \hline
    20-fletcher-0 
    & 0.061 & F\;1.107
    & 0.013 & 0.625 & F\;0.01 & 0.05\\
	20-fletcher-1	
	& 0.093  & 0.556
	& 0.0048 & 0.637 & F\;0.01	& 0.06\\
    20-himmelblau-0
    & 0.045 & 0.251
    & 0.0042 & 0.032 & 0.01 & 0.22\\
    20-himmelblau-1
    & 0.045 & 0.214
    & 0.0038 & 0.010 & F\;0.01 & 0.03 \\
    20-mccormick-0
    & 0.013 & 0.047
    & 0.0015 & 0.029 & 0.00 & 0.01 \\
    20-mccormick-1
    & 0.025 & 0.045
    & 0.0015 & 0.040 & 0.00 & 0.01\\
    20-powell-0
    & 0.154 & 1.849
    & 0.0048 & 0.044 & 0.01 & 0.05 \\
    20-powell-1
    & 0.144 & 0.081
    & 0.0052 & 0.011 & 0.01 & 0.03 \\
    20-rosenbrock-0 
    & 0.028 & 0.171
    & 0.0057 & 0.027 & 0.01 & 0.12 \\
    20-rosenbrock-1 
    & 0.029 & 0.189
    & 0.0050 & 0.031 & INF	& 0.08 \\
    \hline
    40-fletcher-0 
    & F\;0.088 & 0.323
    & 0.013 & 0.801 & F\;0.01 & 0.26 \\
    40-fletcher-1 
    & F\;0.324 & 1.672
    & 0.0091 & 0.057 & F\;0.01	& 0.13 \\
    40-himmelblau-0 
    & 0.104 & 0.609
    & 0.010 & 0.032 & F\;0.01 & 0.16 \\
    40-himmelblau-1 
    & 0.103 & 0.517
    & 0.010 & 0.015 & F\;0.01 & 0.03 \\
    40-mccormick-0
    & 0.033 & 0.107
    & 0.0025 & 0.071 & 0.01	& 0.01 \\
    40-mccormick-1
    & 0.054 & 0.106
    & 0.0023 & 0.062 & 0.01 & 0.01 \\
    40-powell-0 
    & 0.266 & 4.319
    & 0.014 & 0.079 & 0.02 & 0.07 \\
    40-powell-1 
    & 0.264 & 0.206
    & 0.015 & 0.016 & 0.01 & 0.12 \\
    40-rosenbrock-0 
    & 0.052 & 0.396
    & 0.018 & 0.032 & INF & 0.27 \\
    40-rosenbrock-1
    & 0.053 & 0.438
    & 0.015 & 0.038 & INF & 0.18 \\
    \hline
  \end{tabular}
\end{table}

\begin{table}[htb]
  \centering
  \caption{Achieved objective values for Algorithm~
\ref{SLPCC} with Algorithm~\ref{SLPCC_BQP},
and with (\texttt{cauchy})
and without (\text{plain})
Algorithm~\ref{Cauchy}
as well as the NLP solvers
for solving instances of bound-constrained nonlinear MPCCs
(INF = Termination at infeasible point, F = Termination with violated
  optimality tolerance).
\changed{For each problem instance, the lowest achieved objective values are in boldface.}}
  \label{tab:nlpcc_all_obj}
  \scriptsize  
\begin{tabular}{lrrrrrr}  
\hline
\multicolumn{1}{|l|}{\multirow{2}{*}{Prob-Inst.}}
& \multicolumn{6}{c|}{Achieved objective value}
\\
\cline{2-7} 
\multicolumn{1}{|l|}{} 
& \multicolumn{1}{c|}{plain}
& \multicolumn{1}{c|}{cauchy}
& \multicolumn{1}{c|}{filterSQP}
& \multicolumn{1}{c|}{IPOPT}
& \multicolumn{1}{c|}{MINOS}
& \multicolumn{1}{c|}{SNOPT}\\ \hline
    20-fletcher-0 
    & 2246.86 & F\;2246.86
    & 2246.86 & 2246.86 & F\;2246.86 & 2246.86 \\
	20-fletcher-1	
	& 4046.86 & 4046.86
	& 4032.29 & \textbf{3930.27} & F\;3988.57 & 3988.57 \\
    20-himmelblau-0
    & \textbf{826.29} & \textbf{826.29}
    & \textbf{826.29} & 1840.62 & 1840.63 & 1029.16 \\
    20-himmelblau-1
    & \textbf{230.31} & \textbf{230.31}
    & 826.29 & 1422.27 & F~\;954.29 & 1422.27 \\
    20-mccormick-0
    & \textbf{58.93} & \textbf{58.93}
    & \textbf{58.93} & \textbf{58.93} & \textbf{58.93} & 59.92 \\
    20-mccormick-1
    & 58.93 & 58.93
    & 58.93 & 58.93 & 58.93 & 58.93\\
    20-powell-0
    & 9.6$\cdot 10^{-10}$ & 7.5$\cdot 10^{-10}$
    & 3.2$\cdot 10^{-10}$ & 8.0$\cdot 10^{-8}$ & \textbf{6.6}$\mathbf{\cdot 10^{-14}}$ & 1.3$\cdot 10^{-13}$ \\
    20-powell-1
    & 3.0$\cdot 10^{-9}$ & \textbf{1.1}$\mathbf{\cdot 10^{-18}}$
    & 2.2$\cdot 10^{-10}$ & 8.7$\cdot 10^{-8}$ & 4.7$\cdot 10^{-15}$ & 6.5$\cdot 10^{-5}$ \\
    20-rosenbrock-0 
    & 58.60 & 58.60
    & 55.88 & 58.39 & \textbf{40.06} & 44.02 \\
    20-rosenbrock-1 
    & 58.61 & 58.61
    & \textbf{58.34} & 58.61 & INF & 58.39 \\
    \hline
    40-fletcher-0 
    & F\;\textbf{4246.86} & \textbf{4246.86}
    & \textbf{4246.86} & \textbf{4246.86} & F\;\textbf{4246.86} & 4679.15 \\
    40-fletcher-1 
    & F\;8046.86 & 8046.86
    & 8032.29 & \textbf{7784.54} & F\;7988.57 & 7959.42 \\
    40-himmelblau-0 
    & \textbf{1652.58} & \textbf{1652.58}
    & \textbf{1652.58} & 1754.02 & F\;5118.83 & 1754.02 \\
    40-himmelblau-1 
    & \textbf{460.63} & \textbf{460.63}
    & 1652.58 & 2844.54 & F\;4074.33 & 2844.54 \\
    40-mccormick-0
    & 118.93 & 118.93
    & 118.93 & 118.93 & 118.93 & 118.93 \\
    40-mccormick-1
    & 118.93 & 118.93
    & 118.93 & 118.93 & 118.93 & 118.93 \\
    40-powell-0 
    & 9.8$\cdot 10^{-9}$ & 1.5$\cdot 10^{-9}$
    & 6.4$\cdot 10^{-10}$ & 5.9$\cdot 10^{-8}$ & \textbf{8.4}$\mathbf{\cdot 10^{-13}}$ & 4.5$\cdot 10^{-10}$ \\
    40-powell-1 
    & 5.9$\cdot 10^{-9}$ & \textbf{2.1}$\mathbf{\cdot 10^{-18}}$
    & 4.4$\cdot 10^{-10}$ & 1.6$\cdot 10^{-7}$ & 2.7$\cdot 10^{-18}$ & 2.1$\cdot 10^{-8} $\\
    40-rosenbrock-0 
    & 118.20 & 118.20
    & 115.48 & 118.20 & INF & \textbf{113.51} \\
    40-rosenbrock-1
    & 118.21 & 118.21
    & \textbf{117.99} & 118.21 & INF & \textbf{117.99} \\
    \hline
  \end{tabular}
\end{table}

\clearpage

\section{Example -- Augmented Lagrangian Integration}\label{sec:alintegration}

To demonstrate one approach for integrating our
algorithm into an augmented Lagrangian method, we have implemented the
algorithm for bound- and equality-constrained NLPs
from~\cite[Section\,17.4]{nocedal2006no} and replaced the subproblems by 
\eqref{eq:alsubproblem} as described in Section~\ref{S:intro}.

We demonstrate the observed behavior of such an approach on the
\texttt{nash1} problem from \texttt{MacMPEC}, which is stated (including slack variables) below:
\[
\begin{aligned}
\min_{x_0,x_1,x_2}\ &\frac{1}{2}\left((x_{0,1} - x_{0,3})^2 + (x_{0,2} - x_{0,4})^2\right) \\
\text{s.t.}\ 
& x_{1,1} = 15 - x_{0,2} - x_{0,3},\ \ x_{1,2} = 15 - x_{0,1} + x_{0,4}, \\
& x_{2,1} = 34 - 2x_{0,3} - \frac{8}{3}x_{0,4},\ \ x_{2,2} = 24.25 - 1.25 x_{0,3} - 2 x_{0,4}, \\
& 0 \le x_{1} \perp x_{2} \ge 0, \\
& 0 \le x_{0,1} \le 10,\ \ 0 \le x_{0,2} \le 10.
\end{aligned}
\]

The augmented Lagrangian method converges
to the strongly stationary point $(x_0^T, x_1^T, x_2^T) = (5, 9, 5, 9, 1, 19, 0, 0)$
within four iterations, where the constraints are satisfied to an accuracy of
$10^{-10}$ in the $\ell_2$-norm and strong stationarity
(computed as the $\ell_\infty$-norm residual of the gradient of the Lagrangian
projected to bounds and complementarity conditions)
is satisfied to an accuracy of $10^{-10}$ as well. Tightening these criteria
further results in numerical instabilities in later iterations (multipliers
and penalty parameters start diverging while the obtained
point does not move).
To help demonstrate the iterations,
we illustrate the convergence of the method in Figure~\ref{fig:nlag}.
\begin{figure}
  \subfloat[][]
    {\includegraphics[clip, width=0.24\linewidth]{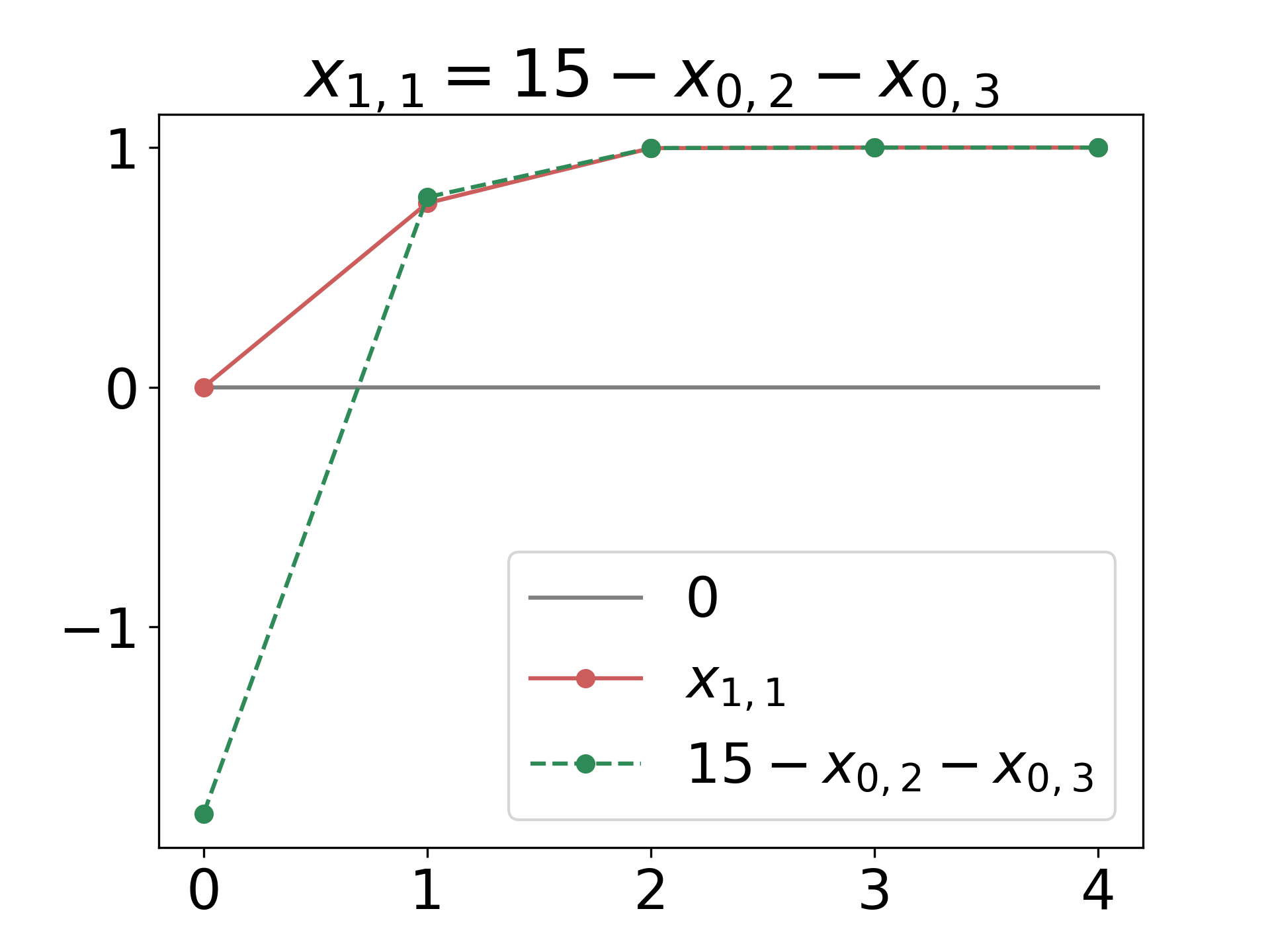}}
    \hfill
  \subfloat[][]
    {\includegraphics[clip, width=0.24\linewidth]{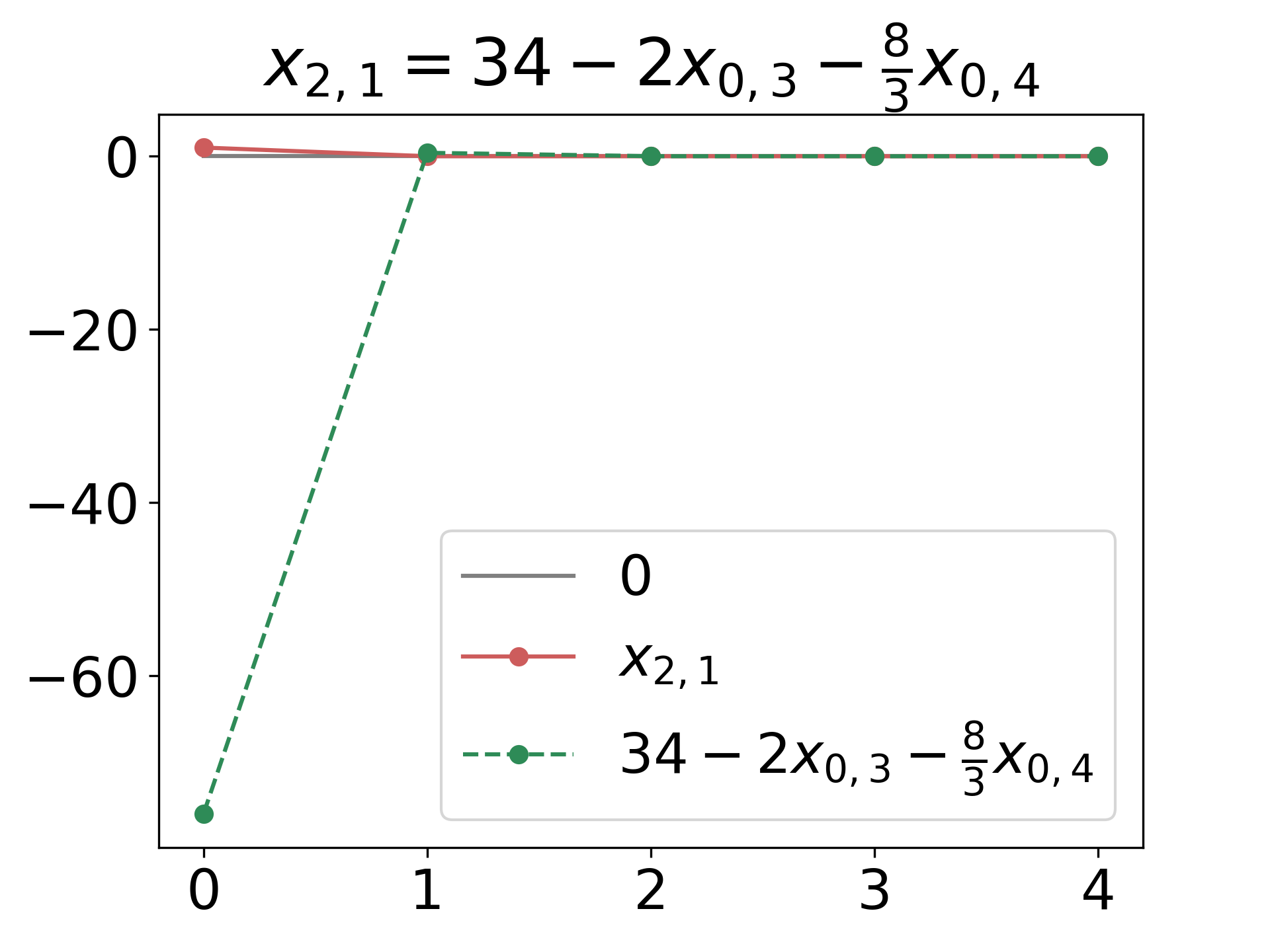}}
    \hfill
\subfloat[][]
    {\includegraphics[clip, width=0.24\linewidth]{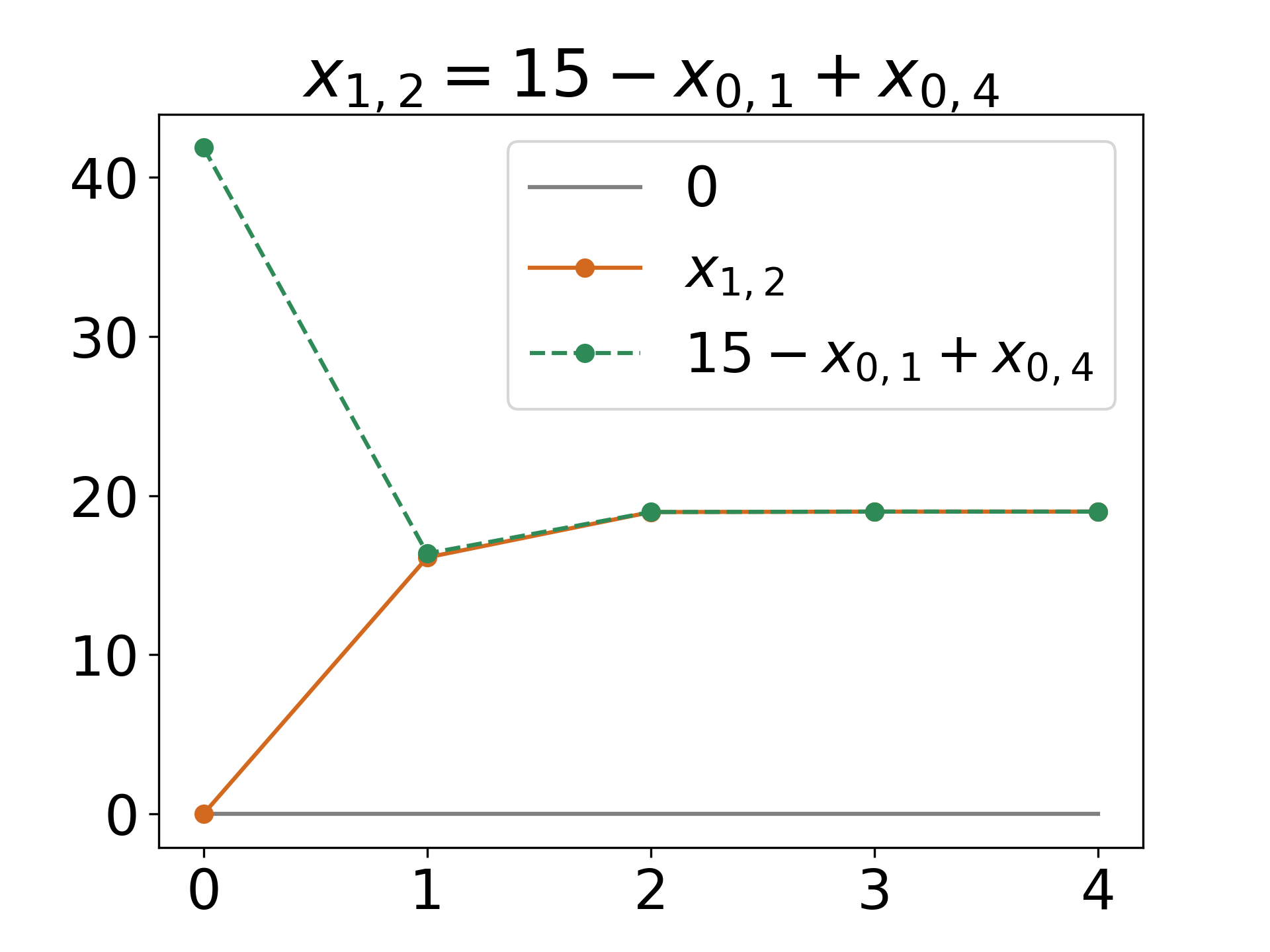}}
    \hfill
\subfloat[][]
    {\includegraphics[clip, width=0.24\linewidth]{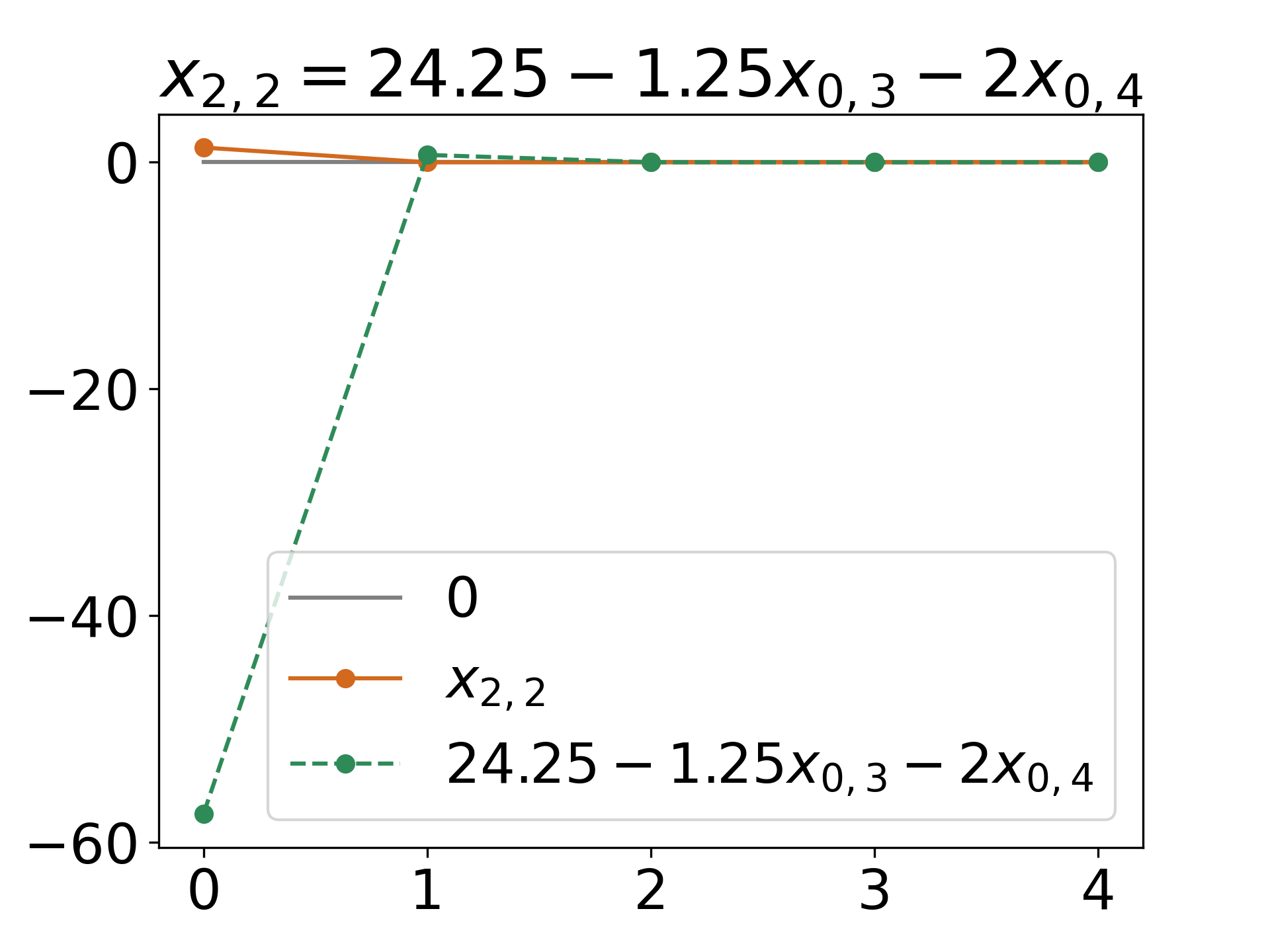}}\\   
  \subfloat[][]
    {\includegraphics[clip, width=0.24\linewidth]{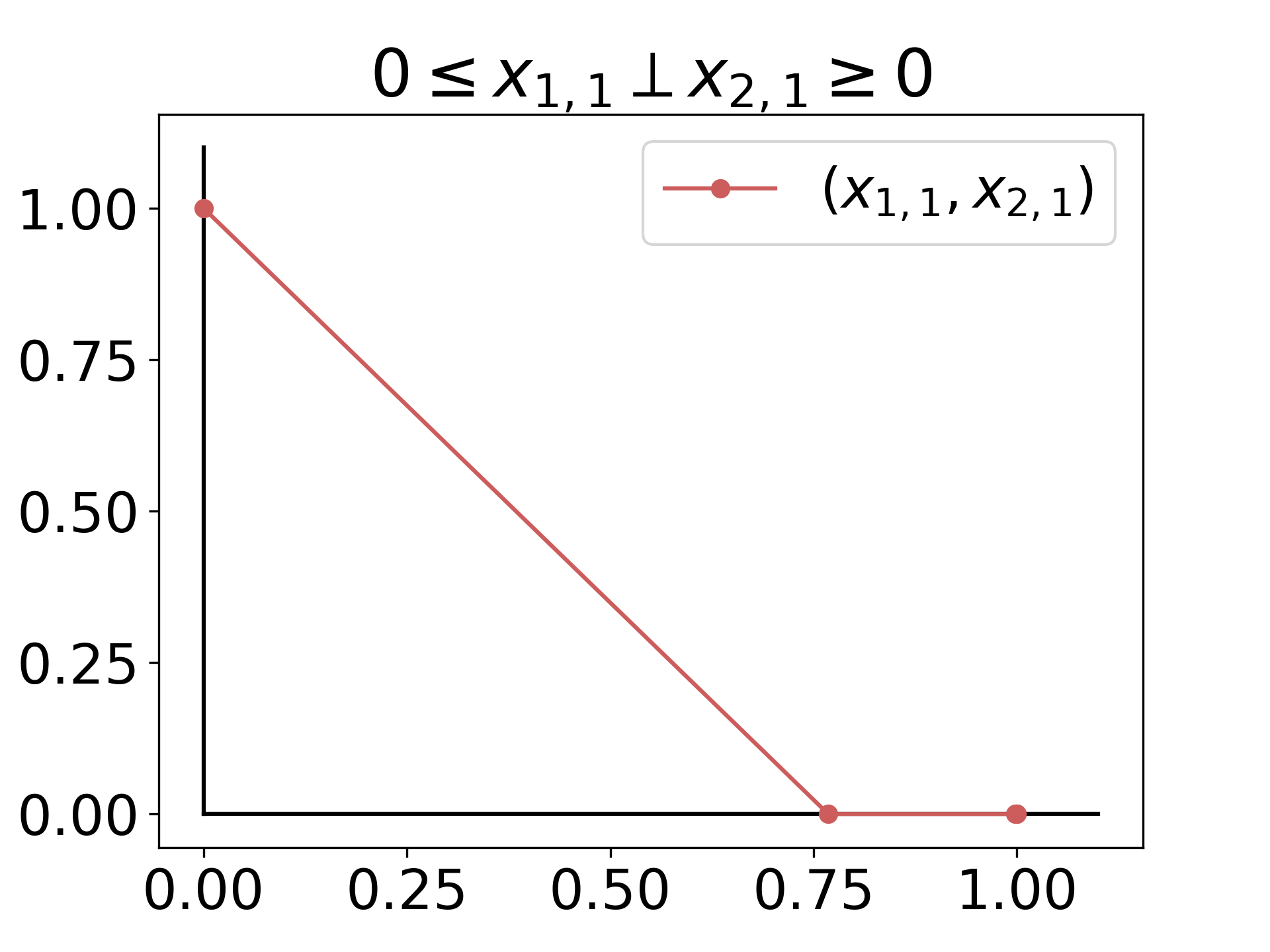}}
    \hfill
  \subfloat[][]
    {\includegraphics[clip, width=0.24\linewidth]{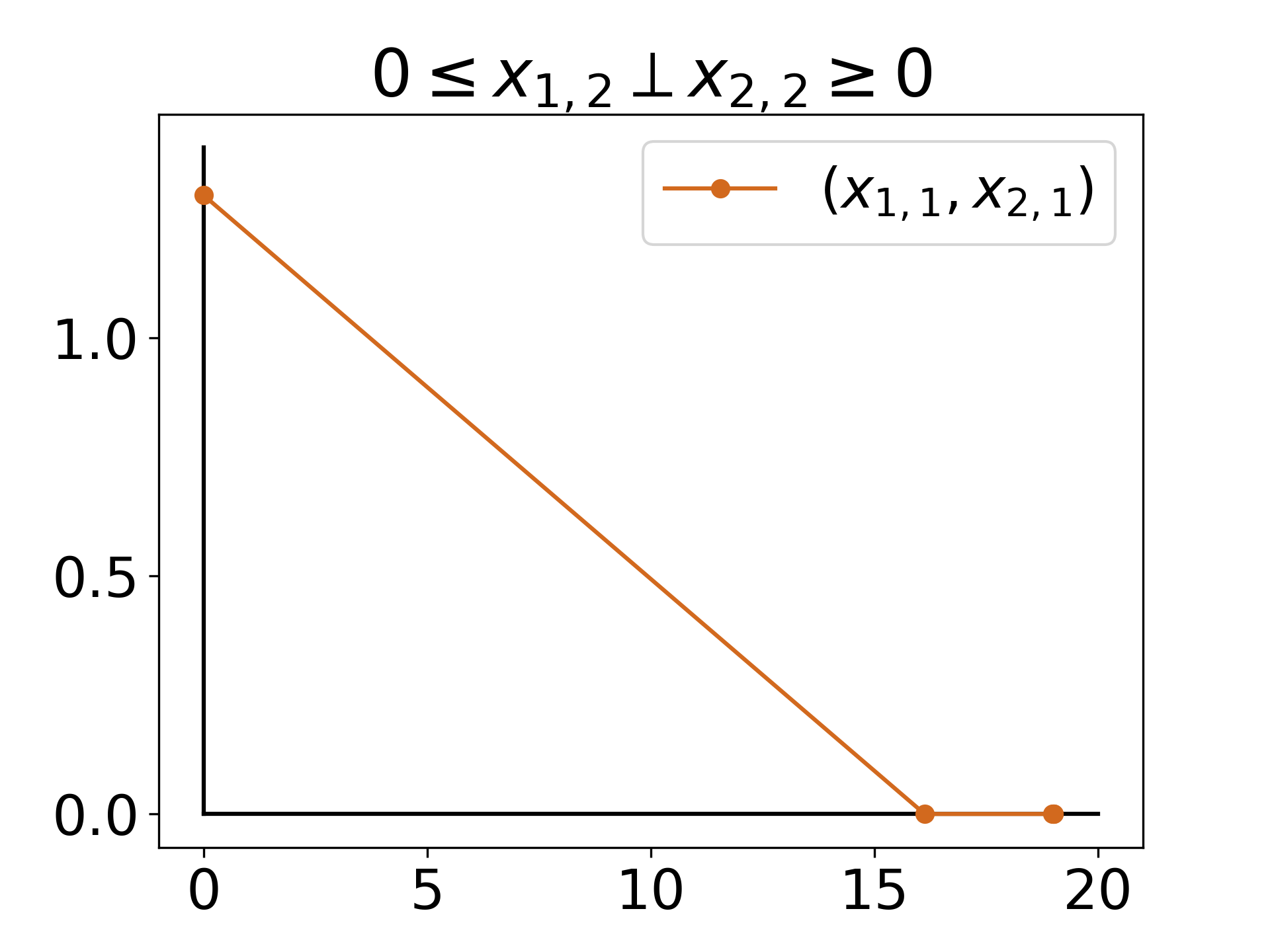}}
    \hfill
\subfloat[][]
    {\includegraphics[clip, width=0.24\linewidth]{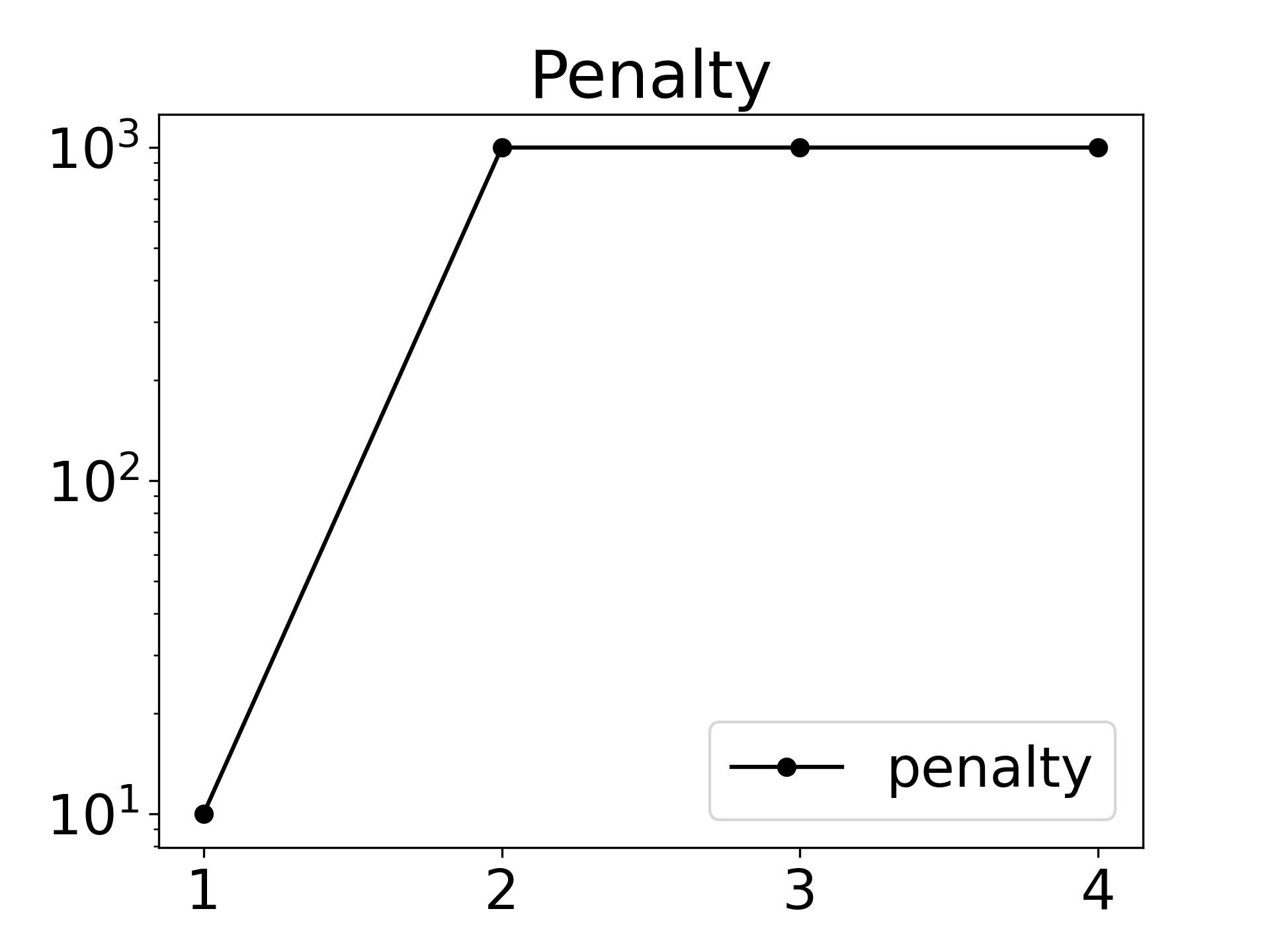}}
    \hfill
\subfloat[][]
    {\includegraphics[clip, width=0.24\linewidth]{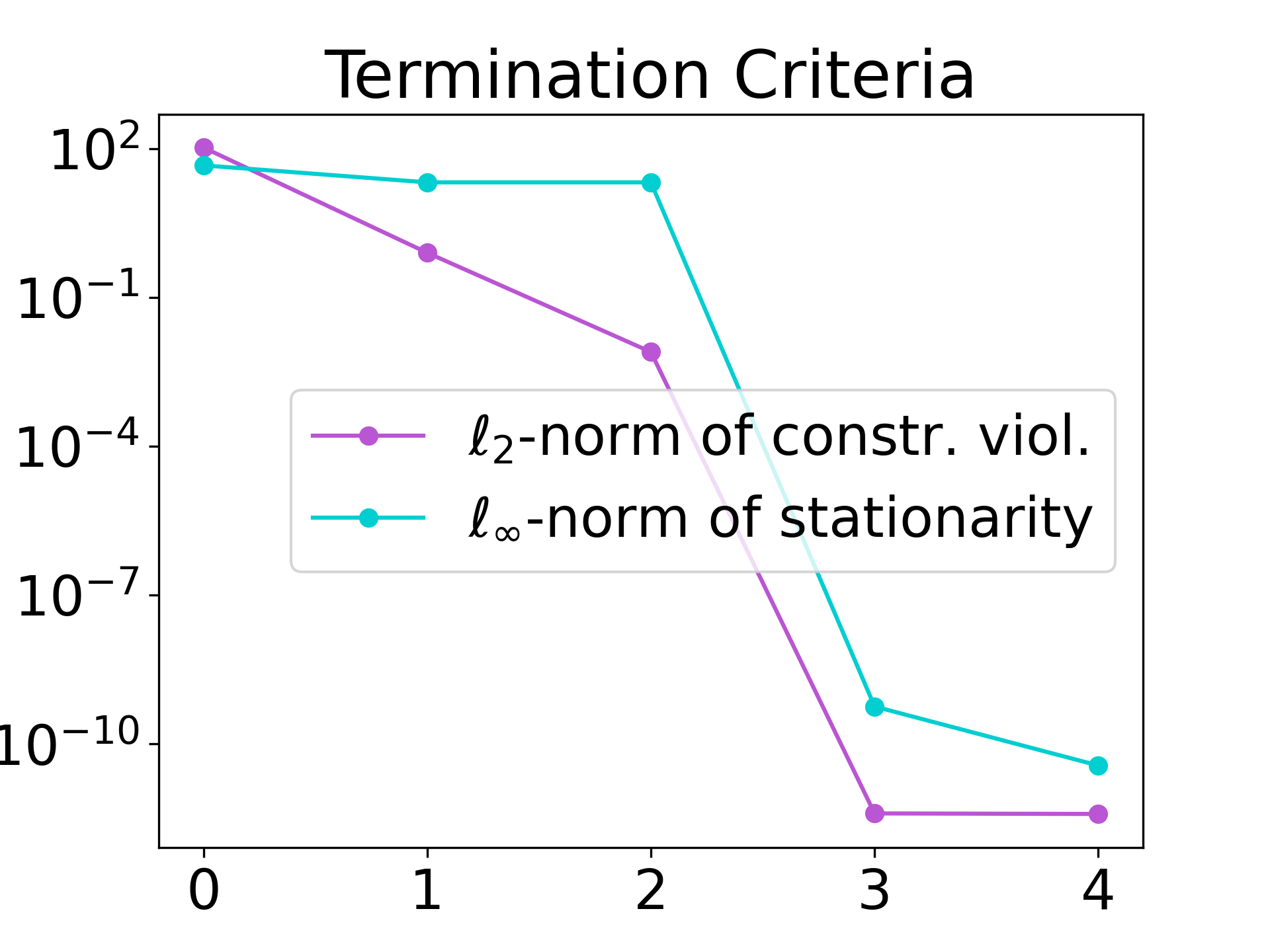}}\\
  \caption{Visualization of the augmented Lagrangian method
  with nonsmooth subproblems,
  convergence of the constraints and slacks for the nonsmooth constraints
  over the iterations ((a)-(d) in top row),
  feasibility of the complementarity conditions ((e)-(f) in bottom row),
  penalty parameter over the iterations ((g) in bottom row) and
  Convergence of optimality and feasibility over the iterations
  ((h) in bottom row).}\label{fig:nlag}
\end{figure}
We compare this approach to an augmented Lagrangian method applied to a
problem where the
complementarity constraint 
$x_{1,i}x_{2,i} \le 0$ is penalized in the objective.
We use again the method from~\cite[Section\,17.4]{nocedal2006no}
(just adapting the multiplier update for the inequality constraints),
and use L-BFGS-B for the subproblem solves.
We provide the same starting point and converge to the same point
although we obtain a slightly worse accuracy of $10^{-7}$ and the algorithm
takes more iterations.
An analogous plot of Figure~\ref{fig:nlag} is given in 
Figure~\ref{fig:nlag_fal}. That our proposed method satisfies
the complementarity constraints throughout all iterations is clearly visible.
\begin{figure}
  \subfloat[][]
    {\includegraphics[clip, width=0.24\linewidth]{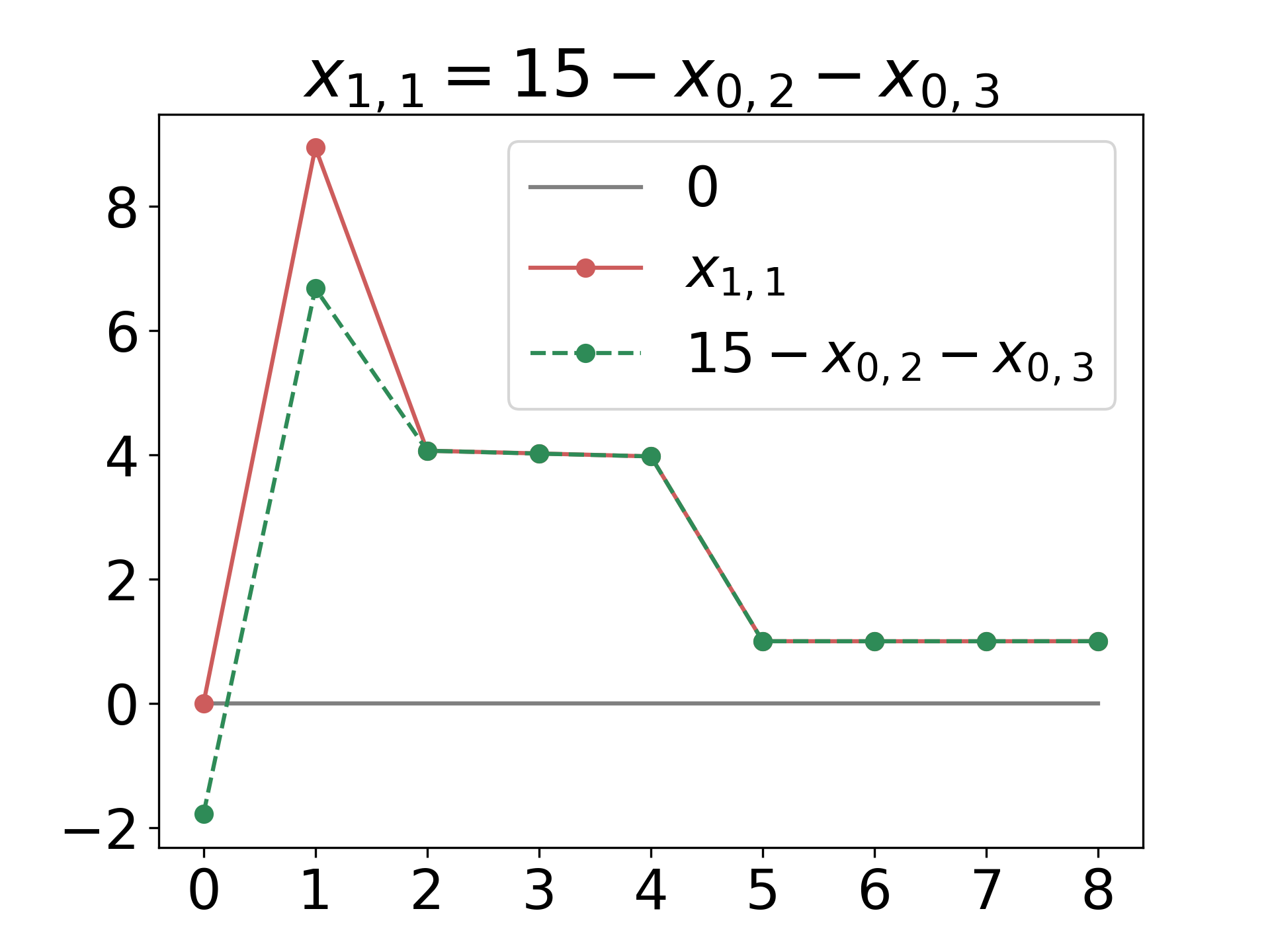}}
    \hfill
  \subfloat[][]
    {\includegraphics[clip, width=0.24\linewidth]{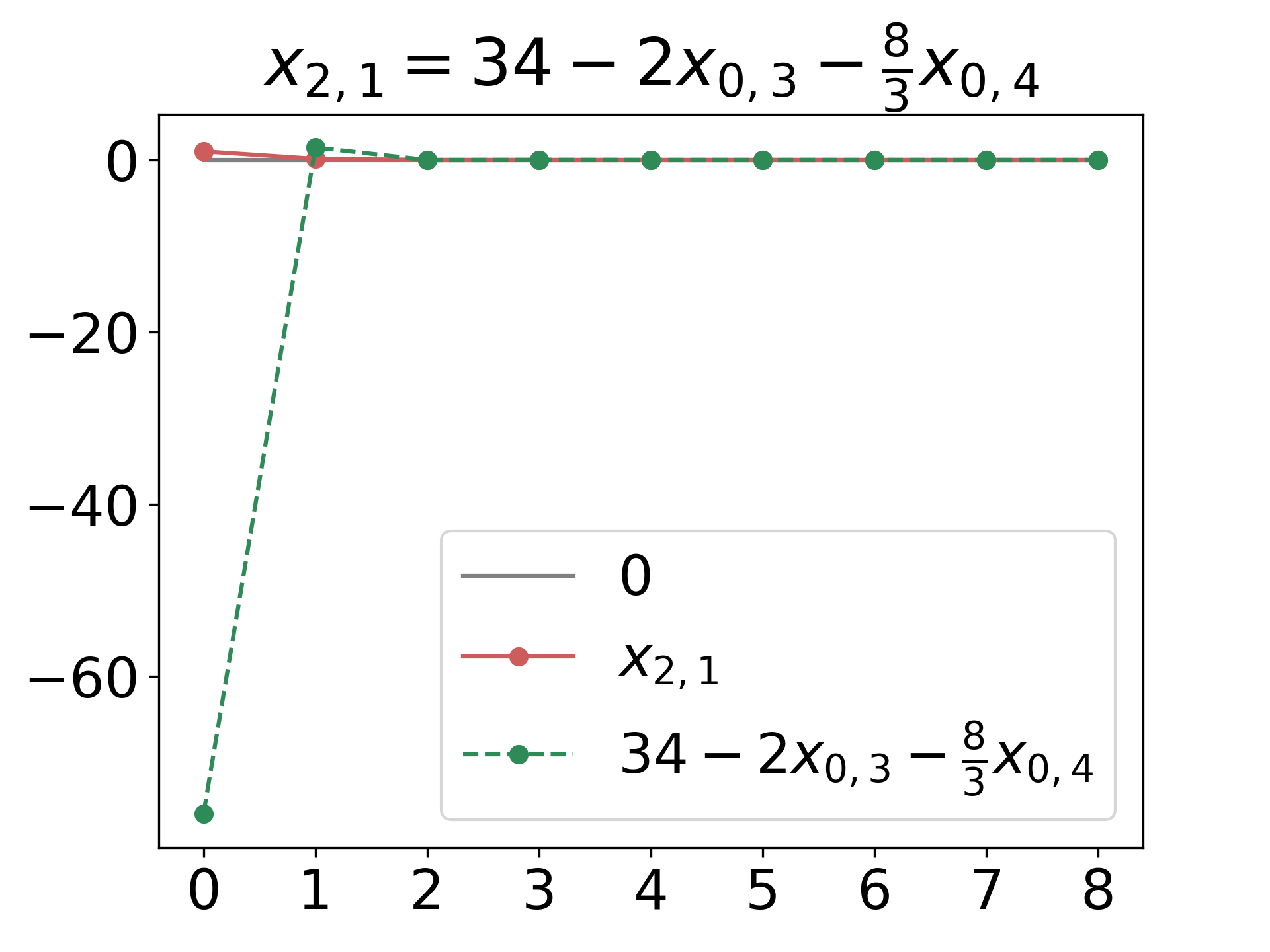}}
    \hfill
\subfloat[][]
    {\includegraphics[clip, width=0.24\linewidth]{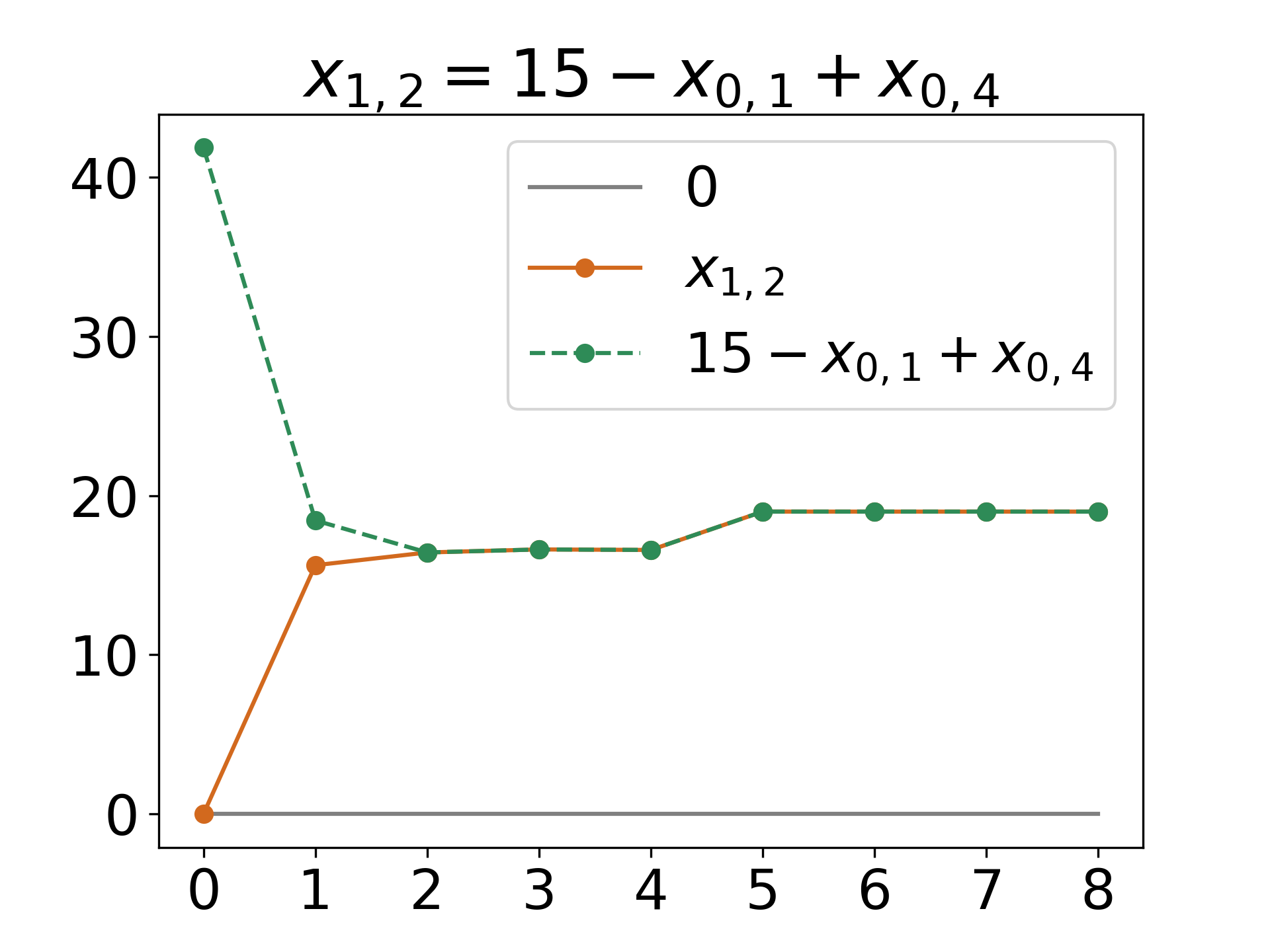}}
    \hfill
\subfloat[][]
    {\includegraphics[clip, width=0.24\linewidth]{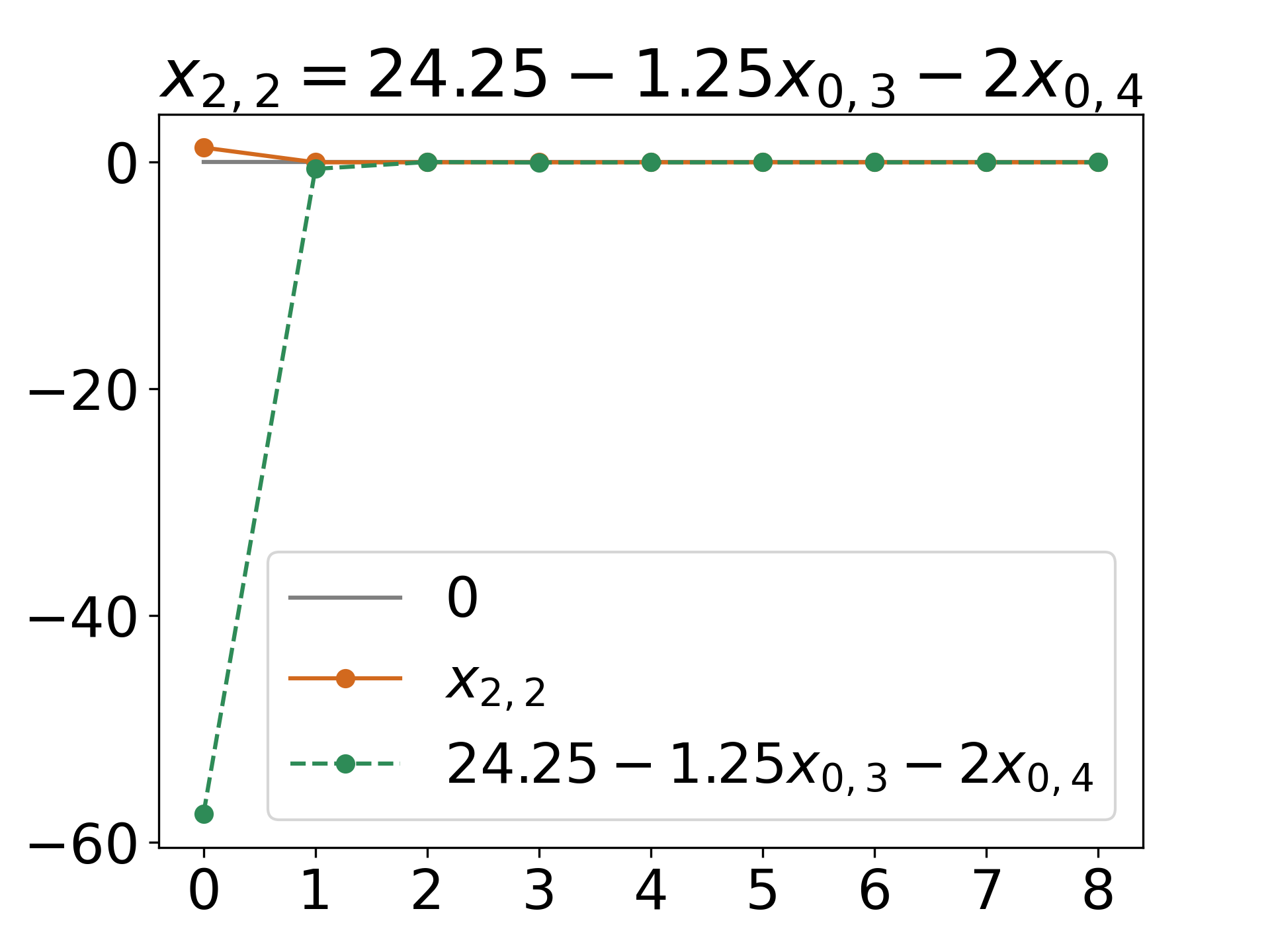}}\\   
  \subfloat[][]
    {\includegraphics[clip, width=0.24\linewidth]{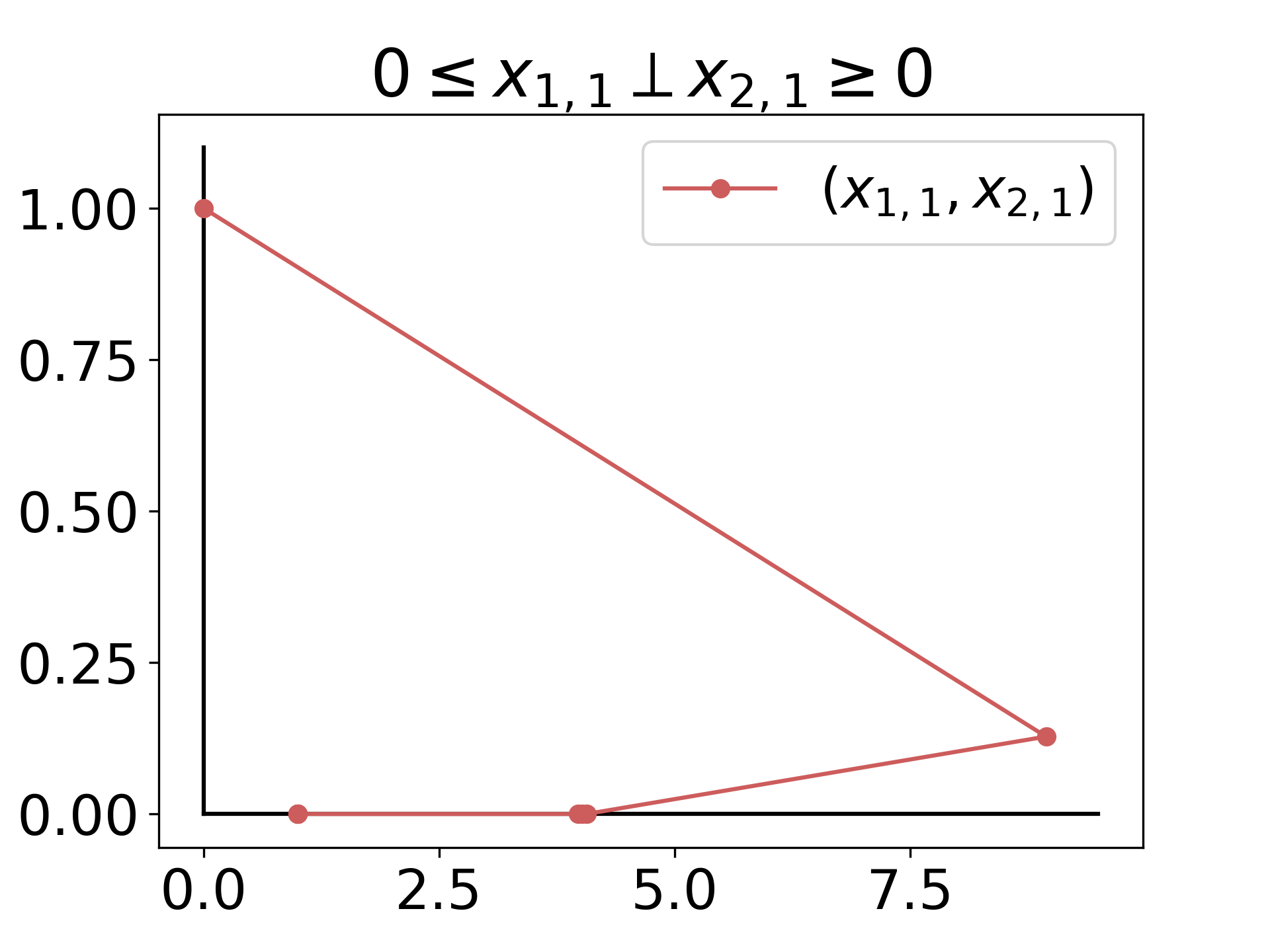}}
    \hfill
  \subfloat[][]
    {\includegraphics[clip, width=0.24\linewidth]{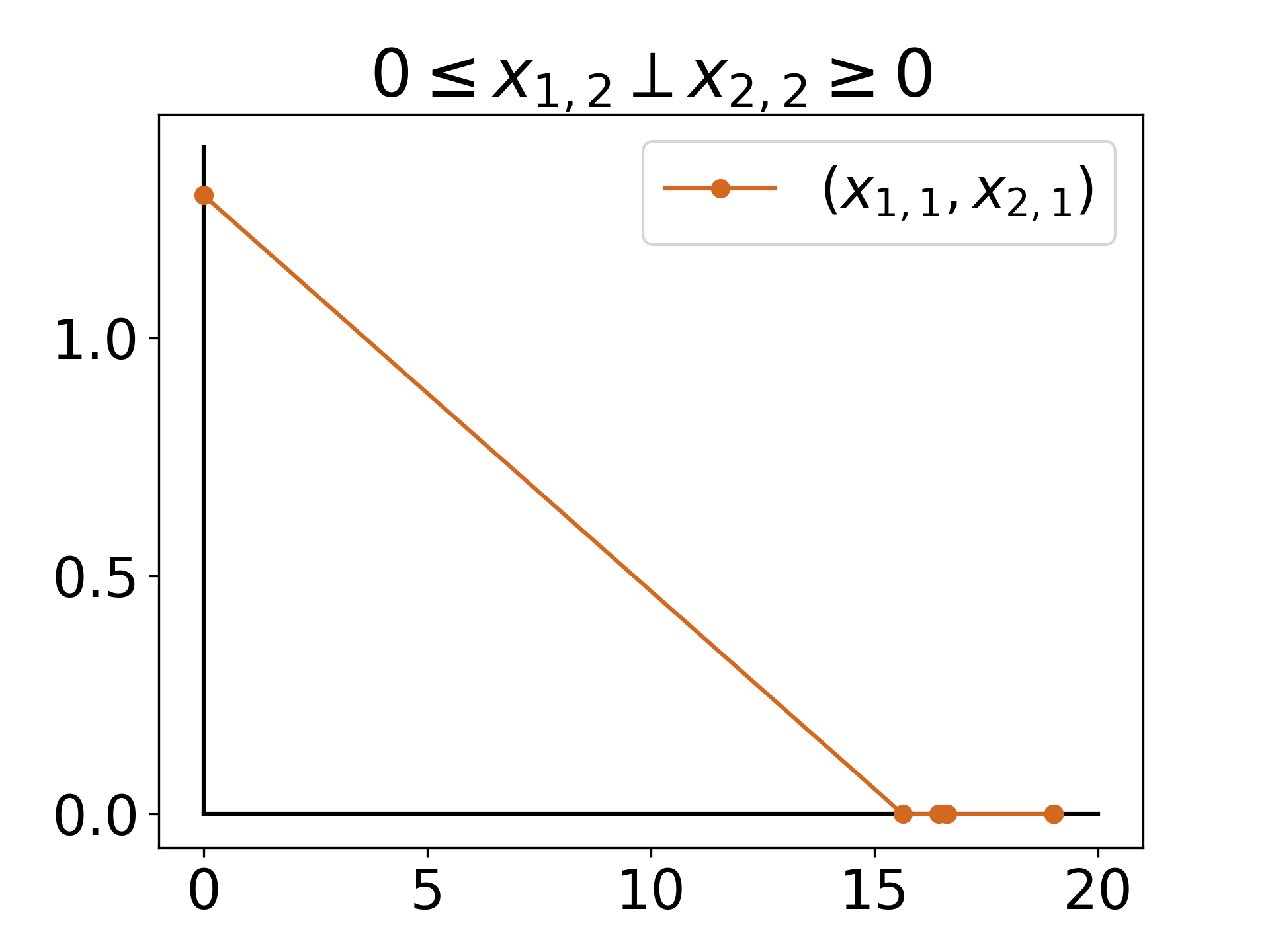}}
    \hfill
\subfloat[][]
    {\includegraphics[clip, width=0.24\linewidth]{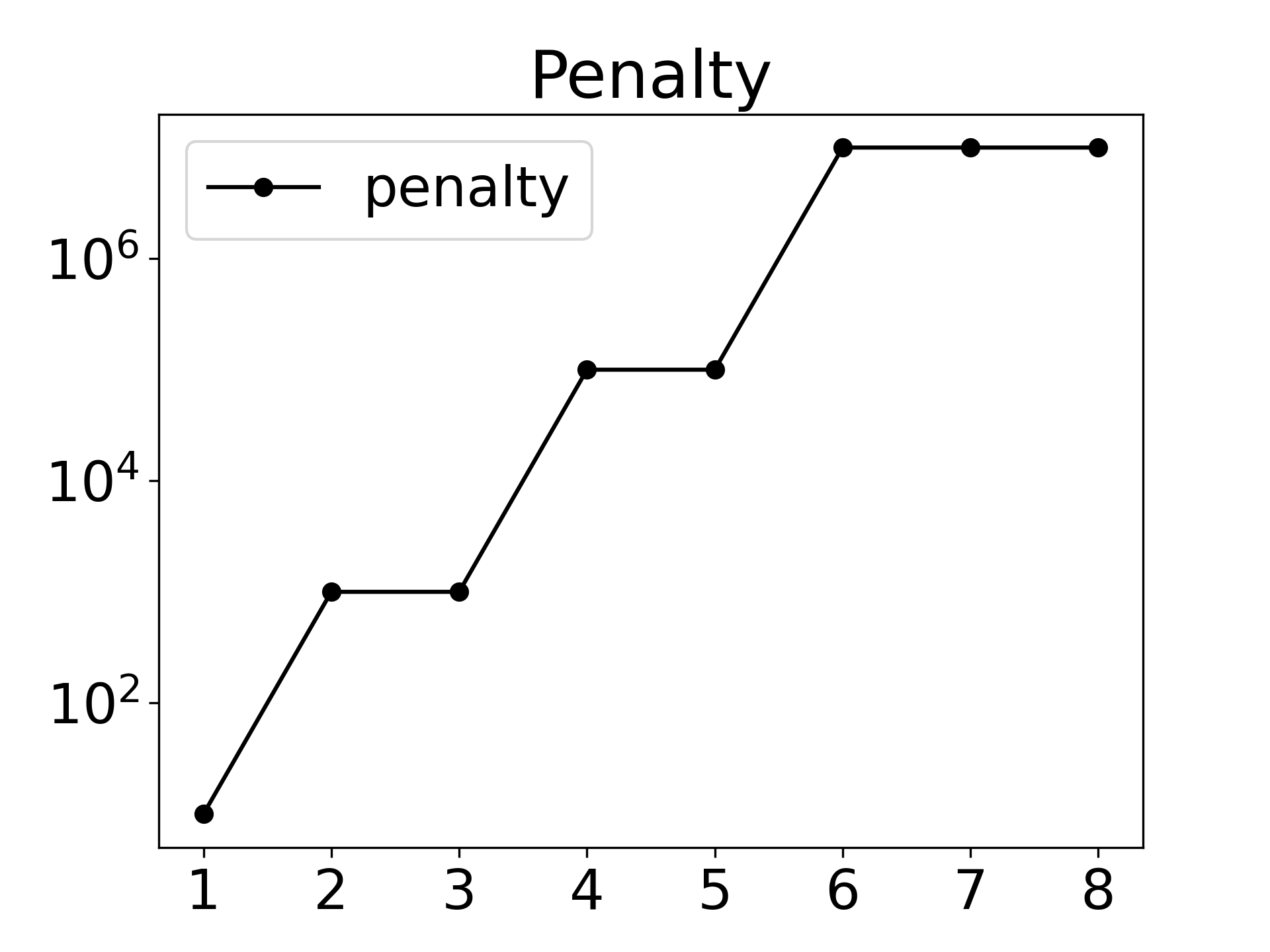}}
    \hfill
\subfloat[][]
    {\includegraphics[clip, width=0.24\linewidth]{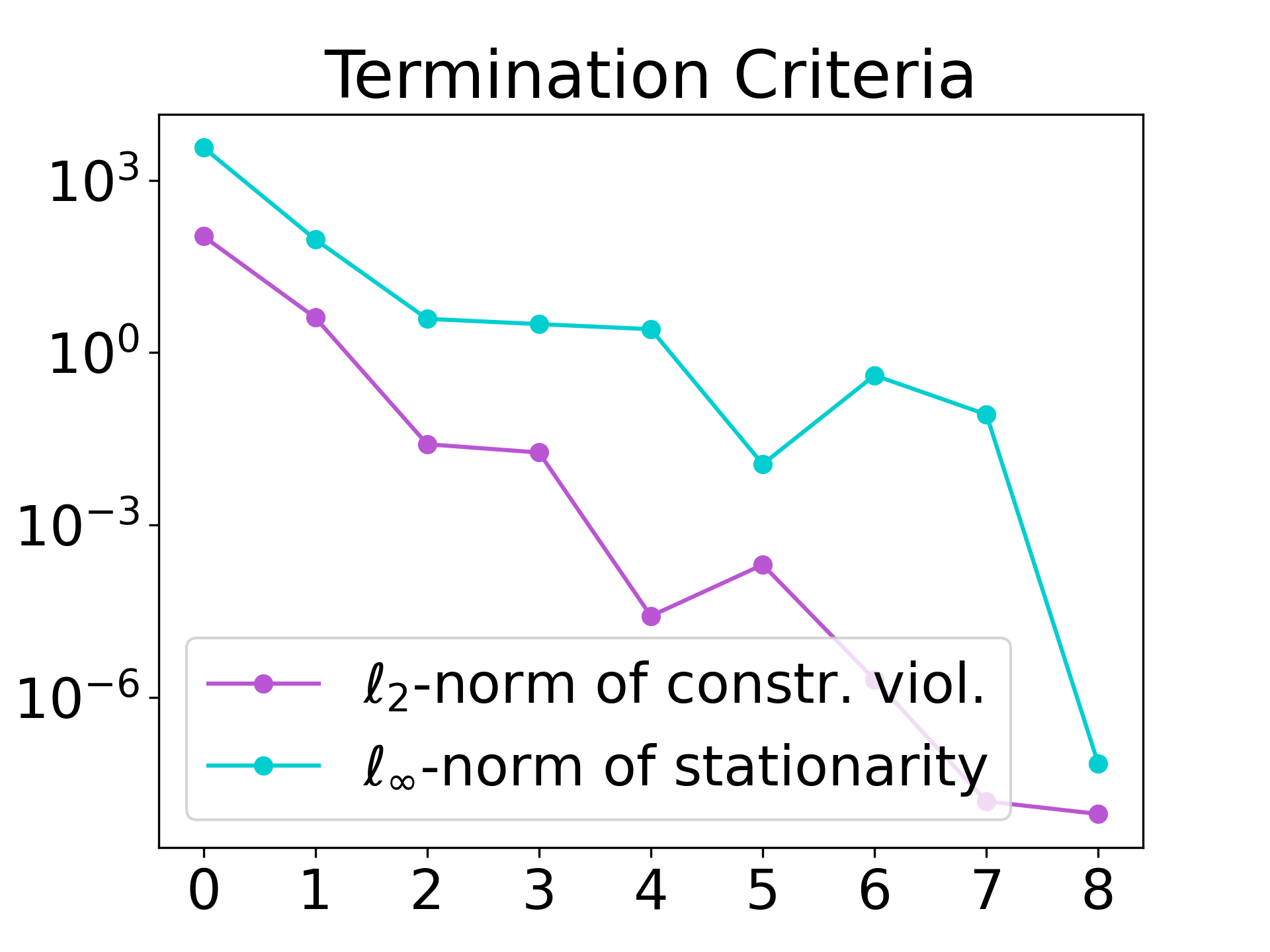}}\\
  \caption{Visualization of the augmented Lagrangian method
  with complementarity constraints treated as general nonlinear constraints,
  convergence of the constraints and slacks for the nonsmooth constraints
  over the iterations ((a)-(d) in top row),
  feasibility of the complementarity conditions ((e)-(f) in bottom row),
  penalty parameter over the iterations ((g) in bottom row) and
  Convergence of optimality and feasibility over the iterations ((h) in bottom row).}\label{fig:nlag_fal}
\end{figure}

\vfill
\framebox{\parbox{.92\linewidth}{The submitted manuscript has been created by
UChicago Argonne, LLC, Operator of Argonne National Laboratory (``Argonne'').
Argonne, a U.S.\ Department of Energy Office of Science laboratory, is operated
under Contract No.\ DE-AC02-06CH11357.  The U.S.\ Government retains for itself,
and others acting on its behalf, a paid-up nonexclusive, irrevocable worldwide
license in said article to reproduce, prepare derivative works, distribute
copies to the public, and perform publicly and display publicly, by or on
behalf of the Government.  The Department of Energy will provide public access
to these results of federally sponsored research in accordance with the DOE
Public Access Plan \url{http://energy.gov/downloads/doe-public-access-plan}.}}
\end{document}